%% file: arxiv-v1.tex
\documentclass[11pt,a4paper,dvipsnames,reqno]{amsart}

\usepackage[margin=1.25in]{geometry}

\usepackage[utf8]{inputenc}
\usepackage{microtype}
\usepackage{amsmath, amssymb}
\usepackage{mathtools}
\usepackage{amsthm, thmtools}

\usepackage{tikz-cd}

\usepackage{mathrsfs} % For \mathscr
\usepackage{stmaryrd} % For \fatslash

\usepackage[mathcal]{euscript}

\DeclareFontFamily{U}{mathx}{}
\DeclareFontShape{U}{mathx}{m}{n}{<-> mathx10}{}
\DeclareSymbolFont{mathx}{U}{mathx}{m}{n}
\DeclareMathAccent{\widehat}{0}{mathx}{"70}
\DeclareMathAccent{\widecheck}{0}{mathx}{"71}

\DeclareFontFamily{U}{wncy}{}
\DeclareFontShape{U}{wncy}{m}{n}{<->wncyr10}{}
\DeclareSymbolFont{mcy}{U}{wncy}{m}{n}
\DeclareMathSymbol{\Sha}{\mathord}{mcy}{"58}

\usepackage[pagebackref=true]{hyperref}
\hypersetup{
    colorlinks=true,
    linktoc=all,
    citecolor=cyan,
    linkcolor=black,
}

\renewcommand*{\backrefalt}[4]{%
\ifcase #1 %
No citations.%
\or
\(\rightarrow\) page #2%
\else
\(\rightarrow\) pages #2%
\fi
}

\usepackage{color}
\usepackage{array}
\usepackage{url}
\usepackage{cleveref}

\declaretheorem[name=Theorem,%
refname={theorem,theorems},%
Refname={Theorem,Theorems},%
within=section]{theorem}

\declaretheorem[name=Proposition,%
refname={proposition,propositions},%
Refname={Proposition,Propositions},%
sibling=theorem]{proposition}

\declaretheorem[name=Lemma,%
refname={lemma,lemmas},%
Refname={Lemma,Lemmas},%
sibling=theorem]{lemma}

\declaretheorem[name=Definition,%
refname={definition,definitions},%
Refname={Definition,Definitions},%
sibling=theorem]{definition}

\declaretheorem[name=Corollary,%
refname={corollary,corollaries},%
Refname={Corollary,Corollaries},%
sibling=theorem]{corollary}

\declaretheorem[name=Remark,%
refname={remark,remarks},%
Refname={Remark,Remarks},%
sibling=theorem,%
style=remark]{remark}

\declaretheorem[name=Expectation,%
refname={expectation,expectations},%
Refname={Expectation,Expectations},%
sibling=theorem,%
style=theorem]{}

\declaretheorem[name=Theorem,%
refname={theorem,theorems},%
Refname={Theorem,Theorems}]{introthm}

\renewcommand{\a}{\mathfrak{a}}

\let\P\relax
\newcommand{\into}{\hookrightarrow}

\newcommand{\Z}{\mathbb{Z}}
\newcommand{\A}{\mathbb{A}}
\newcommand{\R}{\mathbb{R}}
\newcommand{\Q}{\mathbb{Q}}
\newcommand{\C}{\mathbb{C}}

\newcommand{\s}{\mathfrak{s}}
\newcommand{\m}{\mathfrak{m}}
\newcommand{\n}{\mathfrak{n}}

\newcommand{\reg}{reg}

\newcommand{\mat}[4]{\begin{psmallmatrix}%
#1 & #2 \\%
#3 & #4 \\%
\end{psmallmatrix}}

\newcommand{\Mat}[4]{\begin{pmatrix}%
#1 & #2 \\%
#3 & #4 \\%
\end{pmatrix}}

\DeclareMathOperator{\G}{\mathbb{G}}
\DeclareMathOperator{\SL}{SL}
\DeclareMathOperator{\GL}{GL}
\DeclareMathOperator{\PGL}{PGL}

\DeclareMathOperator{\Res}{Res}
\newcommand{\Cc}{C_c^{\infty}}
\DeclareMathOperator{\x}{\times}
\DeclareMathOperator{\supp}{supp}
\DeclareMathOperator{\tr}{tr}

\DeclareMathOperator{\rint}{rint}
\DeclareMathOperator{\F}{\mathcal{F}}
\DeclareMathOperator{\P}{\mathcal{P}}

\DeclareMathOperator{\vol}{vol}
\DeclareMathOperator{\Ad}{Ad}
\DeclareMathOperator{\Nm}{Nm}

\DeclareMathOperator{\Aut}{Aut}
\DeclareMathOperator{\Hom}{Hom}
\DeclareMathOperator{\Lie}{Lie}

\DeclareMathOperator{\Spec}{Spec}
\DeclareMathOperator{\Schw}{\mathcal{S}}

\newcommand{\quash}[1]%

\DeclareMathOperator{\im}{im}

\DeclareMathOperator{\Spl}{spl}

\newcommand{\onto}{\twoheadrightarrow}

\DeclareMathOperator{\Gal}{Gal}

\newcommand{\1}{\mathbf{1}}
\renewcommand{\o}{\mathfrak{o}}
\renewcommand{\O}{\mathcal{O}}
\newcommand{\e}{\mathfrak{e}}

\DeclareMathOperator{\colim}{colim}

\usepackage{enumitem}
\newlist{todolist}{itemize}{2}
\setlist[todolist]{label=$\square$}
\usepackage{pifont}

\usepackage{subfiles}

\title{The Relative Trace Formula for Galois Periods}
\author{Siddharth Mahendraker}

\begin{document}

\begin{abstract}
    Let $E/F$ be a quadratic extension of number fields. We introduce truncated geometric and spectral RTF distributions associated to a Galois symmetric pair $G \subset \Res_{E/F} G_E$, subject to the constraint that $G$ and $\Res_{E/F} G_E$ have the same split rank, and formulate a precise coarse RTF identity. Specializing to $\SL_{2, F} \subset \Res_{E/F} \SL_{2, E}$, we show that the truncated geometric RTF distribution converges, and is given by a linear polynomial in the truncation parameter. We then compute the fine geometric expansion explicitly, including the contribution of the regularized relative unipotent orbital integrals. We propose a geometric viewpoint which guided the computation of these unipotent terms.
\end{abstract}

\maketitle

\setcounter{tocdepth}{3}
\tableofcontents

\newpage

\section*{Introduction}

\subsection{Motivation}

In their seminal 1979 paper \cite{labesse-langlands}, Labesse and Langlands initiated the study of what are now known as global $L$-packets by considering the following simple question: Starting with a cuspidal automorphic representation $\Pi \subset L^2([\SL_{2, F}])$ and an element $g \in \GL_{2, F}(\A)$, we can form a new representation ${}^g\Pi \coloneqq \Pi \circ \Ad(g)$. When is the representation ${}^g\Pi$ automorphic?

This question can be rephrased by noting that if $\Pi$ decomposes as $\Pi = \otimes_v' \Pi_v$, then as one varies $g \in \GL_2(\A)$, the local constituents ${}^{g_v}\Pi_v$ vary through the local $L$-packet associated to $\Pi_v$. It is equivalent then to ask: If we modify $\Pi$ by replacing $\Pi_v$ with another representation in the same local $L$-packet at finitely many places $v$, when is the modified representation automorphic?

The key tool used to resolve this question is the \emph{stable trace formula} for $\SL_{2, F}$. This led Labesse and Langlands to formulate the first \emph{global multiplicity formula}, which gives an expression for the multiplicity in $L^2_{cusp}([\SL_{2, F}])$ of a representation in a global $L$-packet, in terms of the corresponding local $L$-parameters.

Now let $G$ be a connected reductive $F$-group and let $G' \subset G$ be a spherical subgroup. In the relative Langlands program, initiated by \cite{sv, bzsv}, a key question is to understand the portion of the discrete automorphic spectrum of $G$ which is $G'$-distinguished, i.e. the discrete automorphic representations $\Pi$ of $G$ for which there exists a $\varphi \in \Pi$ such that the period integral $$\mathscr{P}_{G'}(\varphi) \coloneqq \int_{[G']} \varphi(g)\,dg$$ is non-vanishing. The notion of $G'$-distinction also has a local counterpart; we say $\Pi_v$ is $G'$-distinguished if $\Hom_{G'(F_v)}(\Pi_v, \1)$ is non-zero. In this new setting, one can ask an analogue of the question of Labesse and Langlands. Namely, starting with a $G'$-distinguished cuspidal automorphic representation $\Pi \subset L^2_{disc}([G]^1)$, we obtain local representations $\Pi_v$ which are again $G'$-distinguished. If we modify $\Pi$ by replacing $\Pi_v$ with another $G'$-distinguished representation in the same local $L$-packet at finitely many places $v$, when is the modified representation again automorphic and $G'$-distinguished? Can one formulate a version of the global multiplicity formula in this setting?

This question suggests a plethora of new problems, both local and global in nature. Locally, as above, one has to contend with the existence of multiple distinguished representations appearing in a single $L$-packet. The new complication: individual local representations $\Pi_v$ can support multiple distinguishing functionals, that is, $$\dim \Hom_{G'(F_v)}(\Pi_v, \1) > 1.$$ This is the problem of parameterizing distinguished local $L$-packets. Globally, one needs to develop a regularized relative trace formula identity associated to $G' \subset G$, and carry out its stabilization.

Fix $E/F$ a quadratic extension of number fields. The simplest setting in which all of these complications arise is when $G = \Res_{E/F} G'_E$ and $G' \subset G$ is embedded as the fixed locus of the natural Galois involution. Such a pair $G' \subset G$ is known as a Galois symmetric pair, and periods associated to such pairs are called Galois periods. Among Galois symmetric pairs, the simplest interesting example is $\SL_{2, F} \subset \Res_{E/F} \SL_{2, E}$. This case can be understood as a twisted version of the original problem studied by Labesse and Langlands; indeed when $E = F \x F$ is the split quadratic étale $F$-algebra, the two questions coincide.

In ongoing joint work with Leslie, the results of this article will be applied to stabilize the RTF corresponding to the Galois period for $\SL_{2, F}$.

\subsection{What is Proved?}

In \cref{part-one}, we give a precise formulation of a coarse, regularized relative trace formula (RTF) identity for Galois symmetric pairs $G' \subset G$ of equal split rank. The key new contributions are appropriate definitions of truncated geometric and spectral kernels. Here $X = G/G'$ is the corresponding Galois symmetric space. The truncated geometric and spectral kernels yield truncated geometric and spectral RTF distributions, and the coarse RTF identity is formulated as an equality of these distributions for appropriately matching test functions.

In \cref{part-two}, we specialize these definitions to the case where $G' = \SL_{2, F}$ and $G = \Res_{E/F} \SL_{2, E}$, and prove that the truncated geometric RTF distribution converges. This allows us to define the \emph{regularized} geometric RTF distribution associated to $\SL_{2, F} \subset \Res_{E/F} \SL_{2, E}$.

Finally, in \cref{part-three}, we give an \emph{explicit formula} for the regularized geometric RTF distribution, suitable for stabilization by hand. This is the fine geometric expansion of the RTF for $\SL_{2, F} \subset \Res_{E/F} \SL_{2, E}$.

In the remainder of this introduction, we discuss each of these parts in more detail and state our main theorems.

\subsection{A Coarse RTF for Galois Symmetric Pairs}

Set $X \coloneqq G/G'$. The relative trace formula identity we formulate was designed to meet the following natural criteria: 
\begin{enumerate}
    \item The contribution from cuspidal automorphic representations to the spectral RTF should contain squares of $G'$-periods.
    \item The geometric RTF distribution should take as its starting point a test function $f \in \Cc(X(\A))$. This is because we eventually hope to (pre-) stabilize the geometric RTF, and in the relative setting, transfer is most naturally expressed as a matching between test functions on $X$ and $X_\e$, where $X_\e$ is an \emph{endoscopic Galois symmetric space} \cite{leslie}.
    \item The continuous contribution to the spectral RTF distribution should be amenable to an analysis in terms of regularized periods of Eisenstein series obtained via \emph{mixed truncation} \cite{zydor}.
\end{enumerate}

To this end, in \cref{assigning-test-fns}, we explain how, beginning with a test function $f \in \Cc(X(\A))$, one can associate a family of test functions $$\{ \; \Phi^\xi_1 \in \Cc(G(\A)^1) \; \}_{\xi \in \ker^1(F, X)}$$ where $\xi$ runs through the Galois cohomology set $\ker^1(F, X)$ indexing $G(F)$-orbits in $X(F)$.

Next in \cref{truncated-kernels}, we define truncated geometric and spectral kernels $$k^T_{f} : [G']^1 \to \C \text{ and } \{ \; k^T_{\Phi^{\xi}_1, \xi} : [G']^1 \x [G_\xi]^1 \to \C \; \}_{\xi \in \ker^1(F, X)}.$$ Here $T$ is a truncation parameter, and $G_\xi$ is a group obtained by twisting $G'$ by a distinguished choice of 1-cocycle representing $\xi$. The corresponding truncated geometric and spectral RTF distributions are given by
$$J^T_{geom}(f) \coloneqq \int_{[G']^1} k^T_f(h)\,dh \;\; \text{ and } J^T_{spec,\,\xi}(\Phi^{\xi}_1) \coloneqq \int_{[G']^1} \int_{[G_\xi]^1} k^T_{\Phi^{\xi}_1, \xi}(h_1, h_2)\,dh_2\,dh_1.$$

In fact, the truncated kernels admit natural refinements: $$k_f^T = \sum_{\o \in \mathcal{O}^X}k_{f, \o}^T \;\; \text{ and } \;\; k_{\Phi^{\xi}_1, \xi}^T = \sum_{\chi \in \mathcal{X}^G} k_{\Phi^{\xi}_1, \xi, \chi}^T.$$ On the geometric side, the refined kernel is indexed by the set of geometric data $\{ \o \in \O^X \}$ which consist of rational points in the fibers of the geometric invariant theory map $\chi : X \to X/\!/G'$ above rational points in $X/\!/G'$. On the spectral side, the refined kernel is indexed by the collection of cuspidal (or spectral) data $\{ \chi \in \mathcal{X}^G \}$ which label the components of the coarse Langlands decomposition of $L^2([G]^1)$.

We write $$J^T_{\o}(f) \coloneqq \int_{[G']^1} k^T_{f, \o}(x)\,dx \;\; \text{ and } \;\; J^T_{\xi, \chi}(\Phi) \coloneqq \int_{[G']^1}\int_{[G_\xi]^1} k^T_{\Phi, \xi, \chi}(x, y)\,dy\,dx$$ for the corresponding refined truncated geometric and spectral RTF distributions.

%(defined in \cref{coarse-rtf})

\begin{introthm}[Coarse RTF]\label{thmA}
    With the notation as above, let $f \in \Cc(X(\A))$ be a test function, and let $\{ \Phi^{\xi}_1 \}_{\xi \in \ker^1(F, X)}$ be the associated collection of test functions in $\Cc(G(\A)^1)$. Let $T \in \a_{0'}$ be a truncation parameter.
    
    Assume that for $T \in \a_{0'}$ sufficiently regular, $$\sum_{\o \in \O^X} \int_{[G']^1} |k^T_{f, \o}(x)|\,dx < \infty \text{ and } \sum_{\xi \in \ker^1(F, X)} \sum_{\chi \in \mathcal{X}^G} \int_{[G']^1} \int_{[G_\xi]^1} |k^T_{\Phi^{\xi}_1, \xi, \chi}(x, y)|\,dy\,dx < \infty.$$

    Then the associated geometric and spectral RTF distributions agree: $$J^T_{geom}(f) = \sum_{\xi \in \ker^1(F, X)} J^T_{spec,\,\xi}(\Phi^{\xi}_1).$$
\end{introthm}

Let us suppose that the hypotheses of the previous theorem are satisfied, and take an informal look at the two sides of the coarse RTF.

First consider the contribution to the spectral side corresponding to a cuspidal automorphic representation $\Pi$ of $[G]^1$, i.e. the contribution from a cuspidal datum $\chi$ corresponding to $\Pi$. The truncated spectral kernel $k^T_{\Phi^{\xi}_1, \xi, \chi}$ is built from parabolic kernels $k_{\Phi_1^{\xi}, \xi, P, \chi}$ and the cuspidality of $\Pi$ means these vanish identically for all non-trivial $P$. Thus the $\chi$-contribution to the spectral side of the RTF takes the form $$\sum_{\xi \in \ker^1(F, X)} \sum_{\varphi} \mathscr{P}_{G'}(R(\Phi^{\xi}_1)\varphi)\overline{\mathscr{P}_{G_\xi}(\varphi)},$$ where $\varphi$ runs through an orthonormal basis for the $\Pi$-isotypic component of $L^2([G]^1)$. Here $R(\Phi^{\xi}_1)$ denotes the right regular action of $\Phi^{\xi}_1 \in \Cc(G(\A)^1)$ on $L^2([G]^1)$.

On the geometric side, consider the contribution of $\o$, a regular relatively elliptic geometric datum with basepoint $\eta_\o \in \o$. Then the truncated geometric kernel $k^T_{f, \o}$ is built from parabolic kernels $k_{f, P, \o}$, and an immediate consequence of the ellipticity of $\o$ is that the contributions of the various parabolic terms $k_{f, P, \o}$ are zero for all non-trivial $P$. Thus, the $\o$-contribution to the geometric side of the RTF is given by $$\sum_{\tau} \vol([G'_{\eta_{\o}^{\tau}}])\int_{G'_{\eta_\o^{\tau}}(\A) \backslash G'(\A)} f(x^{-1}\eta_\o^{\tau}x)\,dx,$$ where $\tau$ runs through $\ker[H^1(F, G'_{\eta_\o}) \to H^1(F, G')]$ and $G_{\eta_{\o}^{\tau}}'$ is the stabilizer of $\eta_\o^{\tau}$ under the conjugation action of $G'$ on $X$.

Taken together, we see informally that the coarse RTF for the Galois symmetric pair $(G, G')$ is an equality of the form \begin{multline*}
    \sum_{\Pi} \sum_{\xi \in \ker^1(F, X)} \sum_{\varphi} \mathscr{P}_{G'}(R(\Phi^{\xi}_1)\varphi)\overline{\mathscr{P}_{G_\xi}(\varphi)} \; + \cdots \\= \sum_{\o \in \O^{X}_{reg. ell.}}\sum_{\tau} \vol([G'_{\eta_{\o}^{\tau}}])\int_{G'_{\eta_\o^{\tau}}(\A) \backslash G'(\A)} f(x^{-1}\eta_\o^{\tau}x)\,dx \; + \cdots
\end{multline*}

 Thus, we see that the equality of distributions in \cref{thmA} deserves its name.

In general, it is expected that the geometric and spectral RTF distributions are polynomial-exponentials in the truncation parameter $T$, when $T$ is sufficiently regular, and their constant terms yield \emph{regularized} coarse RTF distributions.

\subsection{Convergence of the Truncated Geometric RTF Distribution for $\SL_2$}

We now shift our focus to the Galois symmetric pair $\SL_{2, F} \subset \Res_{E/F} \SL_{2, E}$, where we endeavour to understand the entire regularized RTF distribution as explicitly as possible. In this article, we address only the regularized geometric RTF distribution. A forthcoming article will address the regularized spectral RTF distribution. For the remainder of this introduction, set $G = \Res_{E/F} \SL_{2, E}$, and let $G' = \SL_{2, F}$, embedded as the fixed locus of the Galois involution.

\begin{introthm}\label{thmB}
    With notation as above, let $f \in \Cc(X(\A))$. Then $$\sum_{\o \in \O^X} \int_{[G']} |k^T_{f, \o}(x)|\,dx < \infty$$ when the truncation parameter $T \in \a_{0'} \simeq \R$ is sufficiently large relative to $f$.
    
    Thus the truncated geometric RTF distribution $$J^T_{geom}(f) = \sum_{\o \in \O^X} J^T_{\o}(f)$$ converges for $T$ sufficiently large relative to $f$.
    
    Moreover, each summand $J^T_{\o}(f)$ is a linear polynomial in the truncation parameter, and thus $J^T_{geom}(f)$ is as well.
\end{introthm}

The proof of convergence makes systematic use of the relative truncation machinery introduced by Zydor \cite{zydor}.

%Informally, the key idea is that when the truncation parameter $T$ is taken to be sufficiently large relative to the test function $f$, one can show that the terms in the regularized geometric RTF breakup into two simpler summands. The first summand can be bounded for entirely formal reasons. The second summand can be understood geometrically: when $T$ is sufficiently large, the geometry associated to this summand is no longer that of $X$, but instead the geometry of a ``line bundle over $X_M$'', the Levi Galois symmetric space associated to the standard Levi $M \coloneqq \Res_{E/F} \G_{m,E} \subset G$. Ultimately, the bound on this second summand is obtained from a (fiber-wise) application of Poisson summation. 

\begin{remark} The proof that $T \mapsto J^T_{\o}(f)$ is linear as a function of $T$ yields a more refined result. If we write the corresponding distribution as $$J^T_{\o}(f) = J_{\o}(f) + J'_{\o}(f)T,$$ for $T$ sufficiently large, we are able to identify $J'_{\o}(f)$ in terms of the RTF associated to the Levi Galois symmetric space $X_M$. See \cref{geom-dist-is-linear-in-T} and \cref{interpreting-the-non-const-term}. The RTF for the Levi Galois symmetric space $X_M$ and more generally for all Galois symmetric pairs arising from tori of rank one is worked out explicitly in \cref{rtf-for-rk-one-tori}.
\end{remark}

With \cref{thmB} in hand, we define $$J_{geom}(f) \coloneqq \text{ constant term of $T \mapsto J^T_{geom}(f)$, for $T$ sufficiently large}.$$ This is the regularized geometric RTF distribution. It decomposes as a sum over geometric data $$J_{geom}(f) = \sum_{\o \in \O^X} J_\o(f).$$

\subsection{The Fine Geometric Expansion}

Finally, for each class of geometric data $\o \in \O^X$, we give concrete formulae for $J_{\o}(f)$ suitable for future use in the stabilization of the relative trace formula.

\begin{introthm}[Fine Geometric Expansion for $\SL_2$]\label{thmC}
    With the notation as above, fix $\o \in \O^X$ a geometric datum with basepoint $\eta \in \o$. Write $(G_\eta, G'_\eta)$ for the descendant of $(G, G')$ at $\eta$ (see \cref{descendants}).
    \begin{enumerate}
        \item If $\o$ is relatively elliptic, $$J_{\o}(f) =  \sum_{ \tau} \vol([G'_{\eta_\tau}]) \int_{G'_{\eta_\tau}(\A) \backslash G'(\A)} f(x^{-1} \eta_{\tau} x)\,dx,$$ where $\tau$ runs through $H^1(F, G_\eta')$ and $\eta_\tau$ is a representative for the corresponding rational $G'$-orbit. Here $\eta_1 = \eta$. 
        \item If $\o$ is regular semisimple but not elliptic, then $$J_\o(f) = -2\vol([G'_\eta]^1) \int_{G'_\eta(\A)^1 \backslash G'(\A)} f(x^{-1}\eta x)v(x)\,dx,$$ where $v(x) \coloneqq H_{0'}(x) + H_{0'}(wx)$ is a weight function, and $w$ is a representative for the non-trivial element of the Weyl group.
        \item If $\o = \o^+$ is the unipotent datum associated to the characteristic polynomial $p_{\o}(T) = (T - 1)^2 \in F[T]$, $$J_{\o}(f) = \vol([\SL_2])f(1) + \sum_{\substack{\kappa \in \widehat{[\G_m]}\\\kappa^2 = 1, \kappa \neq 1}} Z(f^\kappa, \kappa|\cdot|^1) + \left.\frac{d}{ds} sZ(f^1, |\cdot|^{1 + s})\right|_{s = 0},$$ for certain explicit functions $\{ f^\kappa \}_{\kappa \in \widehat{[\G_m]}, \kappa^2 = 1}$ constructed from $f$ whose definition we elide at the moment. Here $Z(-, \kappa|\cdot|^s)$ is the Tate zeta integral. An entirely analogous result holds for the unipotent datum $\o^-$ associated to the characteristic polynomial $(T + 1)^2 \in F[T]$.
    \end{enumerate}    
\end{introthm}

In fact, we obtain a completely explicit expression for all of $J^T_{\o}(f)$, not just its constant term.

We arrived at the formula for the unipotent contribution to the geometric RTF by considering the \emph{Springer resolution} of the rational nilpotent cone in $\mathfrak{sl}_{2, F}$. There is a natural ``correspondence''

\begin{center}
    \begin{tikzcd}
	& {T^*(B'\backslash G')} \\
	\mathbb{P}^1 = {B'\backslash G'} && {\overline{\o^{+, \reg}} = \mathcal{N}} %& {G' \cdot \eta_b}
	\arrow[from=1-2, to=2-1]
	\arrow[from=1-2, to=2-3]
	%\arrow[hook', from=2-4, to=2-3]
\end{tikzcd}
\end{center}
associated with the Springer resolution which can be used to group together appropriate terms of the truncated unipotent kernel, in such a way that the integral defining $J_{\o^+}(f)$ decomposes into a constant term, followed by a term which unfolds into a sum of Tate zeta integrals. These Tate integrals can be understood as integrals along the open $\G_m$-orbits in the fibers of the cotangent bundle $T^*\mathbb{P}^1 \to \mathbb{P}^1$. The crucial detail required to perform this unfolding is to replace the natural squaring $\G_m$-action on the fibers of the left-most map with the standard $\G_m$-action. In the proof of \cref{thmC} part (3), this is accomplished by an application of adelic Fourier inversion, and uses the fact that the flag variety associated to $\SL_{2, F}$ and $\PGL_{2, F}$ are the same.

This observation is explained in more detail in \cref{geometric-viewpoint}.
 
\subsection{Acknowledgements}

It is a pleasure to thank my advisors Spencer Leslie and Sol Friedberg for their continual guidance and encouragement throughout the preparation of this article. This work would not have been possible without their support. I would especially like to thank Spencer for suggesting this project. I'd also like to thank Keerthi Madapusi and Dubi Kelmer for helpful discussions at various stages of this project. This forms part of my PhD thesis at Boston College.

\part{A Coarse RTF for Galois Symmetric Pairs}\label{part-one}

\section{Basic Notation and Conventions}\label{notation-conventions}

\subsection{Notation}

Let $F$ be a number field, let $\A$ denote the ring of adeles of $F$, and let $\A^{\x}$ denote the group of ideles of $F$. We denote the ring of finite adeles by $\A_f$. We write $| \cdot |_{\A} : \A \to \R^+$ for the usual adelic absolute value. If $v$ is a place of $F$, we write $F_v$ for the corresponding local field. Write $| \cdot |_v : F_v \to \R^+$ for the normalized $v$-adic absolute value.

Let $\overline{F}$ be a fixed algebraic closure of $F$. Fix $E \subset \overline{F}$ a quadratic extension of $F$. We write $c \in \Gal(E/F)$ for the unique non-trivial automorphism. If $x \in E$, we also frequently write $\overline{x}$ for $c(x)$. For $v$ a place of $F$, set $E_v \coloneqq E \otimes_F F_v$.

We denote by $\G_m$ the multiplicative group scheme over $F$. An $F$-algebraic group is an affine, finite-type group scheme over $F$. If $G$ is an $F$-algebraic group, we write $X_*(G) \coloneqq \Hom_F(\G_m, G)$, and $X^*(G) \coloneqq \Hom_F(G, \G_m)$ for the cocharacter and character groups of $G$, respectively. Write $[G]$ for the adelic quotient $G(F) \backslash G(\A)$.

\subsection{Conventions}

Unless specified otherwise, all schemes are assumed to be $F$-schemes and all morphisms between schemes are morphisms of $F$-schemes. To avoid clutter in subscripts, we will often simply write e.g. $\eta \in X$ instead of $\eta \in X(F)$.

If $K/F$ is a finite étale $F$-algebra, and $X$ is a quasi-projective $K$-scheme, then we denote by $\Res_{K/F} X$ the Weil restriction of scalars of $X$ along $\Spec K \to \Spec F$. It is an $F$-scheme.

When working explicitly in coordinates, we often fix a presentation $E = F(\tau^{1/2})$ where $\tau^{1/2} \in E$ is an element with trace zero.

\section{A Review of Galois Symmetric Spaces}\label{galois-symmetric-space-review}

The goal of this section is to review the properties of Galois symmetric spaces. We begin with the definitions and basic properties.

Throughout this section, we fix $E \subset \overline{F}$ a quadratic field extension of $F$ in the algebraic closure. Let $c \in \Gal(E/F)$ denote the unique non-trivial element.

\subsection{Definition, Basic Properties}

In this subsection, fix $H$ an affine $F$-algebraic group. The material in this section is purely geometric and is valid for arbitrary fields $F$. The material in this section also holds equally when $E$ is the split quadratic étale $F$-algebra $E = F \x F$ and the Galois automorphism $c$ is given by swapping factors, however the proofs in that case follow almost immediately from the definitions, so we focus exclusively on the non-split case.

\begin{definition}
    A Galois structure on a smooth affine $F$-algebraic group $H$ of finite-type is the data of an affine $F$-algebraic group $H'$ and an isomorphism $H \xrightarrow{\sim} \Res_{E/F} H'_E$.
\end{definition}

We often leave the isomorphism $H \simeq \Res_{E/F}H'_E$ implicit, and refer to a Galois structure on $H$ via the pair $(H, H')$, called a Galois symmetric pair.

Our first order of business is to describe how a Galois structure can be used to produce an involution on $H$, called the Galois involution. This boils down to producing an involution on $\Res_{E/F} H'_E$; the Galois involution on $H$ is obtained by conjugating this involution by the isomorphism $H \simeq \Res_{E/F} H'_E$.

Consider a test $F$-algebra $A$, and write $H' = \Spec \O(H')$. We have canonical identifications $$(\Res_{E/F} H'_E)(A) = \Hom_E(\O(H')_E, A \otimes_F E) \simeq \Hom_F(\O(H'), A \otimes_F E).$$ We also have a canonical map of $F$-algebras $1 \otimes_F c : A \otimes_F E \to A \otimes_F E$. Applying the functor $H' = \Hom_F(\O(H'), -)$ we obtain the Galois involution $$\theta(A) : (\Res_{E/F} H'_E)(A) \to (\Res_{E/F} H'_E)(A).$$ This construction is visibly functorial, so by Yoneda lemma, this determines a canonical involution $\theta : \Res_{E/F} H'_E \to \Res_{E/F} H'_E$. Moreover it is clear from this description that the canonical inclusion $H' \into \Res_{E/F} H'_E$ is given by the inclusion of the $\theta$-fixed locus in $\Res_{E/F} H'_E$.

\begin{proposition}\label{base-change-splits}\
    \begin{enumerate}
        \item The base change along $\Spec E \to \Spec F$ of $\Res_{E/F} H'_E$ is canonically identified with $H'_E \x H'_E$.
        \item Under this identification, the Galois involution $\theta : \Res_{E/F} H'_E \to \Res_{E/F} H'_E$ is identified with the map $H'_E \x H'_E \to H'_E \x H'_E$ which swaps coordinates.
    \end{enumerate}
\end{proposition}
\begin{proof}
    Let $A$ be a test $E$-algebra with structure morphism $s : E \to A$. Then there is a canonical isomorphism of $E$-algebras $A \otimes_F E \xrightarrow{\sim} A \otimes_E (E \otimes_F E)$ given by $a \otimes_F e \mapsto a \otimes_E 1 \otimes_F e$. The structure morphism for this last algebra is determined by the following diagram
    \begin{center}
        \begin{tikzcd}
        	F & E & A \\
        	E & {E \otimes_F E} & {A \otimes_E (E \otimes_F E)}
        	\arrow[from=1-1, to=1-2]
        	\arrow[from=1-1, to=2-1]
        	\arrow["s", from=1-2, to=1-3]
        	\arrow[from=1-2, to=2-2]
        	\arrow[from=1-3, to=2-3]
        	\arrow["{e \mapsto e \otimes_F 1}"', from=2-1, to=2-2]
        	\arrow["\ulcorner"{anchor=center, pos=0.125, rotate=180}, draw=none, from=2-2, to=1-1]
        	\arrow[from=2-2, to=2-3]
        	\arrow["\ulcorner"{anchor=center, pos=0.125, rotate=180}, draw=none, from=2-3, to=1-2]
        \end{tikzcd}
    \end{center}
    where every square is cocartesian. For any $E$-algebra $E \to B$, write ${}^cB$ for the $E$-algebra with structure morphism $E \xrightarrow{c} E \to B$. Then via the canonical isomorphism $E \otimes_F E \xrightarrow{\sim} E \x {}^cE$ given by $$e_1 \otimes_F e_2 \mapsto (e_1e_2, e_1c(e_2)),$$ we have an isomorphism $A \otimes_E (E \otimes_F E) \xrightarrow{\sim} A \x {}^cA$ given by $$ a \otimes_E e_1 \otimes_F e_2 \mapsto (e_1e_2 \cdot a, e_1c(e_2) \cdot a).$$

    Thus, we compute $(\Res_{E/F} H'_E)_E(A) = \Hom_E(\O(H')_E, A \otimes_F E) \simeq \Hom_E(\O(H')_E, A \otimes_E (E \otimes_F E)) \simeq \Hom_E(\O(H')_E, A \x {}^cA) \simeq \Hom_E(\O(H')_E, A) \x \Hom_E(\O(H')_E, {}^c A) \simeq \Hom_E(\O(H')_E, A) \x \Hom_E({}^c\O(H')_E, A) \simeq \Hom_E(\O(H')_E \otimes_E {}^c \O(H')_E, A)$. Finally, note that since $E \to \O(H')_E$ arises as a base change from $F$, ${}^c\O(H')_E = \O(H')_E$. This completes the proof of claim (1).
    
    To see the second claim, let $x \in \Hom_E(\O(H')_E, A \otimes_F E)$ and consider the following diagram, where the left-most square is cocartesian.
    \begin{center}
        \begin{tikzcd}
        	F & {\O(H')} \\
        	E & {\O(H')_E} & {A \otimes_F E} & {A \otimes_F E} \\
        	&& {A \otimes_E (E \otimes_F E)} & {A \otimes_E (E \otimes_F E)} \\
        	&& {A \x {}^c A} & {A \x {}^c A}
        	\arrow[from=1-1, to=1-2]
        	\arrow[from=1-1, to=2-1]
        	\arrow[from=1-2, to=2-2]
        	\arrow["y", dashed, from=1-2, to=2-3]
        	\arrow[from=2-1, to=2-2]
        	\arrow["\ulcorner"{anchor=center, pos=0.125, rotate=180}, draw=none, from=2-2, to=1-1]
        	\arrow["x", from=2-2, to=2-3]
        	\arrow["{(x_1, x_c)}"', from=2-2, to=4-3]
        	\arrow["{1 \otimes_F c}", from=2-3, to=2-4]
        	\arrow[from=2-3, to=3-3]
        	\arrow[from=2-4, to=3-4]
        	\arrow["(\star)", from=3-3, to=3-4]
        	\arrow[from=3-3, to=4-3]
        	\arrow[from=3-4, to=4-4]
        	\arrow["(\star\star)", from=4-3, to=4-4]
        \end{tikzcd}    
    \end{center}
    The dotted map $y \in \Hom_F(\O(H'), A \otimes_F E)$ is the $F$-morphism determined by $x$. The vertical arrows to the right are the various isomorphisms described above. The components of the map $(x_1, x_c) : \O(H')_E \to A \x {}^c A$ are obtained via the composite of the vertical isomorphisms and projection onto the first and second factor respectively. The two starred horizontal maps are defined by the vertical maps and the top-most map; we will compute them momentarily.

    Starting with an $x$ as above, the Galois involution acts by considering the corresponding $y$, composing it with $1 \otimes_F c$, and then forming the corresponding map $\O(H')_E \to A \otimes_F E$ by extending $E$-linearly. Thus, the second claim will follow if we can identify the bottom-most map $A \x {}^c A \to A \x {}^c A$ as the map which swaps coordinates.

    We calculate that the first horizontal map $(\star)$ is given by $$a \otimes_E e_1 \otimes_F e_2 = e_1a \otimes_E 1 \otimes_F e_2 \mapsfrom e_1a \otimes_F e_2 \mapsto e_1a \otimes_F c(e_2) \mapsto a \otimes_E e_1 \otimes_F c(e_2).$$ Finally, our desired map $(\star\star)$ is given by $$(e_1e_2 \cdot a, e_1c(e_2) \cdot a) \mapsfrom a \otimes_E e_1 \otimes_F e_2 \mapsto a \otimes_E e_1 \otimes_F c(e_2) \mapsto (e_1c(e_2) \cdot a, e_1e_2 \cdot a).$$ Thus, we see that the bottom-most map is indeed the map which swaps coordinates, and claim (2) is proved.
\end{proof}

\quash{
\begin{remark}
    From the proof above, one can also extract the $\Gal(\overline{F}/F)$-action on $H(\overline{F}) = H'_E(\overline{F} \otimes_F E) = \Hom_E(\O(H')_E, \overline{F} \otimes_F E)$; it is given by applying the Galois automorphism in the \emph{left} coordinate. If we identify $\Hom_E(\O(H')_E, \overline{F} \otimes_F E) \simeq \Hom_E(\O(H')_E, \overline{F}) \x \Hom_E(\O(H')_E, \overline{F})$, we see that the Galois action is trivial on $\Gal(\overline{F}/E) \subset \Gal(\overline{F}/F)$ and $c \in \Gal(E/F)$ acts on $(h_1, h_2) \in H(\overline{F}) \x H(\overline{F})$ by $c : (h_1, h_2) \mapsto (c(h_2), c(h_1))$. In this same model, $\theta$ acts by $(h_1, h_2) \mapsto (h_2, h_1)$.
\end{remark}
}

\begin{remark}
    If $H$ supports one Galois structure $(H, H')$, then up to isomorphism all others are classified by $E$-forms of $H'$, which by basic Galois cohomological considerations are classified by $H^1(\Gal(E/F), \Aut_E(H_E))$. When $H$ is semisimple and either simply connected or adjoint, $H \simeq \Res_{K/F} G$ for a unique finite étale $F$-algebra $K$ and group $G$ over $\Spec K$, which is connected, semisimple and absolutely almost-simple over each factor of $K$ \cite[Thm. 24.3]{milne}. Thus, in this case, $H$ admits a Galois structure if and only if $K = E \x \cdots \x E$ (say with $m$ factors) and each factor $G_i$ for $i = 1, \ldots, m$ descends to $F$. If we further assume $K = E$, and $H'$ is a fixed descent of $G$ to $F$, then the remaining Galois structures supported by $H$ are classified by the $E$-forms of the semisimple, absolutely almost-simple $F$-group $H'$. In many cases, classifying such forms boils down to classifying $E$-forms of certain $F$-algebras with auxiliary structure. For example, the Galois structures supported by $\Res_{E/F} \SL_{2, E}$ are in bijection with quaternion algebras $B/F$ which contain $E/F$; the corresponding Galois symmetric pair is $(\Res_{E/F} \SL_{2, E}, \SL_1(B))$.
\end{remark}

If $(H, H')$ is a Galois structure of $H$, we can form the associated Galois symmetric space $X \coloneqq H/H'$. This is an affine, finite-type $F$-scheme. By the proposition above, we see that $$X_E \simeq H'_E \x H'_E/\Delta H'_E.$$ The Galois symmetric space $X$ is equipped with canonical maps $\pi : H \onto X$, $s : X \into H$ such that the diagram below commutes:
\begin{center}
\begin{tikzcd}
	H && H \\
	& X
	\arrow["{h \mapsto h\theta(h)^{-1}}", from=1-1, to=1-3]
	\arrow["\pi"', two heads, from=1-1, to=2-2]
	\arrow["s"', hook, from=2-2, to=1-3]
\end{tikzcd}
\end{center}
We call $s : X \into H$ the symmetrization map, and we will often identify $X$ with its (schematic) image under $s$.

As a right quotient of $H$, $X$ admits a natural (left) $H$-action, which under the symmetrization map corresponds to $\theta$-twisted conjugation $$g \star x \coloneqq gx\theta(g)^{-1} \text{ for all $g \in H, x \in X$}.$$ Note that when restricted to $H' \subset H$, this action is simply conjugation.

It is very useful to have a more concrete description of $X$ inside $H$. Define $\widetilde{X} \coloneqq \{ h \in H : h = \theta(h)^{-1} \}$. As the locus in $H$ where two morphisms agree, it is manifestly a closed subscheme of $H$. One easily checks that the symmetrization map factors through $\widetilde{X}$.

\begin{proposition}\label{concrete-description-of-symm-spc}
    The canonical morphism $X \into \widetilde{X}$ is an isomorphism, and thus $$X = \{ h \in H : h\theta(h) = 1 \}.$$
\end{proposition}
\begin{proof}
    It suffices to check this after base change to $\overline{F}$. By the previous proposition, we can without loss of generality identify $H_{\overline{F}}$ with $H'_{\overline{F}} \x H'_{\overline{F}}$, and we can identify the Galois involution with the map which swaps coordinates. Thus $$\widetilde{X}_{\overline{F}} = \{ (h_1, h_2) \in H'_{\overline{F}} \x H'_{\overline{F}} : (h_1, h_2) = (h_2^{-1}, h_1^{-1}) \} = \{ (h_1, h_1^{-1}) \in H'_{\overline{F}} \x H'_{\overline{F}} \}.$$ But this is clearly in the image of $s_{\overline{F}} : X_{\overline{F}} \into H_{\overline{F}}$.
\end{proof}

\subsection{Levi Galois Symmetric Spaces}

For the remainder of this section, $H = \Res_{E/F} H'_E$ will denote a connected reductive $F$-algebraic group with Galois structure $(H, H')$. Note that as the fixed locus of the Galois involution, $H'$ is necessarily connected reductive \cite[\S1, Paragraph 3]{vust}.

If $P' \subset H'$ is a parabolic subgroup, it is clear that $\Res_{E/F} P'_E \subset \Res_{E/F} H'_E$ is a $\theta$-stable parabolic subgroup, and conversely, any $\theta$-stable parabolic subgroup arises in this way. Similarly, one can show that if $L' \subset H'$ is a Levi subgroup for $P'$, then $L \coloneqq \Res_{E/F} L'_E \subset \Res_{E/F} H'_E$ is a $\theta$-stable Levi for the corresponding parabolic $P = \Res_{E/F} P'_E$, which naturally endows $(L, L')$ with the structure of a Galois symmetric pair compatible with $(H, H')$. Let $X_L$ denote the corresponding Galois symmetric space.

\begin{proposition}
    Let $L$ be a $\theta$-stable Levi as above. The following diagram of $F$-schemes is cartesian:
    \begin{center}
        \begin{tikzcd}
	{X_L} && L \\
	X && H
	\arrow["{s_L}", from=1-1, to=1-3]
	\arrow[hook, from=1-1, to=2-1]
	\arrow[hook, from=1-3, to=2-3]
	\arrow["s"', from=2-1, to=2-3]
\end{tikzcd}
    \end{center}
    Thus, the natural map $X_L \into X$ is a closed embedding, and under the symmetrization map, $X \cap L = X \x_H L = X_L$.
\end{proposition}
\begin{proof}
    This is immediate from the description of $X$ in \cref{concrete-description-of-symm-spc}.
\end{proof}

\subsection{Rational Points, Adelic Points}\label{rational-adelic-pts}

Let $(H, H')$ be a Galois symmetric pair and let $X$ denote the corresponding symmetric space.

Recall that $\A_{\overline{F}}$ is defined as the filtered colimit $\A_{\overline{F}} \coloneqq \colim_{K/F} \A_{K}$ in the category of $F$-algebras, where $K/F$ runs over all finite extensions of $F$. For any finitely presented $F$-scheme $Y$, we have a canonical identification $Y(\A_{\overline{F}}) \simeq \colim_{K/F} Y(\A_K)$. In particular, $\Gal(\overline{F}/F)$ acts on $Y(\A_{\overline{F}})$, with fixed points $Y(\A_F)$. If $Y$ is a group $F$-scheme, we denote the corresponding pointed non-abelian Galois cohomology sets by $H^i(\A, Y)$ (for $i = 0, 1$). We also write $H^i(E/F, Y)$ for the Galois cohomology of $Y(E)$ with the natural $\Gal(E/F)$ action, and likewise $H^i(\A_E/\A_F, Y)$ for the Galois cohomology of $Y(\A_E)$ with the natural $\Gal(E/F)$ action (again $i = 0, 1$).

\begin{definition}\
    \begin{enumerate}
        \item $\ker^1(F, X) \coloneqq \ker[H^1(F, H') \to H^1(F, H)]$.
        \item $\ker^1(\A, X) \coloneqq \ker[H^1(\A, H') \to H^1(\A, H)]$.
    \end{enumerate}
\end{definition}

Since $X(\overline{F})$ (resp. $X(\A_{\overline{F}})$) is a single $\theta$-twisted $H(\overline{F})$-orbit (resp. $H(\A_{\overline{F}})$-orbit) the groups above classify the rational (resp. adelic) $G$-orbits in $X$. Write \begin{equation}\label{alpha}
    \alpha : \ker^1(F, X) \to \ker^1(\A, X)
\end{equation} for the canonical map induced by $\overline{F} \to \A_{\overline{F}}$.

\quash{
\begin{lemma}\
    \begin{enumerate}
        \item $\ker^1(F, X) \simeq H^1(E/F, H')$.
        \item $\ker^1(\A, X) \simeq H^1(\A_E/\A_F, H')$.
    \end{enumerate}
\end{lemma}
\begin{proof}
    We sketch the proof of (1), the proof of (2) follows mutatis mutandis. Inflation yields a map $H^1(E/F, H') \to H^1(F, H')$. By the non-commutative version of Shapiro's lemma, we can identify $H^1(F, H) \simeq H^1(E, H'_E) = H^1(E, H')$ in such a way that the map $H^1(F, H') \to H^1(F, H)$ is identified with the restriction map $H^1(F, H') \to H^1(E, H')$. From the inflation-restriction exact sequence, it follows that $\ker^1(F, X)$ can be identified with $H^1(E/F, H')$.
\end{proof}

\begin{remark}
    By the non-commutative version of Shapiro's lemma, we can identify $H^1(F, H) \simeq H^1(E, H'_E) = H^1(E, H')$, in such a way that map $H^1(F, H') \to H^1(F, H)$ is identified with the restriction map $H^1(F, H') \to H^1(E, H')$. Thus, for Galois symmetric spaces, $\ker^1(F, X)$ is identified with $H^1(E/F, H')$, i.e. the pure inner forms of $H'$ which become trivial after base change to $E$. In particular if $\eta \in X(F)$ lies in the $H(F)$-orbit corresponding to $\xi \in \ker^1(F, X)$, and by abuse we also write $\xi$ for the cocycle corresponding to $\eta$, the twisted $F$-algebraic group $H'_{\xi}$ (the stabilizer of $\eta$ under the $\theta$-twisted conjugation) is an $E$-form of $H'$.
\end{remark}
}

\begin{lemma}\
    \begin{enumerate}
        \item $\ker^1(F, X) \simeq \ker[H^1(E/F, H') \to H^1(E/F, H)]$.
        \item $\ker^1(\A, X) \simeq \ker[H^1(\A_E/\A_F, H') \to H^1(\A_E/\A_F, H)]$.
    \end{enumerate}
\end{lemma}
\begin{proof}
    We prove claim (1); claim (2) is proved analogously by working place by place. Let $\eta \in X(F)$ be a basepoint for the $H(F)$-orbit indexed by $\xi \in \ker^1(F, X)$. Then by \cref{base-change-splits}, we can find $h \in H(E) \subset H(\overline{F})$ such that $\eta = h\theta(h)^{-1}$. Thus, we see that $\xi$ is equivalent to the cocycle obtained by applying the inflation map to the cocycle $c \mapsto h^{-1}c(h)$. Note that $c(\eta) = \eta$ implies $\theta(h^{-1}c(h)) = h^{-1}c(h)$, and thus $h^{-1}c(h) \in H'(E)$. Changing the basepoint doesn't change the equivalence class of the corresponding cocycle, so we obtain an isomorphism $\ker^1(F, X) \simeq \ker[H^1(E/F, H') \to H^1(E/F, H)]$, as desired.
\end{proof}

This has the following important consequence

\begin{corollary}\label{twisted-stabilizers-give-galois-pairs}
    Let $\eta$ be a basepoint in the $\theta$-twisted $H(F)$-orbit indexed by $\xi \in \ker^1(F, X)$ and write $H_\xi$ for the $\theta$-twisted stabilizer of $\eta$, which is also the $F$-algebraic group obtained by twisting $H'$ using the cocycle determined by $\eta$.

    Then $(H_{\xi})_E \simeq H'_E$ and $(H, H_{\xi})$ is another Galois symmetric pair with Galois involution given by $\theta_\xi(h) \coloneqq \eta \theta(h) \eta^{-1}$.
\end{corollary}
\begin{proof}
    That $(H_\xi)_E \simeq H'_E$ is an immediate consequence of the previous lemma. Indeed $H^1(E/F, H')$ classifies $H'$-torsors over $F$ which become trivial after base change to $E$. The final claim follows from the fact that for $h \in H$, $h \star \eta = \eta$ if and only if $\theta_\xi(h) = h$.
\end{proof}

\begin{definition}
    We say a class $\xi \in \ker^1(\A, X)$ is relevant if it lies in the image of $\alpha : \ker^1(F, X) \to \ker^1(\A, X)$. We call an $H(\A)$-orbit relevant if its corresponding class in Galois cohomology is relevant.
\end{definition}

In particular, note that $\xi \in \ker^1(\A, X)$ is relevant if and only if a representative $\eta_\xi \in X(\A)$ for the corresponding $\theta$-twisted $H(\A)$-orbit can be chosen such that $\eta_{\xi} \in X(F) \subset X(\A)$.

\begin{lemma}\label{fibers-of-alpha-are-finite}
    The map $\alpha : \ker^1(F, X) \to \ker^1(\A, X)$ has finite fibers.
\end{lemma}
\begin{proof}
    Let $\zeta \in \ker^1(\A, X)$ be a relevant cocycle. Since $H^1(\A, H') \simeq \bigoplus_v H^1(F_v, H') \into \prod_v H^1(F_v, H')$, for all but finitely many $v$, $\zeta_v$ is trivial \cite[p.~298, Cor. 1]{platonov-rapinchuk}. Let $S$ be the finite set of places $v$ where $\zeta_v$ is non-trivial. Now, by \cite[p.~323, Thm. 6.19]{platonov-rapinchuk}, the fiber of $H^1(F, H') \to \prod_{v \notin S} H^1(F_v, H')$ above the trivial element is finite, so we win.
\end{proof}

In other words, each relevant $H(\A)$-orbit meets finitely many $H(F)$-orbits when intersected with $X(F)$.

%\begin{remark}
%    In practice, one can be much more precise. When $H'$ is semisimple and simply connected, $H^1(F_v, H')$ is trivial for all finite places $v$, see \cite[p.~284, Thm. 6.4]{platonov-rapinchuk}. In turn, this allows one to compute $H^1(F_v, H')$ for arbitrary semisimple $H'$ in terms of $H^2(F_v, Z'_{sc})$, where $Z'_{sc}$ is the kernel of the canonical central isogeny $H'_{sc} \to H'$. When $v$ is a real place Borovoi has shown, in the spirit of a result of Borel-Serre, that $H^1(F_v, H') = H^1(\R, H')$ admits a concrete description in terms of a certain Weyl group quotient of $H^1(\R, T)$, where $T$ is a \emph{fundamental torus} in $H'_{\R}$, see \cite{borovoi-real-cohomology}.
%\end{remark}

\subsection{Integral Models and Basic Functions}\label{int-models-basic-fns}

Write $\O_F$ for the ring of integers of $F$. By spreading out, there exists a smooth connected reductive group scheme $\mathscr{H}' \to \Spec \O_F[1/N]$ with generic fiber isomorphic to $H'$, where $N$ is some finite product of prime ideals in $\Spec \O_F$. Set $\O_E[1/N] \coloneqq \O_E \otimes_{\O_F} \O_F[1/N]$, this is a finite free $\O_F[1/N]$-algebra. Then we can form the affine group scheme $\mathscr{H} \coloneqq \Res_{\O_E[1/N]/\O_F[1/N]} \mathscr{H}'_{\O_E[1/N]}$, and we can consider the integral Galois symmetric pair $(\mathscr{H}, \mathscr{H}')$, where the Galois involution is defined in exactly the same way as above.

If $v$ is a place of $F$, write $F_v$ for the completion of $F$ with respect to $v$, and write $\O_v$ for the ring of integers of $F_v$.

For all places $v$ of $F$ not dividing $N$, we define hyperspecial maximal compact subgroups $K_v \coloneqq \mathscr{H}(\O_v) \subset H(F_v)$ and $K'_v \coloneqq \mathscr{H}'(\O_v) \subset H'(F_v)$, and whenever we discuss maximal compact subgroups of $H(\A)$ or $H'(\A)$, we will implicitly assume that at almost all places these subgroups are given by hyperspecial maximal compact subgroups as above.

Define $\mathscr{X} \to \Spec \O_F[1/N]$ as the fppf quotient $\mathscr{H}/\mathscr{H}'$. This exists as an affine $\O_F[1/N]$-scheme, and yields an integral model for $X \to \Spec F$. For $v$ not dividing $N$, we set $X(\O_v) \coloneqq \mathscr{X}(\O_v)$.

Finally, for all $v$ not dividing $N$, we define the function $\1_{X(\O_v)} \in \Cc(X(F_v))$ to be the indicator function of the compact set $\mathscr{H}(\O_v) \star 1 = X(\O_v)$. The equality $\mathscr{H}(\O_v) \star 1 = X(\O_v)$ is a consequence of the vanishing of $H^1_{fppf}(\Spec \O_v, \mathscr{H}')$ which follows from the fact that $\mathscr{H}'$ is smooth and unramified at $v$.

\begin{definition}
    A factorizable compactly supported function $f : X(\A) \to \C$ is a product of functions of the form $f_v : X(F_v) \to \C \in \Cc(X(F_v))$, where for all but finitely many places $v$, $f_v = \1_{X(\O_v)}$.

    We define $\Cc(X(\A))$ to be the space spanned by finite linear combinations of factorizable compactly supported functions.
\end{definition}

Later, we will also make use of the notion of a Schwartz function $\Schw(\A)$ on the adeles. A factorizable Schwartz function $f : \A_F \to \C$ is a product of functions $f_v : F_v \to \C$, where for all finite places $v < \infty$, $f_v \in \Cc(F_v)$, for all infinite places $v | \infty$, $f_v$ is Schwartz in the usual sense, and for all but finitely many finite places $f_v = \1_{\O_v}$. The Schwartz space $\Schw(\A)$ is the space spanned by finite linear combinations of factorizable Schwartz functions.

\section{A Coarse RTF for Galois Symmetric Pairs}

The goal of this section is to introduce two families of integral kernels, one geometric the other spectral, associated to certain Galois symmetric pairs $(G, G')$, and explain how these kernels can be systematically modified to obtain a regularized coarse relative trace formula identity.

The truncated geometric kernel is essentially a regularized theta series over the $F$-points of $X = G/G'$. The truncated spectral kernel is a truncated version of the automorphic kernel associated to $G$.

\subsection{Setup}\label{setup}

Fix $(G, G')$ a Galois symmetric pair, where $G'$ is a connected reductive $F$-algebraic group and $G = \Res_{E/F} G'_E$. Let $\theta : G \to G$ denote the corresponding Galois involution.

As a general convention, notation without a prime refers to an object associated to $G$, while notation with a prime refers to an object associated to $G'$.

\subsubsection{Tori, Parabolic Subgroups}

\begin{lemma}
    Let $A_0' \subset G'$ be a maximal split torus. Then there exists a $\theta$-stable maximal split torus $A_0 \subset G$ such that $(A_0^\theta)^{\circ} = A_0'$.
\end{lemma}
\begin{proof}
    Consider $M \coloneqq Z_G(A_0')$, the scheme-theoretic centralizer of $A_0'$ in $G$. This is a connected reductive group, in fact a Levi subgroup of $G$. Moreover $M$ is $\theta$-stable. Indeed, if $g \in Z_G(A_0')$ then $gag^{-1} = a$ for $a \in A_0'$, so $\theta(g)a\theta(g)^{-1} = a$ for $a \in A_0'$ and thus $\theta(g) \in Z_G(A_0')$. By \cite[\S2, Prop. 2.3]{helminck-wang}, there exists a $\theta$-stable maximal torus $T \subset M$ whose split part $T_{spl}$ is a maximal split torus in $M$. Since $A_0'$ lies in some maximal split torus in $M$, and all such are conjugate by elements of $M(F)$, we see that $A_0'$ is a subtorus of every maximal split torus in $M$, and in particular, a subtorus of $T_{spl}$. Since $T$ is $\theta$-stable, $T_{spl}$ is also $\theta$-stable. It follows that $A_0' \subset T_{spl}^{\theta} \subset G'$, and thus $A_0' = (T_{spl}^{\theta})^{\circ}$ by maximality of $A_0'$. Finally, note that $T_{spl}$ is in fact a maximal split torus in $G$; if $T_{spl} \subset S$ where $S$ is a split torus in $G$, then $A_0' \subset T_{spl} \subset S$, and thus $A_0'$ commutes with $S$. Thus $T_{spl} \subset S \subset M$, contradicting the maximality of $T_{spl}$. The claim follows by setting $A_0 \coloneqq T_{spl}$.
\end{proof}

\begin{remark}
    The argument above applies to any involution, not just Galois involutions.
\end{remark}

Fix $A_0'$ a maximal split torus in $G'$, and let $A_0$ be a $\theta$-stable maximal split torus in $G$ such that $(A_0^\theta)^{\circ} = A_0'$. By the lemma above, such a choice is always possible.

We call the parabolic subgroups of $G$ containing $A_0$ semi-standard; the set of such is denoted $\F^G(A_0)$. For a subgroup $H \subset G$, we write $\F^G(H)$ to denote the semi-standard parabolic subgroups of $G$ containing $H$.

If $P \in \F^G(A_0)$, we write $N_P$ for the unipotent radical of $P$, and
$M_P$ for the unique Levi subgroup of $P$ containing $A_0$. Write $A_P$ to
denote the maximal $F$-split torus in the center of $M_P$; it is also the
subtorus of $A_0$ centralizing $M_P$. We write $\P^G(A_0)$ for the set of minimal semi-standard parabolics $P$ with $A_P = A_0$.

\subsubsection{A Simplifying Assumption}

We now make the following simplifying assumption:
\begin{equation}\label{assumption-p}
        \text{ The split rank of $G'$ and $G$ agree. }\tag{$\star$}
\end{equation}

This assumption is satisfied for any Galois symmetric pairs $(G, G')$ where $G'$ is a split reductive group.

With our choice of maximal split tori as above, this implies that $A_0' = A_0$. It is well-known that the semi-standard parabolic subgroups of $G$ relative to $A_0$ are given by $P_G(\lambda) = \{ g \in G : \lim_{t \to 0} \lambda(t)g\lambda(t)^{-1} \text{ exists} \}$ for $\lambda \in X_*(A_0)$. If $A_0' = A_0$, it follows that $$\theta(\lim_{t \to 0} \lambda(t)g\lambda(t)^{-1}) = \lim_{t \to 0} \theta(\lambda(t))\theta(g)\theta(\lambda(t))^{-1} = \lim_{t \to 0} \lambda(t)\theta(g)\lambda(t)^{-1}.$$ Thus, $g \in P_G(\lambda)$ if and only if $\theta(g) \in P_G(\lambda)$, and each $P_G(\lambda)$ is $\theta$-stable, and so each semi-standard parabolic subgroup of $G$ is $\theta$-stable.

It follows that

\begin{lemma}\label{lemma-p}
    Under the assumption \labelcref{assumption-p}, the map $P' \mapsto \Res_{E/F} P'_E$ is a bijection between $\F^{G'\!}(A_0')$ and $\F^G(A_0)$.
\end{lemma}

Moreover, if we fix a minimal semi-standard parabolic $P'_0 \in \F^G(A_0')$ in $G'$ with corresponding minimal parabolic $P_0 = \Res_{E/F} P'_{0, E}$, this bijection induces a bijection between standard parabolic subgroups for $G'$ and $G$.

\begin{remark}
    An example of a Galois symmetric pair for which this assumption fails is $(\Res_{E/F} \SL_{2, E}, \SL_1(B))$, where $B$ is a non-split quaternion algebra over $F$ which splits over $E$ (i.e. which contains a copy of $E$).
\end{remark}

\begin{remark}
    We expect similar results to hold when assumption \labelcref{assumption-p} is removed. The generalization to arbitrary Galois symmetric pairs is part of ongoing work. However, the additional complications that ensue do little to shed light on how modified kernels are constructed, so for the moment we are content to treat this simpler case.
\end{remark}

\subsubsection{Geometry of Chambers}\label{chambers}

Let $W$ denote the Weyl group of the pair $(G, A_0)$. Set $\a_0 \coloneqq X_*(A_0) \otimes_{\Z} \R$, and for any $P \in \F^G(A_0)$, set $\a_P \coloneqq X_*(A_P) \otimes_{\Z} \R$. Then we have a natural inclusion $\a_P \subset \a_0$. Similarly, write $\a_P^* \coloneqq X^*(A_P) \otimes_{\Z} \R$.

Choose a $W$-invariant inner product $\langle \cdot, \cdot \rangle$ on $\a_0$,
such that on the orthogonal complement of $\a_G \subset \a_0$, it agrees with
the Killing form of $G$. Moreover, assume it takes rational values on $X_*(A_0)
\x X_*(A_0)$. Using this inner product, we can realize $X^*(A_P)$ as a lattice in $\a_0$, contained in $X_*(A_0) \otimes_{\Z} \Q = X^*(A_0) \otimes_{\Z} \Q$.

Let $M_P$ be the standard Levi subgroup of $P$. Then $M_P$ is isogenous to $M_P^{der} \x Z(M_P)$ and thus the abelianization $M_P/M_P^{der}$ is isogenous to $Z(M_P)$, the center of $M_P$. Since $A_P$ is the maximal split subtorus of $Z(M_P)$, we have a canonical identification $X^*(A_P) = X^*(Z(M_P))$ and since $Z(M_P)$ is isogenous to $M_P/M_P^{der}$, we have an identification of $\Q$-lattices $X^*(Z(M_P)) \otimes_\Z \Q \simeq X^*(M_P) \otimes_\Z \Q$. Thus, up to torsion, the natural map $X^*(M_P) \to X^*(A_P)$ induced by restriction is an isomorphism. It follows that $\langle \cdot , \cdot \rangle$ yields a natural pairing between $X^*(M_P)$ and $\a_P = X_*(A_P) \otimes_{\Z} \R$. This will be used below in the definition of the Harish-Chandra map.

This inner product also induces a topology on $\a_0$; when we speak of open and
closed sets in $\a_0$ we are always referring to open and closed sets relative
to this topology. For a set $S \subset \a_0$, we write $\overline{S}$ for the
closure of $S$ in $\a_0$. If $C \subset \a_0$ is a cone, we write $\rint C$ for the relative interior of $C \subset \a_0$. It is defined as the largest open subset of the span of $C$ contained inside of $C$. Note that $\rint\, \{ 0 \} = \{ 0 \}$. 

For $P \in \F^G(A_0)$, write $\Delta_P \subset \a_P$ for the set of simple
roots of $A_P$ acting on $N_P$. For $P, Q \in \F^G(A_0)$ such that $P \subset
Q$, we have a natural inclusion $\a_Q \subset \a_P$; write $\a_P^Q$ for the
orthogonal complement of $\a_Q$ in $\a_P$. Let $\Delta_P^Q$ be the subset of
$\Delta_P$ that vanishes on $\a_Q$. Then $\Delta_P^Q \subset \a_P^Q$, and is
moreover a basis for this space. Define $\widehat{\Delta}_P^Q \subset \a_P^Q$
to be the corresponding dual basis relative to the inner product above. When $Q
= G$, we omit the superscript $G$ and write $\Delta_P \coloneqq \Delta_P^G$ and
$\widehat{\Delta}_P \coloneqq \widehat{\Delta}_P^G$.

Define the open cone $$\a_P^+ \coloneqq \{ H \in \a_P : \langle H, \alpha
\rangle > 0, \forall \alpha \in \Delta_P \}.$$ Then $\rint \overline{\a_P^+} =
\a_P^+$. For all $P \subset Q$, the cones $\overline{\a_Q^+}$ run
over the faces of $\overline{\a_P^+}$.

If $C$ is a cone and $F$ is a face of $C$, we define the angle cone $A(F, C)$ as the Minkowski sum of the span of $F$ and the cone $C$, $$A(F, C) \coloneqq \langle F \rangle + C.$$ If $C$ is cone, we define the dual cone $C^\vee$ to be the cone consisting of vectors which pair non-negatively with vectors in $C$, $$C^\vee \coloneqq \{ v \in \a_{0} : \langle v, C \rangle \geq 0 \}.$$

We have that $$A(\overline{\a_Q^+},
\overline{\a_P^+}) = \{ H \in \a_P : \langle H, \alpha \rangle \geq 0, \forall
\alpha \in \Delta_P^Q \},$$ and $$A(\overline{\a_Q^+}, \overline{\a_P^+})^\vee
= \{ H \in \a_P : \langle H, \varpi \rangle \geq 0, \forall \varpi \in
\widehat{\Delta}_P^Q \}.$$

\subsubsection{Reduction Theory}

For $P_0 = M_0N_0 \in \P(A_0)$, a minimal semi-standard parabolic, let $K$ be a
maximal compact subgroup of $G(\A)$ in good position relative to $M_0$. For $P
\in \F^G(A_0)$, define $A_P^{\infty}$ to be the connected component of the
$\R$-points of the maximal $\Q$-split torus in $\Res_{F/\Q}A_P$. Define the
Harish-Chandra map $H_P : M_P(\A) \to \a_P$ as the unique map such that $$\langle
\xi, H_P(m) \rangle = \log |\xi(m)|_{\A}, \text{ for all $\xi \in X^*(M_P)$}$$ where $m \in M_P(\A)$. This is a
surjective homomorphism. Write $M_P(\A)^1$ for the kernel of $H_P$; it is given
by $$M_P(\A)^1 \coloneqq \{ m \in M_P(\A) : \forall \xi \in X^*(M_P),
|\xi(m)|_{\A} = 1 \}.$$ The restriction of $H_P$ to $A_P^{\infty}$ induces an
isomorphism onto $\a_P$, which realizes $M_P(\A)$ as the product
$A_P^{\infty}M_P(\A)^1$ of commuting subgroups. Set $[M_P]^1 \coloneqq M_P(F) \backslash M_P(\A)^1$.

We extend $H_P$ to all of $G(\A)$ using the adelic Iwasawa decomposition by,
$$H_P(x) = H_P(m) \text{ for } x = nmk \in G(\A),\; n \in N_P(\A), \; m \in
M_P(\A), \; k \in K.$$ This extended map is well-defined and agrees with the
original map on $M_P(\A)$.

\subsubsection{Relative Chambers}\label{rel-chambers}

Henceforth, fix a minimal parabolic subgroup $P_0' \in \P^{G'}\!(A_0')$ of $G'$
containing $A_0'$ with Levi decomposition $P_0' = M_0'N_0'$. As above, let $K'$
be a maximal compact subgroup of $G'(\A)$ in good position relative to $M_0'$.
Note that $\a_{P_0'} = \a_{0'}$, however the latter vector space is independent
of the choice of $P_0'$.

For $P' \in \F^{G'}\!(A_0')$, define $$\P^G(P') \coloneqq \{ P \in \F^G(M_0) :
\a_{P'}^+ \cap \a_P^+ \neq \emptyset \}.$$

% In Zydor's paper, M_0 above is replaced with the centralizer of A_0' in G, what he calls M_1; but since A_0' = A_0, and Z_G(A_0) = M_0 is the minimal Levi, we can omit mention of M_1.

%Now, for each $P \in \F^G(P_0')$, define
%$$\z_P \coloneqq \a_P \cap \a_{0'}, \;\;\; \z_P^+ \coloneqq \a_P^+ \cap \a_{0'}.$$
%Then $\z_P^+$ is an open cone in $\z_P$, and we have a disjoint union decomposition for $P' \in \F^G(P_0')$:
%\begin{equation*}
%    \a_{P'}^+ = \bigsqcup_{P \in \P^G(P')} \z_P^+, \;\; \overline{\a_{P'}^+} = \bigsqcup_{P \in \F^G(P_0')} \z_P^+.
%\end{equation*}

Then for $P' \in \F^G(P_0')$ we have a disjoint union decomposition
\begin{equation*}
    \a_{P'}^+ = \bigsqcup_{P \in \P^G(P')} \a_P^+, \;\;\;\; \overline{\a_{P'}^+} = \bigsqcup_{P \in \F^G(P_0')} \a_P^+.
\end{equation*}

Finally, note that the Weyl group of $A_0'$ in $G'$ agrees with the Weyl group for $A_0$ in $G$. Thus, the inner product on $\a_0$ induces an inner product on $\a_{0'}$ that is invariant under the Weyl group action.

If $S \subset \a_{0'}$, we write $[S]$ for the indicator function on $S$.

For $P, Q \in \F^G(P_0')$ such that $P \subset Q$, define
\begin{enumerate}
    \item[---]$\tau_P^Q \coloneqq [\rint A(\overline{\a_Q^+}, \overline{\a_P^+})]$. We write $\tau_P$ for $\tau_P^G$.
    \item[---]$\widehat{\tau}_P^Q \coloneqq [\rint A(\overline{\a_Q^+}, \overline{\a_P^+})^\vee]$. We write $\widehat{\tau}_P$ for $\widehat{\tau}_P^G$.
    \item[---] We write $\a^Q$ for the orthogonal complement
    of $\a_Q \subset \a_{0'}$.
    \item[---]$\varepsilon_P^Q \coloneqq (-1)^{\dim \a_P^Q}.$
    \item[---]$X_P, X^P, X_P^Q$, the projections of $X \in \a_{0'}$ onto $\a_P$, $\a^P$ and $\a_P^Q$ respectively.
\end{enumerate}

\subsubsection{Measures}

For each $P \in \F^G(A_0)$ we assign $\a_P$ the unique Haar measure for which the cocharacter lattice $X_*(A_P) \subset \a_P$ has covolume one. We define the measure on $A_P^{\infty}$ to be the unique measure which pushes forward to the measure on $\a_P$ under the isomorphism induced by $H_P : A_P^{\infty} \to \a_P$. We do the same for $G'$ by replacing $G$ with $G'$, $A_0$ with $A_0'$, $P$ with $P'$, etc.

We assign $G(\A)$ its canonical Tamagawa measure. This induces a measure on $[G]$ which we also call the Tamagawa measure. We assign $[G]^1$ the measure for which $$\int_{[G]^1} \int_{A_G^{\infty}} f(ay)\,da\,dy = \int_{[G]} f(x)\,dx$$ for any integrable $f : [G] \to \C$. With these definitions $[G]^1$ always has finite volume. Once again, we do the same with $G$ replaced by $G'$.

For all unipotent groups $N$, we fix a Haar measure on $N(\A)$ which assigns $[N]$ measure one. We choose the Haar measure on the maximal compacts $K \subset G(\A)$ and $K' \subset G'(\A)$ for which they are given measure one.

The choices above determine compatible measures on Levi subgroups $M_P(\A)$ such that $$\int_{P(F) \backslash G(\A)} f(g)\,dg = \int_{K} \int_{[M_P]^1} \int_{A_P^{\infty}} \int_{[N_P]} f(namk) e^{-2\rho_P(H_P(a))} \,dn\,da\,dm\,dk$$
for all integrable $f : P(F) \backslash G(\A) \to \C$. Here $\rho_P$ is the half-sum of positive weights for $A_P$ acting on $N_P$. Again, the same holds with $G$ replaced with $G'$, $P$ replaced with $P'$, etc.

\subsection{Defect and Twisted Orbit Representatives}

We follow the corresponding discussion in Lapid-Rogawski \cite[\S 4.4]{lapid-rogawski}. Recall that $X = G/G'$.

\begin{definition}[{\cite[Def. 4.4.1]{lapid-rogawski}}]
    Let $\mathscr{O}_{\xi} \subset X(F)$ be a $\theta$-twisted $G(F)$-orbit corresponding to $\xi \in \ker^1(F, X)$. The defect of $\mathscr{O}_{\xi}$, written $\mathscr{D}_G(\mathscr{O}_\xi)$, is the smallest Levi subgroup of a standard parabolic subgroup of $G$ which meets $\mathscr{O}_{\xi}$.
\end{definition} 

The crucial result we will use is

\begin{lemma}[{\cite[Lemma 4.4.1]{lapid-rogawski}}]\label{orbit-representatives}
    Let $L$ be a $\theta$-stable Levi subgroup in $G$. The map $$\mathscr{O} \mapsto \mathscr{O} \cap L$$ defines a bijection between $\theta$-twisted $G(F)$-orbits in $X(F)$ with defect contained in $L$ and $\theta$-twisted $L(F)$-orbits in $X_L(F)$. Moreover, $\mathscr{D}_L(\mathscr{O} \cap L) = \mathscr{D}_G(\mathscr{O})$.
\end{lemma}

This bijection allows us to make a particularly convenient choice of rational orbit representatives.

\begin{corollary}
        For each $\xi \in \ker^1(F, X)$ corresponding to a rational orbit $\mathscr{O}_\xi \subset X(F)$, there is a choice of representative $\eta_\xi \in \mathscr{O}_\xi$ such that $\eta_\xi \in X_{\mathscr{D}(\mathscr{O}_\xi)}$.
\end{corollary}
\begin{proof}
    By \cref{lemma-p}, every standard parabolic subgroup $P$ of $G$ is $\theta$-stable, as is the corresponding Levi $M_P$. Thus by applying \cref{orbit-representatives} above, for each $\xi \in \ker^1(F, X)$, we can choose a representative $\eta_\xi \in \mathscr{O}_{\xi}$ such that $\eta_\xi \in X_{\mathscr{D}(\mathscr{O}_\xi)}$.
\end{proof}

We say that such orbit representatives have minimal defect. Henceforth, we fix a choice of orbit representatives with minimal defect. When $\xi = 1$, we choose $\eta_1 = 1 \in G$. 

\subsection{Test Functions}\label{assigning-test-fns}

The final step before introducing the geometric and spectral kernels is to explain how to match functions on $X(\A)$ and $G(\A)$.

Fix $f \in \Cc(X(\A))$, and let $\zeta \in \im \alpha$ be a relevant class (see \cref{alpha} for the definition of $\alpha$). Define $f^{\zeta}$ to be the restriction of $f$ to the $G(\A)$-orbit labelled by $\zeta$. For each $\xi \in \alpha^{-1}(\zeta)$, let $\eta_\xi$ be the corresponding orbit representative with minimal defect. Then $G(\A) \star \eta_\xi$ is the orbit labelled by $\zeta$. Since the corresponding orbit map is a smooth surjection, there exists $\Phi^{\xi} \in \Cc(G(\A))$ such that $$f^{\zeta}(g \star \eta_\xi) = \int_{G_{\xi}(\A)} \Phi^{\xi}(gh)\,dh,$$ where $G_{\xi}$ is the the connected reductive stabilizer of $\eta_{\xi}$ with respect to $\theta$-twisted conjugation. Using reduction theory for $G_\xi$, there exists a factorization into commuting subgroups $G_\xi(\A) = G_\xi(\A)^1 A_\xi^{\infty}$.

% In fact, G_{\xi} are all pure inner forms of G', so if G' is semisimple, so are all of these groups, and thus A_{\xi} is trivial, and therefore A_{\xi^{\infty}} is trivial.

Define $$\Phi^{\xi}_1(z) \coloneqq \int_{A_{\xi}^{\infty}} \Phi^{\xi}(zs)\,ds,$$ so that the equality above can be rewritten $$f^{\zeta}(g \star \eta_\xi) = \int_{G_{\xi}(\A)^1} \Phi_1^{\xi}(gh)\,dh.$$ Finally, we view $\Phi^{\xi}_1$ as a function in $\Cc(G(\A)^1)$ by restriction. Thus, for each relevant class $\zeta$, we obtain a finite collection of test functions $\{ \Phi^{\xi}_1 \in \Cc(G(\A)^1) \}_{\xi \in \alpha^{-1}(\zeta)}$.

Note that for all but finitely many relevant $\zeta$, $f^{\zeta}$ is identically zero, as $f$ is compactly supported and each $\theta$-twisted $G(\A)$-orbit is open in $X(\A)$. Thus, whenever we sum such functions over all $\xi \in \ker^1(F, X)$ or all relevant classes in $\ker^1(\A, X)$, the sum is in fact finite. Second, note that if $f$ is factorizable, so is $f^{\zeta}$ and so are each of the functions $\Phi^{\xi}_1$.

% Remark: In fact, there is some additional flexibility in our definitions; as Zydor does, we can introduce a \lambda \in \a_{G_\xi, \C}^* and consider the same integral over A_\xi^{\infty} with an additional factor of e^{\lambda(H(s))}. Of course, when $\lambda = 0$, we recover the definition above. The point here is that there is a natural way of deforming these test functions in families indexed by the $\lambda$. This is useful at various stages, becaue one can view the integral over A_\xi^{\infty} as a kind of Mellin transform.

\subsection{Geometric and Spectral Kernels}

We are now in a position to introduce the geometric and spectral kernels. Again, fix $f \in \Cc(X(\A))$ as above.

\begin{definition}[Geometric Kernel]
    For each standard parabolic subgroup $P$ in $G$ with Levi decomposition $P = MN_P$ we define $$k_{f, P}(x) \coloneqq \sum_{\xi \in \ker^1(F, X_{M})} \sum_{\eta \in M(F) \star \eta_\xi} \int_{ N_P(\A)/N_{P, \eta_\xi}(\A) } f(x^{-1}\gamma n \eta_{\xi} \theta(n)^{-1} \theta(\gamma)^{-1} x)\,dn,$$ where for each $\eta \in M(F) \star \eta_{\xi}$, $\gamma \in M(F)$ is an element such that $\eta = \gamma\eta_\xi\theta(\gamma)^{-1}$.
\end{definition}

\begin{lemma}
    The geometric kernel is well-defined.
\end{lemma}
\begin{proof}
    We use the notation as above. We need to show that if $\gamma'$ is another element in $M_P(F)$ such that $\eta = \gamma'\eta_{\xi}\theta(\gamma')^{-1}$, then the inner-most integral above remains unchanged. Set $\nu \coloneqq \gamma^{-1}\gamma'$ and note that $\nu \in M_{P, \xi}(F)$, the $\theta$-twisted stabilizer of $\eta_{\xi} \in M_P(F)$. Now consider the change of variables $n \mapsto \nu n \nu^{-1}$. This clearly preserves $N_P(\A)$, and since $\nu \in M_{P, \xi}$ it also preserves $N_{P, \xi}(\A)$. Since $\nu$ is an $F$-rational point, the associated scalar by which the measure changes is one. Finally, note that the inner-most integrand above can be written $f(x^{-1}\gamma n \star \eta_\xi)$, and thus after the change of variables, we obtain $$f(x^{-1}\gamma \nu n \nu^{-1} \star \eta_\xi) = f( x^{-1} \gamma \nu n \star (\nu^{-1} \star \eta_{\xi})) = f(x^{-1} \gamma \nu n) = f(x^{-1} \gamma' n),$$ so we win.
\end{proof}

We also recall the usual spectral kernels:

\begin{definition}
    For $\Phi \in \Cc(G(\A)^1)$, and for each standard parabolic subgroup $P$ in $G$, define $$k_{\Phi, P}(x, y) \coloneqq \sum_{\gamma
\in M_{P}(F)}\int_{N_{P}(\A)} \Phi(x^{-1}\gamma n y) \,dn,$$
\end{definition}

We now explain the basic compatibility between these two kernels.

\begin{proposition}[Basic Compatibility]\label{basic-compat}
    Let $f \in \Cc(X(\A))$ and let $\{ \Phi_1^{\xi} \in \Cc(G(\A)^1) \}$ be the collection of test functions determined by $f$, where $\xi$ runs through $\ker^1(F, X)$. Then for all standard parabolic subgroups $P$ of $G$ with Levi decomposition $P = M N$, we have the equality $$\sum_{\xi \in \ker^1(F, X_M)} \int_{[G_{\xi}]^1} \sum_{\delta \in P_{\xi}(F) \backslash G_{\xi}(F)} k_{\Phi_1^{\xi}, P}(x, \delta y)\,dy =
    k_{f, P}(x),$$ for all $x \in G'(\A)^1$. Here the notation $H_\xi$ indicates the subgroup of $H$ which fixes $\eta_\xi$ under $\theta$-twisted conjugation, and $P_{\xi} = M_{\xi} N_{\xi}$ is a corresponding parabolic subgroup of $G_{\xi}$.
\end{proposition}
\begin{proof}
    Fix $\xi \in \ker^1(F, X_M)$. Then by our choice of orbit representatives with minimal defect, we have fixed a rational representative $\eta_\xi$ for the orbit corresponding to $\xi$ such that $\eta_\xi \in X_M(F)$ (of course $\eta_\xi$ could potentially come from an even smaller Levi).

    Thus, for $x \in G'(\A)$, we calculate
    \begin{align*}
        \sum_{\eta \in M(F) \star \eta_{\xi}} \int_{N(\A)/N_{\xi}(\A)} f(x^{-1} \gamma n \eta_\xi \theta(n)^{-1} \theta(\gamma)^{-1} x)\,dn &= \\
        \sum_{\gamma \in M(F)/M_{\xi}(F)} \int_{N(\A)/N_{\xi}(\A)} f(x^{-1} \gamma n \eta_\xi \theta(n)^{-1} \theta(\gamma)^{-1} x)\,dn &=  \\
        \sum_{\gamma \in M(F)/M_{\xi}(F)} \int_{N(\A)/N_{\xi}(\A)} \int_{G_\xi(\A)^1} \Phi_1^{\xi}(x^{-1} \gamma n y)\,dy\,dn
    \end{align*}

    We now unfold the inner-most integral along the unimodular subgroup $M_\xi(F) N_{\xi}(\A) \subset G_\xi(\A)^1$ to obtain $$\int_{M_\xi(F) N_{\xi}(\A) \backslash G_{\xi}(\A)^1 } \sum_{\gamma' \in M_{\xi}(F)} \int_{N_{\xi}(\A)} \Phi_1^{\xi}(x^{-1} \gamma \gamma' n n' y)\,dn'\,dy,$$ where we are free to commute $\gamma'$ and $n$ as $\gamma' \in M_\xi(F) \subset M(F)$. Inserting this integral above and folding, we obtain $$\sum_{\gamma \in M(F)} \int_{N(\A)} \int_{M_\xi(F)N_{\xi}(\A) \backslash G_\xi(\A)^1} \Phi^{\xi}_1(x^{-1}\gamma n y)\,dy \, dn.$$ Now one sees immediately that the inner-most integral is invariant under the change of variables $y \mapsto \overline{n} y$ for $\overline{n} \in N_\xi(F) \backslash N_\xi(\A)$. Moreover, this quotient can be identified with $P_\xi(F) \backslash M_\xi(F)N_\xi(\A)$ in such a way that measures are preserved. Folding along $P_\xi(F) \subset M_\xi(F)N_\xi(\A) \subset G_\xi(\A)^1$, we obtain $$\sum_{\gamma \in M(F)} \int_{N(\A)} \int_{P_{\xi}(F) \backslash G_\xi(\A)^1} \Phi^{\xi}_1(x^{-1}\gamma n y)\,dy\,dn.$$ Finally, unfolding along $P_\xi(F) \subset G_\xi(F) \subset G_\xi(\A)^1$, and rearranging we have shown $$\sum_{\eta \in M(F) \star \eta_{\xi}} \int_{N(\A)/N_{\xi}(\A)} f(x^{-1} \gamma n \eta_\xi \theta(n)^{-1} \theta(\gamma)^{-1} x)\,dn = \int_{[G_\xi]^1} \sum_{\delta \in P_\xi(F) \backslash G_\xi(F)} k_{\Phi^1_{\xi}, P}(x, \delta y)\,dy.$$ Summing over $\xi \in \ker^1(F, X_M)$ yields the desired identity.

    We have implicitly used our assumption \labelcref{assumption-p} to ensure that $P_\xi$ is a standard parabolic subgroup of $G_\xi$ with Levi decomposition $P_\xi = M_\xi N_\xi$. Indeed, one can define a new involution $\theta_\xi(g) \coloneqq \eta_{\xi}\theta(g)\eta_{\xi}^{-1}$ whose fixed locus in $G$ is precisely $G^{\theta_\xi} = G_\xi$ and which is also a Galois involution (see \cref{twisted-stabilizers-give-galois-pairs}). The assumption \labelcref{assumption-p} along with our judicious choice of $\eta_\xi$ ensure that $P$ is $\theta_\xi$-stable. Thus $P_\xi = P^{\theta_\xi} = G^{\theta_\xi} \cap P$ is clearly a standard parabolic with the stated Levi decomposition.
\end{proof}

\subsection{Truncated Kernels}\label{truncated-kernels}

With this preparation, we can now discuss the mechanics of modifying the spectral and geometric kernels.

The first step is to introduce refined kernels $$\{ k_{P, f, \o} \}_{\o \in \O^X} \text{ and } \{ k_{P, \Phi, \chi} \}_{\chi \in \mathcal{X}^G},$$ where $\{ \o \in \O^X \}$ denotes the collection of geometric data associated to $X$ and $\{ \chi \in \mathcal{X}^G \}$ is the collection of cuspidal data associated to $G$.

Cuspidal data $\{ \chi \in \mathcal{X}^G \}$ are defined in \cite[\S II.1]{moeglin-waldspurger}. They index the summands of the coarse Langlands decomposition of $$L^2([G]^1) = \hat{\bigoplus}_{\chi \in \mathcal{X}^G} L^2_{\chi}([G]^1).$$ In particular, for each cuspidal datum we obtain a projector onto the corresponding summand of $L^2([G]^1)$. Decomposing the right regular action of $\Phi \in \Cc(G(\A)^1)$ along these projectors, we obtain a decomposition of the parabolic spectral kernels $$k_{\Phi, P} = \sum_{\chi \in \mathcal{X}^G} k_{\Phi, P, \chi}.$$

We now introduce geometric data. Let $\chi : X \to X/\!/G'$ denote the canonical GIT quotient. Over $\overline{F}$, the space $(X/\!/G')_{\overline{F}} \simeq X_{\overline{F}}/\!/G'_{\overline{F}}$ can be understood as the moduli of closed $G'(\overline{F})$-orbits in $X(\overline{F})$; two points in $X_{\overline{F}}$ have the same image under $\chi_{\overline{F}}$ if and only if the closures of the $G'_{\overline{F}}$-orbits through these points meet. In particular, over $\overline{F}$, the fibers of $\chi_{\overline{F}}$ contain a unique closed $G'_{\overline{F}}$-orbit.

\begin{definition}[Geometric Data]
    Define the set of geometric data $\O^X$ as the collection of $F$-points of closed subschemes of $X$ obtained as fibers of the map $\chi : X \to X/\!/G'$ above $F$-rational points.
\end{definition}

In particular, we have a decomposition $$X(F) = \bigsqcup_{\o \in \O^X} \o.$$ This allows us to define the refined geometric kernels:

\begin{definition}[Refined Geometric Kernel]
    For each geometric datum $\o \in \O^X$, and each standard parabolic subgroup $P$ in $G$ with Levi decomposition $P = MN_P$ we define $$k_{f, P, \o}(x) \coloneqq \sum_{\xi \in \ker^1(F, X_{M})} \sum_{\eta \in (M(F) \star \eta_\xi \cap \o)} \int_{ N_P(\A)/N_{P, \eta_\xi}(\A) } f(x^{-1}\gamma n \eta_{\xi} \theta(n)^{-1} \theta(\gamma)^{-1} x)\,dn,$$ where for each $\eta \in M(F) \star \eta_{\xi}$, $\gamma \in M(F)$ is an element such that $\eta = \gamma\eta_\xi\theta(\gamma)^{-1}$.
\end{definition}

In particular, we see that for all standard parabolic subgroups $P$, $$k_{f, P} = \sum_{\o \in \O^X} k_{f, P, \o}.$$

Informally, the next step is to perform a mixed truncation along $G' \subset G$ in the $x$ variable.

\begin{definition}
    Let $T \in \a_{0'}$ be the truncation parameter. Let $f \in \Cc(X(\A))$. Define for each $\o \in \O^X$ the truncated geometric kernel $$k^T_{f, \o}(x) \coloneqq \sum_{P \in \F^G(P_0')} \varepsilon_P^G \sum_{\delta_1 \in P'(F) \backslash G'(F)} \widehat{\tau}_P(H_{0'}(\delta_1 x)_P - T_P) k_{f, P, \o}(\delta_1 x).$$ Similarly, let $\Phi \in \Cc(G(\A)^1)$, and define for each $\xi \in \ker^1(F, X)$ and each cuspidal datum $\chi \in \mathcal{X}^G$ the truncated spectral kernel $$k^T_{\Phi, \xi, \chi}(x, y) \coloneqq \sum_{P \in \F^G(P_0')} \varepsilon_P^G \sum_{\substack{\delta_1 \in P'(F) \backslash G'(F)\\ \delta_2 \in P_{\xi}(F) \backslash G_{\xi}(F)}}\widehat{\tau}_{P}(H_{0'}(\delta_1 x)_P - T_P) k_{\Phi, P, \chi}(\delta_1 x, \delta_2 y).$$
\end{definition}

\subsection{The Coarse RTF}\label{coarse-rtf}

Again let $T \in \a_{0'}$ be a truncation parameter. We define the coarse truncated geometric RTF distribution via $$J_{\o}^T(f) \coloneqq \int_{[G']^1} k^T_{f, \o}(x)\,dx, \text{ for all $f \in \Cc(X(\A))$.}$$ Likewise, for $\Phi \in \Cc(G(\A)^1)$ and for $\xi \in \ker^1(F, X)$, we define the $\xi$-part of the coarse truncated spectral RTF distribution via $$J_{\xi, \chi}^T(\Phi) \coloneqq \int_{[G']^1} \int_{[G_\xi]^1} k^T_{\Phi, \xi, \chi}(x, y)\,dy\,dx.$$ Moreover, define $J^T_{geom}(f) \coloneqq \sum_{\o \in \O^X} J^T_{\o}(f)$, and $J^T_{spec,\,\xi}(\Phi) \coloneqq \sum_{\chi \in \mathcal{X}^G} J^T_{\xi, \chi}(\Phi)$.

We can now formulate precisely the expected form of the coarse relative trace formula for Galois symmetric pairs.

\begin{proposition}[Coarse RTF for Galois Symmetric Spaces]
    Let $(G, G')$ be a Galois symmetric pair of equal split rank, and let $X$ be the associated Galois symmetric space.
    
    Let $f \in \Cc(X(\A))$ and let $\{ \Phi_1^{\xi} \}$ be the associated collection of test functions in $\Cc(G(\A)^1)$ (see \cref{assigning-test-fns}).
    
    Assume that for $T \in \a_{0'}$ sufficiently regular, $$\sum_{\o \in \O^X} \int_{[G']^1} |k^T_{f, \o}(x)|\,dx < \infty \text{ and } \sum_{\xi \in \ker^1(F, X)} \sum_{\chi \in \mathcal{X}^G} \int_{[G']^1} \int_{[G_\xi]^1} |k^T_{\Phi^{\xi}_1, \xi, \chi}(x, y)|\,dy\,dx < \infty.$$ Then we have the following equality of distributions $$J^T_{geom}(f) = \sum_{\zeta \in \im \alpha} \sum_{\xi \in \alpha^{-1}(\zeta)} J^T_{spec,\,\xi}(\Phi^{\xi}_1) = \sum_{\xi \in \ker^1(F, X)} J^T_{spec,\,\xi}(\Phi^{\xi}_1).$$
\end{proposition}
\begin{proof}
    This is an entirely formal calculation utilizing the definitions above and \cref{basic-compat}.
\end{proof}

Note that all but finitely many contributions of the $\zeta \in \im \alpha$ on the right-hand side are zero, and for each fixed $\zeta \in \im \alpha$, the sum over $\alpha^{-1}(\zeta)$ is finite by \cref{fibers-of-alpha-are-finite}. Thus, only finitely many $\xi$ contribute on the right-hand side.

\part{The Galois Symmetric Pair for $\SL_2$}\label{part-two}

The remainder of this paper is devoted to establishing the fine geometric RTF for the Galois symmetric pair $(\Res_{E/F} \SL_{2, E}, \SL_{2, F})$.

Let us fix some notation. Set $G = \Res_{E/F} \SL_{2, E}$ and $G' = \SL_{2, F}$. Let $A_0' = A_0$ be the standard split maximal torus in $G$, which is manifestly $\theta$-stable. Fix $P_0 = B$ the standard upper-triangular Borel subgroup in $G$, and let $P_0' = B'$ be the standard upper-triangular Borel in $G'$. These will be our fixed minimal parabolic subgroups.

In this setting, the Galois cohomological complications discussed above do not enter into the picture; $\ker^1(F, X)$ and $\ker^1(\A, X)$ are both trivial, and the non-trivial Levi subgroups of $G'$ (and $G$) have trivial Galois cohomology (they are quasi-split tori).

We write $M = M_B$ for the standard Levi subgroup of $B$, and $M' = M_{B'}$ for the standard Levi subgroup of $B'$.

Note that the Weyl groups of $G$ and $G'$ agree. Fix Weyl group representatives $1, w = \mat{0}{1}{-1}{0} \in G'(F)$.

Finally, note that the Harish-Chandra maps for $G$ and $G'$ are related by $H_{0}(x) = 2H_{0'}(x)$ for $x \in G'(\A)$. This is a consequence of the fact that $|a|_{\A_E} = |a|_{\A_F}^2$ for all $a \in \A^{\x}_F$.

At various points in the proof we will use the properties of certain explicit indicator function; these have been written out in \cref{explicit-trunc-calcs}.

\section{Convergence of the Truncated Geometric RTF}

Our goal in this section is to prove the following

\begin{theorem}\label{geom-cvg}
    Let $\o \in \O^X$ be a fixed geometric datum, and let $f \in \Cc(X(\A))$. Then for $T \in \a_{0'}$ sufficiently large, $$\int_{[G']} |k^T_{f, \o}(x)|\,dx < \infty.$$ Thus, $J_\o^T(f)$ is well-defined for $T \in \a_{0'}$ sufficiently regular.
\end{theorem}

That this theorem implies the convergence of $J_{geom}^T(f)$ is explained in the paragraph following \cref{refined-kernel-for-sl2}.

We begin by defining the notion of geometric datum and the refined geometric kernels $\{ k_{f, P, \o} \}_{P \in \F^G(B')}$. We then introduce a version of the Cayley transform adapted to our situation, and describe its basic properties.

The proof begins in earnest in \cref{geom-cvg-proof}, and breaks down into three steps. It follows the basic outline put forward by Arthur \cite{arthur-1, arthur-intro}, adapted to the Galois symmetric space setting. See also \cite{coarse}.

The first step is to use a partition of unity to rewrite $k^T_{f, \o}(x)$ as a sum of indicator functions multiplied by signed sums of parabolic kernels. The basic idea is to obtain cancellation within these signed sums, and thereby control the integral.

The next step introduces a geometric lemma, a variant of the Levi decomposition for $X$ that we call a Levi retraction, which in combination with an analysis of the relevant indicator functions, is used to show that the corresponding signed sum of parabolic kernels can be rewritten in terms of the Fourier transform of an auxiliary function on a line bundle above the Levi Galois symmetric space.

In the final step, the cancellation that was sought is achieved by applying the Poisson summation formula.

\subsection{Invariant Theory, Geometric Data}\label{geometric-data}

In this section, following the standard conventions in algebraic geometry we write $\A^1$ to denote the affine line, i.e. $\A^1 = \Spec F[T]$.

If we fix a presentation $E = F(\tau^{1/2})$ for $\tau^{1/2} \in E$ an element of trace zero, we obtain using \cref{concrete-description-of-symm-spc} the following explicit presentation for the Galois symmetric space $$X = \{ \Mat{a}{b\tau^{1/2}}{c\tau^{1/2}}{\overline{a}} \in G : a \in E, b, c \in F, a\overline{a} - bc\tau = 1 \}.$$

We now recall some basic definitions from geometric invariant theory.

\begin{definition}
    Fix $\eta \in X(F)$.
    \begin{itemize}
        \item We say $\eta$ is regular if the dimension of its centralizer under the $G'$-action is minimal.
        \item We say $\eta$ is semisimple if the $G'$-orbit through $\eta$ is closed.
    \end{itemize}
\end{definition}

Write $X^{rss}$ for set the $F$-rational points of $X$ which are regular semisimple.

\begin{proposition}
    The map $\chi : X \to \A^1$ given by $$\eta \mapsto \frac{1}{2}\tr(\eta)$$ realizes $\A^1$ as the GIT quotient of $X$ by $G'$. It is surjective on $F$-points.
\end{proposition}
\begin{proof}
    Recall that by constuction, the underlying ring of the GIT quotient is obtained as the equalizer of the standard inclusion $\O(X) \to \O(G') \otimes_F \O(X)$ (dual to the projection map) and the action map $\O(X) \to \O(G') \otimes_F \O(X)$. Thus, the formation of GIT quotients commutes with arbitrary field extension. After base change to $\overline{F}$, by \cref{base-change-splits}, we can identify $X_{\overline{F}} \simeq G'_{\overline{F}}$, in such a way that the $G'$-action is identified with conjugation and the map above is the trace map. But now the result is well-known.
\end{proof}

We say a point $a \in \A^1(F)$ is regular (resp. semisimple) if all the $F$-rational points in $\chi^{-1}(a)$ are regular (resp. semisimple). In particular, note that for all regular semisimple $a \in \A^1(F)$, the fiber above $a$ is a single $G'(\overline{F})$-orbit. We denote the set of $F$-rational points in $\A^1$ which are regular semisimple points by $\A^{1, rss}$.

It will often be useful to identify $a \in \A^1(F)$ with the polynomial $p_a(T) = T^2 - 2aT + 1 \in F[T]$.

\begin{definition}[Geometric Data]
    Define the collection of geometric data, written $\O^{X}$, to be the collection of $F$-points of closed subschemes of $X$ obtained as fibers of the map $\chi : X \to \A^1$ above $F$-rational points.
\end{definition}

If $\o \in \O^X$ is a geometric datum, we write $p_{\o}(T) \in F[T]$ for the polynomial $p_{\chi(\o)}$.

One can show that $\A^{1, rss} = F \setminus \{ \pm 1 \}$. Indeed, the discriminant of $p_a(T)$ is given by $4a^2 - 4$. A straight-forward calculation shows that above $\pm 1 \in \A^1(F)$, the fiber is naturally identified with the nilpotent cone of the Lie algebra $\mathfrak{sl}_{2, F}$. Thus, over $\overline{F}$, the fiber above $\pm 1$ is a union of two $G'(\overline{F})$-orbits; the zero orbit, and the principal orbit. We will make this more precise later, when we come to the fine geometric expansion.

\subsection{Refined Geometric Kernels}

\begin{definition}\label{refined-kernel-for-sl2}
    Let $\o \in \O^X$ be a geometric datum. Let $f \in \Cc(X(\A))$. For each standard parabolic subgroup $P = MN$ in $G$, define $$k_{f, P, \o}(x) \coloneqq \sum_{\eta \in X_M \cap \o} \int_{N_P(\A)/N_{P'}(\A)} f(x^{-1} \gamma n \theta(n)^{-1} \theta(\gamma)^{-1} x)\,dn,$$ where $\gamma \in M_P(F)$ is an element such that $\eta = \gamma\theta(\gamma)^{-1}$.
\end{definition}

We have chosen to incorporate the simplifications which occur in the $(\Res_{E/F} \SL_{2, E}, \SL_{2, F})$ case into the definition above.

Note that \cref{geom-cvg} implies $$\sum_{\o \in \O^X} \int_{[G']} |k^T_{f, \o}(x)|\,dx < \infty.$$ Indeed since $f \in \Cc(X(\A))$ is compactly supported, $\chi(\supp f) \subset (X/\!/G')(\A)$ is compact, and thus meets the discrete set $(X/\!/G')(F) \subset (X/\!/G')(\A)$ in finitely many rational points. Thus, for all but finitely many $\o \in \O^X$, the kernel $k_{f, G, \o}$ is identically zero. A similar argument shows that for all but finitely many $\o \in \O^X$, all kernels $\{ k_{f, P, \o} \}$ are zero. Thus $k^T_{f, \o}$ is zero for all but finitely many geometric datum $\o \in \O^X$.

\subsection{The Cayley Transform}\label{cayley-section}

Let $\s \coloneqq \{ X \in \mathfrak{sl}_{2, E} : X + \overline{X}  = 0 \}$ be the normal slice to $1 \in X$. It admits a natural adjoint action by $G'$. For $E = F(\tau^{1/2})$, we can describe $\s$ explicitly as $$\s = \{ \tau^{1/2}\mat{a}{b}{c}{-a} \in \mathfrak{sl}_{2, E} : a, b, c \in F \}.$$

Write $\m_B \coloneqq \Lie(M_B)$, $\n_B \coloneqq \Lie(N_B)$. For $\nu \in E^{\x}$, define $$D_{\nu} \coloneqq \{ X \in \mathfrak{s} : \det(1 - \nu X) = 0 \} \subset \mathfrak{s}.$$ Similarly, define $$D_{\tau^{1/2}}' \coloneqq \{ X \in \mathfrak{sl}_{2, F} : \det(1 - \tau^{1/2} X) = 0 \} \subset \mathfrak{sl}_{2, F}.$$

Let $N_B^{opp} \subset N_B$ denote the subgroup which is sent to its negation under the Galois involution. It is clear that we can identify $N_B^{opp}(\A)$ with $N_B(\A)/N_{B'}(\A)$, and we frequently do so without further comment.

\begin{definition}
     For $\varepsilon = \pm 1$, the Cayley transform $\kappa_{\varepsilon} : \s \setminus D_1 \to G$ is given by $$\kappa_{\varepsilon}(Y) \coloneqq -\varepsilon(1 + Y)(1 - Y)^{-1}.$$
\end{definition}

\begin{proposition} For $\varepsilon = \pm 1$, we have:
    \begin{enumerate}
        \item The map $\kappa_{\varepsilon}$ is well-defined and factors through $X \subset G$.
        \item The map $\kappa_{\varepsilon}$ induces an isomorphism $\kappa_{\varepsilon} : \s \setminus D_1 \to X \setminus D_{\varepsilon}$, with inverse $$ \kappa_{\varepsilon}^{-1} : g \mapsto -(\varepsilon + g)(\varepsilon - g)^{-1}.$$ Here $D_{\varepsilon} = \{ g \in X : \det(\varepsilon - g) = 0 \} \subset X$.
        \item The diagram below commutes:
        \begin{center}
            \begin{tikzcd}
            \s \setminus D_1 \arrow[r, "\kappa_{\varepsilon}"] \arrow[d, "-\det"'] & X \setminus D_{\varepsilon} \arrow[d, "\chi"] \\
            \A^1_F \setminus \{ 1 \} \arrow[r]                                & \A^1_F \setminus \{ \varepsilon \}
            \end{tikzcd}
        \end{center}
        The bottom map is the Cayley transform given by the same formula as $\kappa_{\varepsilon}$, except $Y$ is interpreted as an element of $\A^1_F$.
    \end{enumerate}
\end{proposition}
\begin{proof}
    Fix $Y = \tau^{1/2}\mat{a}{b}{c}{-a} \in \mathfrak{s} \setminus D_1$. The condition that $\det(1 - Y) \neq 0$ translates to $\tau (a^2 +  bc) \neq 1$. We calculate $$\kappa_{\varepsilon}(Y) = \frac{-\varepsilon}{1 - \tau a^2 - \tau bc}\Mat{(1 + \tau^{1/2}a)^2 + \tau bc}{2\tau^{1/2} b}{2\tau^{1/2} c}{(1 - \tau^{1/2}a)^2 + \tau bc}.$$ A further calculation reveals $\det(\kappa_\varepsilon(Y)) = \varepsilon^2 = 1$. Thus, $\kappa_\varepsilon$ is well-defined as a map into $G$. Based on our description of $X$ in coordinates, we see immediately that this factors through $X \subset G$. This proves (1).

    The proof of (2) is a routine calculation we leave to the reader.

    A further calculation shows that $$\chi(\kappa_\varepsilon(Y)) = -\varepsilon \frac{1 + \tau(a^2 + bc)}{1 - \tau(a^2 + bc)}.$$ Since $-\det(Y) = \tau(a^2 + bc)$, this verifies (4).    
\end{proof}

\begin{lemma}\label{cayley-is-equivariant}
    For $Y \in (X \setminus D_{\varepsilon})(\A)$ and $x \in G'(\A)$, $$g\kappa_{\varepsilon}(Y)g^{-1} = \kappa_{\varepsilon}(\Ad(g)Y).$$
\end{lemma}
\begin{proof}
    By definition, $\kappa_{\varepsilon}(Y) = -\varepsilon(1 + Y)(1 - Y)^{-1}.$ The claim is immediate.
\end{proof}

\begin{proposition}\label{cayley-map-properties}
    Let $\varepsilon \in \{ \pm 1 \}$, $\eta \in (X_M \setminus D_{\varepsilon})(F)$. Set $\tau^{1/2} \zeta \coloneqq \kappa_{\varepsilon}^{-1}(\eta)$. Finally, choose $\gamma \in M_B(F)$ such that $\eta = \gamma\theta(\gamma)^{-1}$.

    \begin{enumerate}
        \item $\zeta \in (\m_{B'} \setminus D_{\tau^{1/2}}')(F)$.
        \item For all $n \in N_B(\A)$, $\gamma n \theta(n)^{-1} \theta(\gamma)^{-1} \in (X \setminus D_{\varepsilon})(\A)$.
        \item For all $n \in N_B(\A)$, there exists a unique $N \in \n_{B'}(\A)$ such that $$\kappa_{\varepsilon}^{-1}(\gamma n \theta(n)^{-1} \theta(\gamma)^{-1}) = \tau^{1/2}(\zeta + N).$$
        \item The map $n \mapsto N$ above induces a homeomorphism $N_B(\A) / N_{B'}(\A) \to \n_{B'}(\A)$ which preserves Haar measure. Thus, for all $\Phi \in \Cc(X(\A))$, $$\int_{N_B(\A) / N_{B'}(\A)} \Phi(\gamma n \theta(n)^{-1} \theta(\gamma)^{-1})\,dn = \int_{\n_{B'}(\A)} \Phi(\kappa_{\varepsilon}(\tau^{1/2}(\zeta + N)))\,dN$$ and the integral on the left is independent of the choice of $\gamma\in M_B(F)$.
    \end{enumerate}
\end{proposition}

\begin{proof}
    From our explicit description of points in $X$, we can write $\eta = \mat{x}{}{}{\overline{x}}$ for $x \in E^{\x}$ and $x\overline{x} = 1$. Since $\eta \not\in D_{\varepsilon}$, $\det(\varepsilon - \eta) \neq 0$. Thus, $\kappa_{\varepsilon}^{-1}$ is well-defined, and we calculate $$\kappa_{\varepsilon}^{-1}(\eta) = -(\varepsilon + \eta)(\varepsilon - \eta)^{-1} = -\Mat{\frac{\varepsilon + x}{\varepsilon - x}}{}{}{\frac{\varepsilon + \overline{x}}{\varepsilon - \overline{x}}}.$$ We calculate the trace: $$-\left(\frac{\varepsilon + x}{\varepsilon - x} + \frac{\varepsilon + \overline{x}}{\varepsilon - \overline{x}}\right) = -\frac{(\varepsilon^2 - \varepsilon\overline{x} + \varepsilon x - 1) + (\varepsilon^2 + \varepsilon \overline{x} - \varepsilon x - 1)}{(\varepsilon - x)(\varepsilon + \overline{x})} = 0.$$ This matrix is clearly sent to its negation under the Galois involution, thus $\zeta = \frac{1}{\tau^{1/2}}\kappa_{\varepsilon}^{-1}(\eta)$ is preserved by the Galois involution, and lies in $\mathfrak{m}_{B'}(F)$. Since $\kappa_{\varepsilon}^{-1}(\eta) \not\in D_1$, $\zeta \not\in D_{\tau^{1/2}}'$. This proves (1).

    Let $y \in E^{\x}$ such that $x = y/\overline{y}$, and set $\gamma = \mat{y}{}{}{y^{-1}}$. Clearly, $\gamma n \theta(n)^{-1}\theta(\gamma)^{-1} \in X(\A)$. To check that this does not lie in $D_\varepsilon$, we must show that $\det(\varepsilon - \gamma n \theta(n)^{-1}\theta(\gamma)^{-1}) \neq 0$. Write $n = \mat{1}{a}{}{1}$ for some $a \in \A_E$. Then we calculate $$\gamma n \theta(n)^{-1}\theta(\gamma)^{-1} = \Mat{x}{y\overline{y}(a - \overline{a})}{}{\overline{x}}.$$ It follows that $0 \neq \det(\varepsilon - \eta) = \det(\varepsilon - \gamma n \theta(n)^{-1}\theta(\gamma)^{-1})$, as desired. This proves (2).

    From (1) and (2), the quantity $\kappa_{\varepsilon}^{-1}(\gamma n \theta(n)^{-1}\theta(\gamma)^{-1}) \in \mathfrak{s} \setminus D_1$ is well-defined. We are now searching for $N \in \mathfrak{n}_{B'}(\A)$ such that $$\kappa_{\varepsilon}^{-1}(\gamma n \theta(n)^{-1}\theta(\gamma)^{-1}) = \tau^{1/2}(\zeta + N).$$ Since $\kappa_{\varepsilon}^{-1}$ is invertible on this locus, if such an $N \in \mathfrak{n}_{B'}(\A)$ exists, it is unique. Without loss of generality, we may assume $n \in N_B^{opp}(\A)$. Thus, write $n = \mat{1}{\tau^{1/2} a/2y\overline{y}}{}{1}$ for some $a \in \A_F$. Then we calculate $$\gamma n \theta(n)^{-1}\theta(\gamma)^{-1} = \eta \Mat{1}{\tau^{1/2} a/x}{}{1}.$$ Now using our formula for $\kappa_{\varepsilon}^{-1}$, we calculate \begin{multline*}
        \kappa_{\varepsilon}^{-1}(\eta \Mat{1}{\tau^{1/2} a/x}{}{1}) = -(\varepsilon + \eta \Mat{1}{\tau^{1/2} a/x}{}{1})(\varepsilon - \eta \Mat{1}{\tau^{1/2} a/x}{}{1})^{-1} \\ = -\Mat{\varepsilon + x}{\tau^{1/2} a}{}{\varepsilon + \overline{x}}\Mat{\varepsilon - x}{-\tau^{1/2} a}{}{\varepsilon - \overline{x}}^{-1} = -\Mat{\varepsilon + x}{\tau^{1/2} a}{}{\varepsilon + \overline{x}}\Mat{\frac{1}{\varepsilon - x}}{\frac{\tau^{1/2} a}{(\varepsilon - x)(\varepsilon - \overline{x})}}{}{\frac{1}{\varepsilon - \overline{x}}}.
    \end{multline*}
    Finally, let us factor $$\Mat{\varepsilon + x}{\tau^{1/2} a}{}{\varepsilon + \overline{x}} = \Mat{\varepsilon + x}{}{}{\varepsilon + \overline{x}} + \Mat{0}{\tau^{1/2} a}{}{0},$$ and
    $$\Mat{\frac{1}{\varepsilon - x}}{\frac{\tau^{1/2} a}{(\varepsilon - x)(\varepsilon - \overline{x})}}{}{\frac{1}{\varepsilon - \overline{x}}} = \Mat{\frac{1}{\varepsilon - x}}{}{}{\frac{1}{\varepsilon - \overline{x}}} + \Mat{0}{\frac{\tau^{1/2} a}{(\varepsilon - x)(\varepsilon - \overline{x})}}{}{0}.$$
    Then we obtain $$\kappa_{\varepsilon}^{-1}(\eta \Mat{1}{\tau^{1/2} a/x}{}{1}) = \kappa_{\varepsilon}^{-1}(\eta) - \tau^{1/2} \Mat{0}{\frac{a(\varepsilon + x)}{(\varepsilon - x)(\varepsilon - \overline{x})} + \frac{a}{\varepsilon - \overline{x}}}{}{0} = \tau^{1/2}(\zeta + N),$$ where $$N \coloneqq \Mat{0}{\frac{-2a\varepsilon}{(\varepsilon - x)(\varepsilon - \overline{x})}}{}{0}.$$ This completes the proof of (3). The final claim is a straightforward consequence of the explicit calculations performed above. The point is simply that $N$ depends \emph{linearly} on the coordinate $a$ parameterizing $n$.
\end{proof}

\subsection{Proof of \Cref{geom-cvg}}\label{geom-cvg-proof}

\subsubsection{Step 1: Partition of Unity}\label{partition-of-unity}

We now begin the proof of \cref{geom-cvg}. The first step is to apply a partition of unity result to decompose the regularized geometric kernel.

This partition of unity utilizes a family of indicator functions $\{ F^P \}_{P \in \mathcal{F}^G(B')}$ on $\a_{0'}$ whose definition is unfortunately quite technical. To give some flavour, when $P = G$, $F^G(x, T)$ is an indicator function on $x \in G'(\A)$ which is one when $x$ lies in a $G'(F)$-translate of a compact set depending on the truncation parameter $T$. We refer the reader to \cite[\S3.5, Paragraph 2]{zydor} for the definition.

We also make use of a family of indicator functions $\{ \sigma^{P_2}_{P_1} \}$ for various parabolic subgroups $P_1, P_2$. These are defined and computed explicitly in \cref{sigma-calculation} of the appendix.

\begin{proposition}[{\cite[Prop. 3.5]{coarse}}]
    For all $x \in G'(\A)$ and all $T \in \a_{0'}$ sufficiently regular, and all $Q \in \F^G(B')$, we have $$\sum_{\substack{P \in \F^G(B')\\P \subset Q}} \sum_{\delta \in P'(F) \backslash Q'(F)}F^P(\delta x, T) \tau_P^Q(H_{0'}(\delta x)_P - T_P) = 1.$$
\end{proposition}

Recall that $$k^T_{f, \o}(x) \coloneqq \sum_{P \in \F^G(B')} \varepsilon_P^G \sum_{\delta \in P'(F) \backslash G'(F)} \widehat{\tau}_P(H_{0'}(\delta x)_P - T_P) k_{f, P, \o}(\delta x).$$

Applying the above identity (for $\delta x$) in the definition of $k^T_{f, \o}$ we obtain
\begin{multline*}
    k^T_{f, \o}(x) = \sum_{P \in \F^G(B')} \varepsilon_P^G \\\sum_{\delta \in P'(F) \backslash G'(F)} \sum_{P_1 \subset P} \sum_{\delta_1 \in P_1'(F) \backslash P'(F)} F^{P_1}(\delta_1\delta x, T)\tau_{P_1}^P(H_{0'}(\delta_1 \delta x)_{P_1} - T_{P_1})\widehat{\tau}_P(H_{0'}(\delta x)_P - T_P) k_{f, P, \o}(\delta x).
\end{multline*}

Now, note that $k_{f, P, \o}(\delta_1\delta x) = k_{f, P, \o}(\delta x)$ and $\widehat{\tau}_P(H_{0'}(\delta_1\delta x) - T_P) = \widehat{\tau}_P(H_{0'}(\delta x) - T_P)$. Thus, we can fold the sum along $P_1'(F) \subset P'(F) \subset G'(F)$ to obtain
\begin{multline*}
    k^T_{f, \o}(x) = \sum_{P \in \F^G(B')} \varepsilon_P^G \\ \sum_{P_1 \subset P} \sum_{\delta \in P_1'(F) \backslash G'(F)} F^{P_1}(\delta x, T)\tau_{P_1}^P(H_{0'}(\delta x)_{P_1} - T_{P_1})\widehat{\tau}_P(H_{0'}(\delta x)_P - T_P) k_{f, P, \o}(\delta x).
\end{multline*}

We now use the contraction relation (\cite[Lemma 3.3, (4)]{coarse}) $\tau_{P_1}^P\widehat{\tau}_P = \sum_{\substack{P_2 \in \F^G(B')\\P_2 \supset P}} \sigma_{P_1}^{P_2}$ to write this as 
\begin{multline*}
    k^T_{f, \o}(x) = \sum_{P \in \F^G(B')} \varepsilon_P^G \\ \sum_{P_1 \subset P} \sum_{\delta \in P_1'(F) \backslash G'(F)} F^{P_1}(\delta x, T)\left(\sum_{\substack{P_2 \in \F^G(B')\\P_2 \supset P}} \sigma_{P_1}^{P_2}(H_{0'}(\delta x)_{P_1} - T_{P_1})\right) k_{f, P, \o}(\delta x).
\end{multline*}

Finally, we can collect the terms which depend explicitly on $P$ in an inner-most sum, to obtain
\begin{multline*}
    k^T_{f, \o}(x) = \sum_{P_1 \subset P_2} \sum_{\delta \in P_1'(F) \backslash G'(F)} F^{P_1}(\delta x, T)\sigma_{P_1}^{P_2}(H_{0'}(\delta x)_{P_1} - T_{P_1}) \left[\sum_{P_1 \subset P \subset P_2} \varepsilon_P^G k_{f, P, \o}(\delta x) \right].
\end{multline*}

Now define $$i^T_{P_1, P_2}(x) \coloneqq F^{P_1}(x, T)\sigma_{P_1}^{P_2}(H_{0'}(x)_{P_1} - T_{P_1})$$ and set $$k_{f, P_1, P_2, \o}(x) \coloneqq \sum_{P_1 \subset P \subset P_2} \varepsilon_P^G k_{f, P, \o}(x).$$

The upshot of these manipulations is that $i^T_{P_1, P_2}$ is an indicator function (i.e. it takes on values 0 and 1), so $$|k^T_{f, \o}(x)| \leq \sum_{P_1 \subset P_2} \sum_{\delta \in P_1'(F) \backslash G'(F)} i_{P_1, P_2}^T(\delta x)|k_{P_1, P_2, \o}(\delta x)|,$$ and thus to prove \cref{geom-cvg} it suffices to prove that for each pair $P_1 \subset P_2$   
$$\int_{P_1'(F) \backslash G'(\A)} i_{P_1, P_2}^T(x)|k_{f, P_1, P_2, \o}(x)|\,dx < \infty.$$

\begin{remark}
    Of course, in our setting there are only two standard parabolics, so the manipulations above could be (slightly) simplified, however, this gives less insight into the mechanics of truncation. Zydor's partition of unity applies very broadly, and we expect that it will feature prominently in any proof of convergence of the geometric RTF.
\end{remark}

\subsubsection{Step 2: Retract onto Levi Galois Symmetric Spaces}

In our setting, the only pairs $P_1 \subset P_2$ for which $\sigma_{P_1}^{P_2}$ is non-zero are $P_1 = G, P_2 = G$ and $P_1 = B, P_2 = G$. In the first case, $i^T_{G, G}(x) = F^G(x, T)$ is compactly supported. Thus, 
$$\int_{G'(F) \backslash G'(\A)} i_{G, G}^T(x)|k_{f, G, G, \o}(x)|\,dx < \infty.$$

The remainder of the proof will focus on the other pair $P_1 = B, P_2 = G$. In this case, the function $k_{f, B, G, \mathfrak{o}}$ decomposes as \begin{align*}
    k_{f, B, G, \mathfrak{o}}(x) &= k_{f, G, \mathfrak{o}}(x) - k_{f, B, \mathfrak{o}}(x)\\
    &= \sum_{\eta_1 \in \o} f(x^{-1}\eta_1 x) - \sum_{\eta_2 \in X_M(F) \cap \o} \int_{N_B(\A)/N_{B'}(\A)} f(x^{-1}\gamma n \theta(n)^{-1} \theta(\gamma)^{-1} x)\,dn,
\end{align*}
where $\gamma \in M_B(F)$ is an element such that $\eta_2 = \gamma\theta(\gamma)^{-1}.$

The following lemma shows that when the truncation parameter $T$ is sufficiently large, the first sum can be taken over $\eta \in \o \cap B(F)$.

\begin{proposition}
    There exists a positive constant $C$ which depends only on $f$, such that for all truncation parameters $T > C$, if $x \in B'(F)\backslash G'(\A)$ is such that $i_{B, G}^T(x) \neq 0$, then $$k_{f, B, G, \mathfrak{o}}(x) = \sum_{\eta_1 \in B(F) \cap \mathfrak{o}} f(x^{-1}\eta_1 x) - \sum_{\eta_2 \in X_M(F) \cap \mathfrak{o}} \int_{N_B(\A)/N_{B'}(\A)} f(x^{-1}\gamma n \theta(n)^{-1} \theta(\gamma)^{-1})\,dn,$$ where $\gamma \in M_B(F)$ is such that $\eta_2 = \gamma\theta(\gamma)^{-1}$.
\end{proposition}
\begin{proof}
    The second term on the RHS is precisely $-k_{f, B, \mathfrak{o}}(x)$, so it suffices to show that $k_{f, G, \mathfrak{o}}(x)$ and the first term on the RHS agree. Now, consider $\eta = \eta_1 \in X(F) \cap \mathfrak{o}$ that does not lie in $B(F)$, so its lower-left entry is non-zero, say $\tau^{1/2} y$ for $y \in F$. Write $x = nmak$ relative to the Iwasawa decomposition, where $n \in N_{B'}(\A)$, $a \in A_{B'}^{\infty}$, $m \in M_{B'}(\A)^1$ and $k \in K'$. Then we compute that $x^{-1}\eta x = k^{-1}a^{-1} m^{-1} n^{-1} \eta n m a k$ has the form $$\Mat{*}{*}{\tau^{1/2}y \alpha_1(m) \alpha_1(a)}{*},$$ where $\alpha_1$ is the positive root for $G'$. Since $y$ is non-zero, this has adelic absolute value $|\tau^{1/2}y\alpha_1(m)\alpha_1(a)|_{\A} = e^{\alpha_1(H_{0}(a))}$. 

    Now since the support of $f$ is compact, there exists a positive constant $C_f$ such that when you project the support of $f$ onto the lower-left coordinate, its adelic absolute value is strictly bounded above by $C_f$. But by choosing the truncation parameter sufficiently large ($T > C \coloneqq 2C_f^{1/2}$ suffices) our condition $i_{B, G}^T(x) = 1$ forces $\sigma_B^G(H_{0'}(a) - T) = 1$, and so the torus component $a$ of $x$ has $H_{0}(a) > T/2$, and thus $e^{\alpha_1(H_{0}(a))} \gg T^2/4 > C_f$. Thus, $f(x^{-1}\eta x)$ must be zero. The final observation is that the lower bound on $T$ required to run this argument depends only on $f$, and not $\eta$.
\end{proof}

In turn, the set $B(F) \cap \o$ naturally retracts onto the Levi Galois symmetric space $X_M$.

\begin{lemma}\label{retraction-lemma}
    Fix $\eta \in B(F) \cap \o$.
    \begin{enumerate}
        \item The element $\eta$ can be decomposed uniquely into $\eta = \eta_{M_B} \eta_{N_B}$, where $\eta_{M_B} \in X_M(F) \cap \o$ and $\eta_{N_B} \in N_B(F)$.
        \item Moreover, if we fix $\gamma \in M_B(F)$ such that $\eta_{M_B} = \gamma\theta(\gamma)^{-1}$, there exists a unique $n_1 \in N_B^{opp}(F)$ such that $\gamma n_1$ is a representative for $\eta$, that is $\eta = \gamma n_1 \theta(\gamma n_1)^{-1}$.
    \end{enumerate}
\end{lemma}
\begin{proof}
    Let $\lambda \in X_*(A_0') = X_*(A_0)$ be a generic cocharacter, such that $B = P_G(\lambda)$. Then $M_B = Z_G(\lambda)$ and $N_B = \{ g \in G : \lim_{t \to 0} \lambda(t)g\lambda(t)^{-1} = 1 \}$.

    Since $\eta \in B(F)$, we have a Levi decomposition $\eta = \eta_{M_B}\eta_{N_B}$. Since $\eta_{M_B} = \lim_{t \to 0}\lambda(t)\eta\lambda(t)^{-1},$ we see that $$\theta(\eta_{M_B}) = \lim_{t \to 0}\lambda(t)\theta(\eta)\lambda(t)^{-1} = \lim_{t \to 0}\lambda(t)\eta^{-1}\lambda(t)^{-1} = \eta_{M_B}^{-1},$$ and thus $\eta_{M_B} \in X_M(F)$. Since $\chi : X \to \A^1$ is $G'$-conjugation invariant, we have $$\chi(\eta_{M_B}) = \lim_{t \to 0} \chi(\lambda(t)\eta\lambda(t)^{-1}) = \chi(\eta),$$ and thus $\eta_{M_B} \in \o$. This proves (1).

    To see the proof of (2), using our explicit coordinates on $X$, we can write $\eta = \mat{x}{\tau^{1/2} a}{}{\overline{x}}$ where $x \in E^{\x}$ such that $x\overline{x} = 1$. Since $x$ has unit norm, we can write $x = y/\overline{y}$ for $y \in E^{\x}$. Then $\eta_{M_B} = \mat{x}{}{}{\overline{x}}$ and one choice for $\gamma$ is $\gamma = \mat{y}{}{}{y^{-1}}$.
            
    Now define $$n_1 \coloneqq \Mat{1}{\tau^{1/2} a/2y\overline{y}}{}{1}.$$ We easily compute $$n_1\theta(n_1)^{-1} = \Mat{1}{\tau^{1/2} a/y\overline{y}}{}{1}.$$ Moreover, a short computation using the identity $x\overline{y}^2 = y\overline{y}$ shows $\gamma n_1 \theta(\gamma n_1)^{-1} = \eta$. Thus, $\gamma n_1$ is a representative for $\eta$, as desired. Finally, notice that $n_1 \in N_B^{opp}(F)$. This completes the proof of (2).
\end{proof}

As an immediate consequence of the above lemma, we see that when the truncation parameter is large relative to our test function $f$, we can rewrite the difference of parabolic kernels as \begin{multline}\label{pre-fourier}
    \sum_{\eta \in X_M(F) \cap \o} \left[ \sum_{n_1 \in N_B^{opp}(F)} f(x^{-1} \gamma n_1 \theta(n_1)^{-1} \theta(\gamma)^{-1} x) \right] \\- \int_{N_B^{opp}(\A)} f(x^{-1} \gamma n\theta(n)^{-1} \theta(\gamma)^{-1} x)\,dn,
\end{multline} where $\gamma \in M_B(F)$ is such that $\eta = \gamma\theta(\gamma)^{-1}$ and we have identified $N_B/N_{B'}$ with $N_B^{opp}$.

\subsubsection{Step 3: Fourier Analysis}

Now, for $x \in G'(\A)$, define $f_x(y) \coloneqq f(x^{-1}yx)$, and for each $\eta \in X_M(F) \cap \o$, set $\tau^{1/2} \zeta \coloneqq \kappa_\varepsilon^{-1}(\eta)$. Here we choose $\varepsilon = 1$ or $-1$ depending on $\chi(\o)$. If $\chi(\o) \neq \pm 1$, the choice is immaterial, and our convention is to choose $\varepsilon = 1$. If $\chi(\o) = 1$, choose $\varepsilon = -1$, and vice versa if $\chi(\o) = -1$.

For the remainder of this section, we take the truncation parameter sufficiently large, so that we can apply the results of the previous section.

Then we can use the Cayley transform discussed in \cref{cayley-section} to rewrite \cref{pre-fourier} as
\begin{multline}
    k_{f, B, G, \mathfrak{o}}(x) = \sum_{\eta \in X_M \cap \mathfrak{o}} \sum_{N_1 \in \mathfrak{n}_{B'}(F)} f_x(\kappa_{\varepsilon}(\tau^{1/2}(\zeta + N_1)))\\ - \int_{\mathfrak{n}_{B'}(\A)} f_x(\kappa_{\varepsilon}(\tau^{1/2}(\zeta + N)))\,dN.
\end{multline}

Let $\psi : F \backslash \A \to \C^{\x}$ be a fixed additive character compatible with the Haar measure on $\mathfrak{n}_{B'}(\A)$. Finally, define $g_{x, \eta}(N) \coloneqq f_x(\kappa_{\varepsilon}(\tau^{1/2}(\zeta + N)))$. Since $f_x \in \Cc(X(\A))$, we see that $g_{x, \eta} \in \Cc(\mathfrak{n}_{B'}(\A))$.

By Poisson summation, we have that $$\sum_{N_1 \in \mathfrak{n}_{B'}(F)} g_{x, \eta}(N_1) = \sum_{N_1 \in \mathfrak{n}_{B'}(F)} \widehat{g}_{x, \eta}(N_1).$$ However, $\widehat{g}_{x, \eta}(0) = \int_{\mathfrak{n}_{B'}(\A)} f_x(\kappa_{\varepsilon}(\tau^{1/2}(\zeta + N)))\,dN.$ It follows that \begin{align*}
    k_{f, B, G, \mathfrak{o}}(x) &= \sum_{\eta \in X_M(F) \cap \mathfrak{o}} \sum_{N_1 \neq 0 \in \mathfrak{n}_{B'}(F)} \widehat{g}_{x, \eta}(N_1) \\ &= \sum_{\eta \in X_M(F) \cap \mathfrak{o}} \sum_{N_1 \neq 0 \in \mathfrak{n}_{B'}(F)} \int_{\mathfrak{n}_{B'}(\A)} f_x(\kappa_{\varepsilon}(\tau^{1/2}(\zeta + N)))\psi(\langle N_1, N \rangle)\,dN.
\end{align*}
We have achieved the cancellation that was sought after. Following Arthur, the key point is that the innermost integral is now rapidly decaying in the torus coordinate of $x \in G'(\A)$.

To see this, we now return to the integral, $$\int_{B'(F) \backslash G'(\A)} i_{B, G}^T(x)|k_{f, B, G, \mathfrak{o}}(x)|\,dx.$$ Write $x = namk$ where $n \in [N_{B'}]$, $a \in A_{B'}^\infty$, $m \in [M_{B'}]^1$, $k \in K'$, along with the corresponding change of measure $dx = e^{-2\rho_{B'}(H_{0'}(a))}dn\,da\,dm\,dk$. Then it suffices to bound $$\sum_{\eta \in X_M(F) \cap \mathfrak{o}} \sum_{\xi \neq 0 \in \mathfrak{n}_{B'}(F)} \sup_y \int_{H_{0'}(a) > T} |\widehat{g}_{\eta, y}(\Ad(a)\xi)|\,da,$$ where the supremum is taken over the compact set of $y = a^{-1}namk$ such that $i_{B, G}^T(y) = 1$.

Here we've used \cref{cayley-is-equivariant} along with a change of variables to ensure the measures work out correctly. Since $\Ad(a)$ acts by dilation on $\xi \neq 0$, such an integral converges. The reason is that $\widehat{g}_{\eta, y}(-)$ is \emph{Schwartz}, and thus rapidly decaying at infinity, so the entire sum over non-zero $\xi \in \mathfrak{n}_{B'}(F)$ can be bounded by a multiple of $$ \int_{H_{0'}(a) > T} \sum_{\xi \neq 0} ||\Ad(a)\xi||^{-N}\,da,$$ where $N > 0$ is taken sufficiently large. Here $|| \cdot ||$ is an appropriate norm on the vector space $\mathfrak{n}_{B'}$. Such a norm can be chosen such that $$||\Ad(a)\xi|| \gg ||\xi||e^{2H_{0'}(a)},$$ where the implicit constant depends only on the groups involved. It follows that this integral can be bounded by a zeta value multiplied by an integral of the form \begin{equation}\label{integral-bound}
    \int_{H_{0'}(a) > T} e^{-2NH_{0'}(a)}\,da.
\end{equation} This integral clearly converges. This completes the proof of the convergence of the truncated geometric kernel. \qed

This final convergence argument follows the same steps as \cite[p.~946-947]{arthur-1}.

\section{Evaluating the Constant Term of the Truncated Geometric RTF}

As above, fix a test function $f \in \Cc(X(\A))$. With \cref{geom-cvg} in hand, we'd now like to show that the resulting geometric RTF distribution $J_\o^T(f)$ is linear in the truncation parameter $T \in \a_{0'}$.

\begin{theorem}\label{geom-dist-is-linear-in-T}
    Let $f \in \Cc(X(\A))$ and let $\o$ be a geometric datum. Let $T \in \a_{0'}$ be the truncation parameter.
    
    The geometric RTF distribution $J_\o^T(f)$ is a linear function of $T$, when $T$ is sufficiently large. More precisely, for $T > T'$, where $T'$ is a \textbf{fixed} auxiliary truncation parameter, $$J_\o^T(f) = J^{T'}_\o(f) + (T - T')\left[\sum_{\eta \in X_M \cap \o} J^{M_B}_\eta(f_B)\right],$$
    where the notation is explained in the proof below.
\end{theorem}

The key input is a result of Zydor we call the absorption lemma. 

To state the absorption lemma, we require yet another family of indicator functions, this time on $\a_{0'} \x \a_{0'}$: $$\Gamma_B^G(-, -), \Gamma_G^G(-, -), \text{ and } \Gamma_{B}^B(-, -).$$ These functions have been defined and computed explicitly in \cref{gamma-calc}.

\begin{lemma}[{\cite[Lemma 3.3, (3)]{coarse}}]\label{absorption-lemma}
    For any $H, X \in \a_{0'}$,
    $$\widehat{\tau}_B(H - X) = \sum_{\substack{R \in \F^G(B')\\B \subset R \subset G}} \varepsilon_R^G\widehat{\tau}_B^R(H^R)\Gamma_R^G(H_R, X_R).$$
\end{lemma}

The absorption lemma will allow us to introduce a new auxiliary truncation parameter which will ``absorb the main term'' of $J_\o^T(f)$, leaving behind an error term which is polynomial-exponential (in fact linear) in the truncation parameter.

\subsection{Parabolic Descent}

In fact, we will see that the error term we obtain from the absorption lemma can be expressed in terms of the RTF for the Levi Galois symmetric space $X_M$. In \cref{rtf-for-rk-one-tori}, we specialize the theory from part one to Galois symmetric spaces associated to rank one tori, of which $X_M$ is an example.

\subsubsection{Descent for Test Functions}

\begin{definition}
    Let $f \in \Cc(X(\A))$. Define a new function $f_M \in \Cc(X_M(\A))$ by $$f_{M}(\eta) = \int_{K'} \int_{N_B(\A)/N_{B'}(\A)} f(k^{-1} \gamma n\theta(n)^{-1} \theta(\gamma)^{-1} k) e^{-\rho_B(H_B(\gamma))}\,dn\,dk,$$ where $\gamma \in M_{B}(\A)$ such that $\eta = \gamma\theta(\gamma)^{-1}$.
\end{definition}

% To see that the factor above is correct, note that the natural quotient measure on $N_B(\A)/N_{B'}(\A)$ transforms like $d(a n a^{-1}) = e^{2H_{B'}(a)}dn = e^{H_B(a)}dn$. It's modulus is the quotient of the moduli of its factors.

Note that every $\eta \in X_M(\A)$ can indeed be written this way, as $\ker^1(\A, X_M)$ vanishes. One easily verifies that the right-hand side is independent of the choice of $\gamma \in M_B(\A)$, and the resulting function is smooth and compactly supported. Moreover, it is clear that if $f$ is factorizable, so is $f_M$.

\subsubsection{Descent for Geometric Data}

Note that the natural $M_{B'}$ action on $X_M$ is trivial, so the corresponding invariant quotient $X_M/\!/M_{B'}$ is isomorphic to $X_M$, and thus the geometric data associated to $X_M$ is simply $\O^{X_M} \coloneqq X_M(F)$, the set of $F$-points of $X_M$.

There is a natural map of sets $\iota_B : \O^{X_M} \to \O^X$ given by $$\iota_B(\eta) \coloneqq \text{ the unique $\o \in \O^X$ such that $\eta \in \o$}.$$ Associated to $\iota_B : \O^{X_M} \to \O^X$ is the following basic finiteness result:

\begin{lemma}\label{iota-finite-fibers}
    The fibers of $\iota_B : \O^{X_M} \to \O^X$ are finite of order $0, 1$ or $2$.
\end{lemma}
\begin{proof}
    Fix $\o \in \O^X$. If $X_M(F) \cap \o$ is empty, $\iota_B^{-1}(\o)$ is empty. Otherwise, let $\eta_i = \mat{x_i}{}{}{\overline{x_i}} \in X_M(F) \cap \o$ be two elements in this intersection, where $x_i\overline{x}_i = 1$, $i = 1, 2$. If $x_1 \neq \overline{x_1}$ are distinct, then the common characteristic polynomial $p_\o(T)$ of $\eta_1, \eta_2$ is separable with roots $\{ x_1, \overline{x_1} \}$, and thus the only possibilities are $x_1 = x_2$ (so $\eta_1 = \eta_2$) or $x_1 = \overline{x_2}$ (so $\eta_1 = w\eta_2w^{-1}$). Otherwise, if $x_1 = \overline{x_1}$, we must have $x_1 = \pm 1$, and in each case, there is only one such $\eta \in X_M(F)$.
\end{proof}

In particular, we see that $\iota_B^{-1}(\o) = X_M(F) \cap \o$ is always finite.

\subsubsection{Compatibility of Refined Geometric Kernels}

\begin{proposition}\label{descent-compat}
    Let $f \in \Cc(X(\A))$ and let $f_M \in \Cc(X_M(\A))$ be as above. Then $$\int_{[M_{B'}]^1} \int_{K'} k_{f, B, \o}(mk)\,dk\,dm = \sum_{\eta \in \iota_B^{-1}(\o)} \int_{[M_{B'}]^1} k_{f_M, M_B, \eta}(m)\,dm.$$
\end{proposition}

%Note that there is no truncation needed on the right-hand side because $[M_{B'}]^1$ is compact.

\begin{proof}
    This is an easy consequence of the definitions.
\end{proof}

\begin{corollary}\label{descent-compat-cor}
    Let $f \in \Cc(X(\A))$ and let $f_M \in \Cc(X_M(\A))$ be as above. Then $$\int_{[M_{B'}]^1} \int_{K'} k_{f, B, \o}(mk)\,dk\,dm = \sum_{\eta \in \iota_B^{-1}(\o)} J_\eta^{M_B}(f_M),$$
    where on the right-hand side we see geometric RTF distributions associated to the Levi Galois symmetric space $X_M$ (see \cref{rtf-for-rk-one-tori}).
\end{corollary}

Note that the geometric RTF distributions for $X_M$ do not need to be regularized, as $[M_{B'}]^1$ is compact.

\subsection{Proof of \Cref{geom-dist-is-linear-in-T}}

Recall that $$J_\o^T(f) = \int_{[G']} k_{f, G, \o}(x) - \sum_{\delta \in B'(F) \backslash G'(F)} \widehat{\tau}_B(H_{0'}(\delta x)_B - T_B)k_{f, B, \o}(\delta x)\,dx.$$

Fix $T' \in \a_{0'}$ a sufficiently large auxiliary truncation parameter. Choose $T \in T' + \a_{0'}^+$. By the absorption lemma \labelcref{absorption-lemma}, applied with $H = H_{0'}(\delta x)_B - T'_B$ and $X = T_B - T'_B$ we can rewrite the sum above as
\begin{multline*}
    \int_{[G']} k_{f, G, \o}(x) - \sum_{\delta \in B'(F) \backslash G'(F)} \big[\widehat{\tau}_B^G(H_{0'}(\delta x)_B^G - (T')_B^G)\Gamma_G^G(H_{0'}(\delta x)_G - T'_G, T_G - T'_G) \\ - \widehat{\tau}_B^B(H_{0'}(\delta x)_B^B - (T')_B^B)\Gamma_B^G(H_{0'}(\delta x)_B - T_B, T_B - T'_B) \big]k_{f, B, \o}(\delta x)\,dx.
\end{multline*}

Now, splitting this integral in two, and simplifying (using $X^G = X$, $\Gamma_G^G(H_G, X_G) = 1$, and $\widehat{\tau}_B^B = 1$) we obtain
\begin{multline*}
    \int_{[G']} k_{f, G, \o}(x) - \sum_{\delta \in B'(F)\backslash G'(F)} \widehat{\tau}_B(H_{0'}(\delta x)_B - (T')_B) k_{f, B, \o}(\delta x)\,dx \\ + \int_{[G']} \sum_{\delta \in B'(F) \backslash G'(F)} \Gamma_B^G(H_{0'}(\delta x)_B - T_B, T_B - T'_B)k_{f, B, \o}(\delta x)\,dx.
\end{multline*}

We recognize the expression on the first line as $J^{T'}_\o(f)$. To evaluate the second integral, unfold along $B'(F) \subset G'(F) \subset G'(\A)$ to arrive at $$\int_{B'(F) \backslash G'(\A)} \Gamma_B^G(H_{0'}(x)_B - T_B, T_B - T'_B)k_{f, B, \o}(x)\,dx.$$

We now apply the adelic Iwasawa decomposition $G'(\A) = B'(\A)K'$ to write $B'(F)\backslash G'(\A) = [N_{B'}] \x A_{B'}^{\infty} \x [M_{B'}]^1 \x K'$. With respect to this decomposition, write $x = n a m k$. Decompose the Haar measure on $G'(\A)$ as $dx = e^{-2\rho_{B'}(H_{B'}(a))}dn\,da\,dm\,dk$. Thus, we have
$$
    \int_{[M_{B'}]^1}\int_{A_{B'}^{\infty}}\int_{K'} \Gamma_B^G(H_{0'}(a)_B - T_B, T_B - T'_B)e^{-2\rho_{B'}(H_{B'}(a))}e^{2\rho_{B'}(H_{B'}(a))}
    k_{f, B, \o}(mk)\,da\,dk\,dm.
$$
where we have used that $k_{f, B, \o}$ is $N_{B'}(\A)$-invariant and that $k_{f, B, \o}(ax) = e^{2\rho_{B'}(H_{B'}(a))}k_{f, B, \o}(x)$.

This integral factors into the product of $$\int_{[M_{B'}]^1} \int_{K'} k_{f, B, \o}(mk)\,dk dm \;\; \text{ and }  \int_{A_{B'}^{\infty}}\Gamma_B^G(H_{0'}(a)_B - T_B, T_B - T'_B)\,da.$$ By \cref{descent-compat-cor}, this first integral is equal to $$\sum_{\eta \in \iota_B^{-1}(\o)} J_\eta^{M_B}(f_B).$$

Using the identification $A_{B'}^{\infty} \simeq \a_{B'} = \R$ afforded by $H_{0'}(-)$ and our choice of measure (the unique Haar measure which assigns the cocharacter lattice covolume one) we can write the second integral as $$\int_{\R} \Gamma_B^G(X - T', T - T')\,dX.$$

Finally, by \cref{gamma-calc} we have $\Gamma^G_B(X - T', T - T') = [\R_{>0}](X - T') - [\R_{>0}](X - T)$. Since $T > T'$ by hypothesis, this function is supported on the half-open interval $(T', T]$ where it is equal to one. Thus the integral is equal to $$\int_{T'}^{T}\,dX = T - T'.$$

We conclude that $$J_\o^T(f) = J^{T'}_\o(f) + (T - T')\!\!\sum_{\eta \in \iota_B^{-1}(\o)} J_\eta^{M_B}(f_B).$$ Since $T'$ was fixed, we see that for all $T$ sufficiently large, $J_\o^T(f)$ is linear in $T$. The constant term is visibly $$J_{\o}^{T'}(f) - T'\sum_{\eta \in \iota_B^{-1}(\o)} J^{M_B}_\eta(f_B).$$

This completes the proof of \cref{geom-dist-is-linear-in-T}. \qed

\subsection{The Regularized Geometric RTF Distribution}

We are now (finally!) in a position to define the \emph{regularized} geometric RTF distribution.

\begin{definition}
    Let $f \in \Cc(X(\A))$ and let $\o$ be a geometric datum. By \cref{geom-dist-is-linear-in-T}, for all $T$ sufficiently large, $T \mapsto J^T_\o(f)$ is a linear polynomial in $T$. 
    
    Define $J_{\o}(f)$ as the constant term of this linear polynomial.
\end{definition}

The remainder of this article is devoted to obtaining a \emph{fine} geometric expansion of the regularized RTF distribution.

\begin{remark}\label{interpreting-the-non-const-term}
    If we regard $T \mapsto J_\o^T(f)$ as a function of $T$, \cref{geom-dist-is-linear-in-T} can be rephrased as $$\frac{d}{dT}J_\o^T(f) = \sum_{\eta \in \iota_B^{-1}(\o)} J^{M_B}_\eta(f_M),$$ for $T$ sufficiently large. Thus, we see that the derivative of the geometric RTF with respect to the truncation parameter $T$ has a concrete interpretation in terms of a geometric RTF for the Levi Galois symmetric space. By the RTF for Galois periods of rank one tori (see \cref{rank-one-tori-galois-symmetric-space}), this can also be re-interpreted spectrally. This suggests that it might be profitable to retain higher-order terms of RTF distributions when stabilizing, or when comparing two RTF. 
\end{remark}

\part{The Fine Geometric Expansion}\label{part-three}

Let $\mathfrak{o} \in \mathcal{O}^X$ be a geometric datum and fix $f \in \Cc(X(\A))$. Our goal now is to obtain a more useful expression for the regularized geometric distribution $J_{\o}(f)$. This task can be split into three cases, depending on the nature of the geometric datum. The elliptic datum are easily dealt with; in this case $J_\o(f)$ is given by a relative orbital integral. The regular semisimple but non-elliptic datum yield weighted relative orbital integrals. Finally, one must deal with the unipotent data, and this is where the bulk of the work lies. Here $J_\o(f)$ can be expressed as a volume term plus a certain sum of special values of Tate zeta integrals. Surprisingly, the key step is to re-interpret the truncated geometric kernel for the unipotent datum in terms of the \emph{Springer resolution} of the rational nilpotent cone in $\mathfrak{sl}_{2, F}$. This geometric viewpoint is simultaneously linked to the concrete mechanics of truncation, as well as the presence of Tate zeta functions.

\subsection{Descendant Galois Symmetric Pairs}\label{descendants}

\begin{definition}
    Let $\o \in \mathcal{\O}$ be a geometric datum, and let $p_{\o}(T) \in F[T]$ be the polynomial associated to $\chi(\o)$. Define the splitting type of $\o$ to be the $F$-algebra $\Spl(\o) \coloneqq F[T]/p_{\o}(T)$.
\end{definition}

It is easy to see that $\o$ is regular semisimple (in the invariant-theoretic sense, see \cref{geometric-data}) if and only if its splitting type is a quadratic étale $F$-algebra.

\begin{definition}
    Let $\o \in \mathcal{\O}$ be a geometric datum. Then we say
    \begin{enumerate}
        \item $\o$ is elliptic if its splitting type is a quadratic étale $F$-algebra distinct from $E$.
        \item $\o$ is regular semisimple if its splitting type is an arbitrary quadratic étale $F$-algebra.
        \item $\o$ is unipotent if its splitting type is non-reduced.
    \end{enumerate}
\end{definition}

The definitions above are justified by the following construction, which uses the classification of Galois symmetric spaces associated to rank one tori. This is discussed in detail in \cref{rtf-for-rk-one-tori}. In particular, see \cref{galois-symmectric-spaces-from-tori} where we explain how to obtain such a Galois symmetric space from a so-called generalized biquadratic extension $M/F$ and a quadratic subextension $L \subset M$.

\begin{proposition}
    Let $\o$ be a regular semisimple geometric datum and let $\eta \in \o$ be a basepoint. Let $G_\eta$, $G'_\eta$ be the stabilizer of $\eta$ under the $G$-conjugation and $G'$-conjugation actions, respectively.

    Then $(G_\eta, G'_\eta)$ is a Galois symmetric pair of rank one tori, isomorphic to the Galois symmetric pair associated to the generalized biquadratic extension $\Spl(\o) \otimes_F E$ and quadratic subextension $\Spl(\o)'$, obtained by reflecting $\Spl(\o)$ through $E$.
\end{proposition} 

\begin{proof}
     To begin, note that $\eta \in X(F) \subset G(F) = G'_E(F \otimes_F E) = G'(E)$.
     
     First, we claim that $\Res_{E/F}(G'_{\eta, E}) = G_{\eta}$ as subgroups of $G$. Here $G'_{\eta, E}$ denotes the base change of $G'_{\eta}$ to $E$. Indeed, let $A$ be a test $F$-algebra. Then
        \begin{align*}
            G_\eta(A) &= \{ g \in G(A) : g\eta g^{-1} = \eta \} \\
            &= \{ g \in G'_E(A \otimes_F E) : g \eta g^{-1} = \eta \} \\ 
            &= \{ g \in G'(A \otimes_F E) : g \eta g^{-1} = \eta \} = \Res_{E/F}(G_\eta')(A).
        \end{align*}
    Thus, $(G_\eta, G'_{\eta})$ is a Galois symmetric pair, with Galois involution given by restricting the Galois involution on $(G, G')$.

     Since the characteristic polynomial of $\eta$ is separable, we see that $\eta$ is regular semisimple \emph{as an element of $G$}, and thus $G_\eta$ is a maximal torus in $G$. Thus, we've shown that $(G_\eta, G'_\eta)$ is a Galois symmetric pair of rank one tori.

     Now, let $M/F$ be the generalized biquadratic extension corresponding to $G_\eta$, and let $L/F$ be the subextension corresponding to $G'_\eta$ as in \cref{biquadratic-tori}.

     It remains to show that $L \simeq \Spl(\o)'$, where $\Spl(\o)'$ is the reflection of $\Spl(\o)$ through $E$ (see \cref{reflection}). We spell out the details in the most interesting case: when $M$ and $L$ are both honest field extensions, and so $$(G_\eta, G'_\eta) = (\Res_{E/F} \Nm^1_{M/E}, \Nm^1_{L/F}).$$ The remaining cases are handled entirely analogously.

     Write $L = F(\alpha^{1/2})$, $L' = F(\beta^{1/2})$ and $E = F(\alpha^{1/2}\beta^{1/2})$ for the three quadratic subextensions of $M$. With this presentation of $E$, we can write $X$ as $$X = \{ \Mat{a}{b\alpha^{1/2}\beta^{1/2}}{c\alpha^{1/2}\beta^{1/2}}{\overline{a}} : a \in E, b, c \in F, a\overline{a} - bc\alpha\beta = 1 \}.$$ Without loss of generality, we can identify $G'_\eta$ with a copy of $\Nm^1_{L/F}$ inside $G'$ embedded as $$\Nm^1_{L/F} = \{ \Mat{x}{\alpha y}{y}{x} : x, y \in F, x^2 - \alpha y^2 = 1 \} \subset G'.$$

     An easy calculation shows that $$\Mat{x}{\alpha y}{y}{x} \text{ stabilizes } \Mat{a}{b\alpha^{1/2}\beta^{1/2}}{c\alpha^{1/2}\beta^{1/2}}{\overline{a}} \text{ if and only if $a = \overline{a}$ and $b = \alpha c$.}$$ Thus, we see that the subscheme fixed by $\Nm^1_{L/F}$ is precisely the matrices that have the form $$\eta = \Mat{a}{\alpha c \alpha^{1/2}\beta^{1/2}}{c \alpha^{1/2}\beta^{1/2}}{a}$$ where $a, c \in F$ and $a^2 - \beta(\alpha c)^2 = 1$.

     If $\eta \in \o$ is such an element, it has characteristic polynomial $T^2 - 2aT + 1 \in F[T]$, which has roots $a \pm \sqrt{a^2 - 1}$. Since $a^2 - 1 = \beta (\alpha c)^2$, these roots are $a \pm \alpha c \beta^{1/2}$. Thus, we see that $\Spl(\o) \simeq F(\beta^{1/2}) = L'$, or equivalently $\Spl(\o)' \simeq L$.
\end{proof}

For regular semisimple $\o \in \O^X$ with a basepoint $\eta \in \o$, we call the corresponding Galois symmetric pair of rank one tori $(G_\eta, G'_\eta)$ the descendant Galois symmetric pair, following \cite{aizenbud-gourevitch}. The isomorphism class of such is independent of the chosen basepoint (it depends only on $\Spl(\o)$).

\begin{corollary}\label{ellipticity-concretely}
    Let $\o \in \mathcal{\O}$ be a regular semisimple geometric datum and let $(H, H')$ be the (isomorphism class) of the descendant Galois symmetric pair. Then the following are equivalent
    \begin{enumerate}
        \item $\o$ is elliptic.
        \item $H'$ is elliptic.
        \item $\o$ does not meet $X_M$.
    \end{enumerate}
\end{corollary}
\begin{proof}
    We have that $\o$ is elliptic if and only if $\Spl(\o) \neq E$ if and only if $\Spl(\o)' \neq F \x F$. The claim now follows immediately from the previous proposition.
\end{proof}

% We are now ready to define the endoscopic symmetric spaces that will appear in the pre-stabilization.
\quash{
\begin{definition}
    Let $\e = L'$ be an étale quadratic $F$-algebra. The endoscopic symmetric space $X_\e$ is defined to be the Galois symmetric space associated to the pair $$(\Res_{E/F} \Nm^1_{M/E}, \Nm^1_{L/F})$$ where $M = L' \otimes_F E$.
\end{definition}
}

\subsection{The Elliptic Terms}

Let $\o$ be an elliptic geometric datum. Fix a basepoint $\eta \in \o$, and let $(G_\eta, G'_\eta)$ be the descendant Galois symmetric pair. By \cref{ellipticity-concretely}, $G'_{\eta}$ is elliptic, and $\o \cap M_B$ is empty. It follows that $k^T_{f, B, \o}$ is identically zero, and we conclude that $$J^T_{\o}(f) = \int_{[G']} k^T_{f, \o}(x)\,dx = \int_{[G']} k_{f, G, \o}(x)\,dx.$$ Recall that $$k_{f, G, \o}(x) = \sum_{\eta \in \o} f(x^{-1}\eta x).$$ Let $\{ \eta_{\xi} \}$ be a system of representatives for the rational $G'$-orbits in $\o$, where $\xi$ runs through $\ker[H^1(F, G'_\eta) \to H^1(F, G')]$, and $\eta_1 = \eta$. Thus, we can rewrite this as $$k_{f, G, \o}(x) = \sum_{\xi \in H^1(F, G'_\eta)} \sum_{ \delta \in G'_{\eta_\xi}(F) \backslash G'(F) } f(x^{-1}\delta^{-1} \eta_\xi \delta x).$$ Unfolding, we see that $$J^T_{\o}(f) = \sum_{\xi \in H^1(F, G'_\eta)} \vol([G'_{\eta_{\xi}}]) \int_{G'_{\eta_{\xi}}(\A) \backslash G'(\A)} f(x^{-1}\eta_\xi x)\,dx.$$

We summarize our calculation in the following

\begin{proposition}\label{elliptic-calc}
    Let $\o \in \mathcal{O}^X$ be an elliptic geometric datum with basepoint $\eta \in \o$. Let $\{ \eta_{\xi} \}$ be a system of representatives for the rational $G'$-orbits in $\o$, where $\xi$ runs through $\ker[H^1(F, G'_\eta) \to H^1(F, G')]$, and $\eta_1 = \eta$. Then $$J_{\o}(f) = J_{\o}^T(f) = \sum_{\xi \in H^1(F, G'_\eta)} \vol([G'_{\eta_{\xi}}]) \int_{G'_{\eta_{\xi}}(\A) \backslash G'(\A)} f(x^{-1}\eta_\xi x)\,dx.$$
\end{proposition}

\subsection{The Regular Semisimple Terms}

Now let $\o$ be a regular semisimple geometric datum that is not elliptic. Fix $\eta \in \o$ a basepoint such that $\eta = \mat{y}{}{}{\overline{y}}$, where $y \in E^{\x}$ and $y\overline{y} = 1$. Then we see that $w \cdot \eta = w\eta w^{-1} = \mat{\overline{y}}{}{}{y}$ also lies in $\o$ and by \cref{iota-finite-fibers} this exhausts $\o \cap X_M$. By Hilbert's Theorem 90, there exists $u \in E^{\x}$ such that $y = u/\overline{u}$. If we set $\gamma = \mat{u}{}{}{u^{-1}} \in M_B(F)$, we have that $\eta = \gamma\theta(\gamma)^{-1}.$ The descendant Galois symmetric pair is $(M_B, M_{B'})$. Thus $\o$ is a single $G'(F)$-orbit.

\begin{proposition}\label{rss-calc}
    Let $\o \in \mathcal{O}^X$ be a regular semisimple geometric datum that is not elliptic. Fix a basepoint $\eta \in \o$, and define $v(x) \coloneqq H_{0'}(x) + H_{0'}(wx)$ for all $x \in G'(\A)$. Then
    \begin{multline*}
    J^T_\o(f) = -2 \vol([M_{B'}]^1) \int_{M_{B'}(\A)^1 \backslash G'(\A)} f(x^{-1}\eta x) v(x)\,dx \\ + 4T \vol([M_{B'}]^1) \int_{M_{B'}(\A)^1 \backslash G'(\A)} f(x^{-1}\eta x)\,dx.
    \end{multline*}
    In particular, $$J_\o(f) = -2 \vol([M_{B'}]^1) \int_{M_{B'}(\A)^1 \backslash G'(\A)} f(x^{-1}\eta x) v(x)\,dx.$$
\end{proposition}

\subsubsection{Modified Kernels}

To make further progress, following Arthur, we introduce modified kernels that are more amenable to explicit computation. Set $$\widetilde{k}_{f, G, \o}(x) \coloneqq k_{f, G, \o}(x)$$ and define $$\widetilde{k}_{f, B, \o}(x) \coloneqq \sum_{\eta \in \o \cap X_M(F)} \sum_{n \in N_B(F)/N_{B'}(F)} f(x^{-1}\gamma n \theta(n)^{-1} \theta(\gamma)^{-1} x).$$ Likewise, for $T \in \a_{0'}$, define the modified truncated kernel $$\widetilde{k}^T_{f, \o}(x) \coloneqq \widetilde{k}_{f, G, \o}(x) - \sum_{\delta \in B'(F) \backslash G'(F)} \widehat{\tau}_B(H_{0'}(\delta x) - T) \widetilde{k}_{f, B, \o}(\delta x),$$ and the modified geometric RTF distribution $$j^T_\o(f) \coloneqq \int_{[G']} \widetilde{k}^T_{f, \o}(x)\,dx.$$

\begin{proposition}
    With the notation as above,
    \begin{enumerate}
        \item For $T \in \a_{0'}$ sufficiently large, the integral of the modified kernel converges absolutely $$\int_{[G']} |\widetilde{k}^T_{f, \o}(x)| \,dx < \infty.$$
        \item For $T \in \a_{0'}$ sufficiently large, we have $$\left| \int_{[G']} k^T_{f, \o}(x) - \widetilde{k}^T_{f, \o}(x)\,dx \right| \ll e^{-T}.$$
    \end{enumerate}
\end{proposition}
\begin{proof}
    The proof of claim (1) follows the same pattern as the proof of absolute convergence of the usual truncated geometric kernel, but it is even simpler. The same formal partition of unity manipulations apply, so we are reduced once again to bounding $$\int_{B'(F) \backslash G'(\A)} i^T_{B, G}(x)|\widetilde{k}_{f, B, G, \o}(x)|\,dx,$$ where $\widetilde{k}_{f, B, G, \o}$ is defined by the same formula as $k_{f, B, G, \o}$, but with $\widetilde{k}_{f, P, \o}$ instead of $k_{f, P, \o}$, for all $P \in \F^G(B)$. Now using retraction onto the Levi Galois symmetric space, we see that when $T$ is sufficiently large, the term $\widetilde{k}_{f, B, G, \o}(x)$ vanishes identically. This establishes (1).

    To prove claim (2), we rewrite the two truncated kernels using the partition of unity and take differences. Since $k_{f, G, \o} = \widetilde{k}_{f, G, \o}$, these terms cancel, and since $\widetilde{k}_{f, B, G, \o}$ vanishes identically when $T$ is sufficiently large, we see this difference is nothing other than $$\int_{[G']} i^T_{B, G}(x)|k_{f, B, G, \o}(x)|\,dx,$$ the same term that was shown to be finite in the proof of convergence. Thus, the claim amounts to keeping track of how this bound depends on the truncation parameter $T$. By examining the proof of \cref{geom-cvg}, specifically \cref{integral-bound}, we see that $$\int_{[G']} i^T_{B, G}(x)|k_{f, B, G, \o}(x)|\,dx \ll \int_{H_{0'}(a) > T} e^{-2NH_{0'}(a)}\,da = \int_{T}^{\infty} e^{-2N X}\,dX \ll e^{-T},$$ where we are free to choose $N > 0$. The claim is proved.
\end{proof}

\subsubsection{Proof of \Cref{rss-calc}}

The plan now is to calculate $\int_{[G']} \widetilde{k}^T_{f, \o}(x)\,dx$ and show that it is also a linear polynomial in the truncation parameter. By part (2) of the proposition above, this will imply $$J_{\o}^T(f) = \int_{[G']} k^T_{f, \o}(x)\,dx = \int_{[G']} \widetilde{k}^T_{f, \o}(x)\,dx = j_\o^T(f).$$

Returning to our calculation, note that $w\eta w^{-1} = [w\gamma w^{-1}]\theta(w\gamma w^{-1})$ as we have chosen a representative $w \in G'(F)$. Thus,

$$
    \widetilde{k}_{f, B, \o}(x) = \sum_{w \in W} \sum_{ n \in N_B^{opp}(F) } f(x^{-1}[w\gamma w^{-1}]n \theta(n)^{-1} \theta(w\gamma w^{-1})^{-1} x).
$$   

We require a geometric lemma.

\begin{lemma}
    For $w \in W$, we have that $$\{ [w\gamma w^{-1}]n \theta(n)^{-1} \theta(w\gamma w^{-1})^{-1} : n \in N_B^{opp}(F) \}$$ is the $B'(F)$-orbit through $w \cdot \eta$.
\end{lemma}
\begin{proof}
    We prove the claim when $w$ is trivial; the case where $w$ is non-trivial is entirely analogous. Since $M_{B'}(F)$ stabilizes $\eta$, it suffices to show that $N_{B'}(F) \cdot \eta$ agrees with the set above. Now write $n = \mat{1}{a}{}{1}$ for $a \in F$. Then we calculate $$n\eta n^{-1} = \Mat{y}{ a(\overline{y} - y)}{}{\overline{y}}.$$

    On the other hand, if $n' \in N_B^{opp}$, then $n' = \mat{1}{\tau^{1/2} b}{}{1}$ for $b \in F$, and we calculate $$\Mat{u}{}{}{u^{-1}}\Mat{1}{\tau^{1/2}b}{}{1} \theta(\Mat{u}{}{}{u^{-1}}\Mat{1}{\tau^{1/2}b}{}{1})^{-1} = \Mat{y}{2\tau^{1/2}b u\overline{u}}{}{\overline{y}}.$$

    We see that as $a$ ranges through $F$ and $b$ ranges through $F$, these two sets clearly agree.
\end{proof}

As an immediate consequence, we have $$\widetilde{k}_{f, B, \o}(x) = \sum_{w \in W} \sum_{ \delta_2 \in B'_{w \cdot \eta} \backslash B'(F)} f(x^{-1} \delta_2^{-1} (w \cdot \eta) \delta_2 x).$$

Thus, we obtain the following expression for $\widetilde{k}^T_{f, \o}(x)$:
\begin{multline*}
    \widetilde{k}^T_{f, \o}(x) = \sum_{\delta \in M_{B'}(F) \backslash G'(F)} f(x^{-1} \delta^{-1} \eta \delta x) \\ - \sum_{\delta_1 \in B'(F) \backslash G'(F)} \widehat{\tau}_B(H_{0'}(\delta_1 x) - T)\sum_{w \in W} \sum_{ \delta_2 \in B'_{w \cdot \eta}(F) \backslash B'(F)} f(x^{-1} \delta_1^{-1} \delta_2^{-1} (w \cdot \eta) \delta_2 \delta_1 x).
\end{multline*}

Note that $H_{0'}(\delta_2 \delta_1 x) = H_{0'}(\delta_1 x)$, thus we may fold along $B'_{w \cdot \eta}(F) \subset B'(F) \subset G'(F)$ to obtain 
\begin{multline*}
    \widetilde{k}^T_{f, \o}(x) = \sum_{\delta \in M_{B'}(F) \backslash G'(F)} f(x^{-1} \delta^{-1} \eta \delta x) \\ - \sum_{w \in W} \sum_{\delta_1 \in B'_{w \cdot \eta}(F) \backslash G'(F)} \widehat{\tau}_B(H_{0'}(\delta_1 x) - T) f(x^{-1} \delta_1^{-1} w \eta w^{-1} \delta_1 x).
\end{multline*}

Now, in the inner-most sum, write $\delta_2 = w^{-1} \delta_1$. As $\delta_1$ ranges through representatives for $B'_{w \cdot \eta}(F) \backslash G'(F)$, we see that $\delta_2$ ranges through $B'_{\eta}(F) \backslash G'(F)$. Since $B'_{\eta} = G'_\eta = M_{B'}$, we conclude that $$\widetilde{k}^T_{f, \o}(x) = \sum_{\delta \in M_{B'}(F) \backslash G'(F)} f(x^{-1}\delta^{-1} \eta \delta x)\left[1 - \sum_{w \in W} \widehat{\tau}_B(H_{0'}(w \delta x) - T)\right].$$

It follows that $$\int_{[G']} \widetilde{k}^T_{f, \o}(x)\,dx = \int_{M_{B'}(F) \backslash G'(\A)} f(x^{-1}\eta x)\psi^T\! (x)\,dx,$$ where $$\psi^T\!(x) \coloneqq 1 - \sum_{w \in W} \widehat{\tau}_B(H_{0'}(wx) - T).$$

\begin{lemma}
    For $m \in M_{B'}(\A)^1$, $\psi^T\!(mx) = \psi^T\!(x).$
\end{lemma}
\begin{proof}
    Indeed, for $w \in W$, $H_{0'}(wmx) = H_{0'}([wmw^{-1}]wx) = H_{0'}(wx)$.
\end{proof}

Since the stabilizer of $\eta$ in $G'(\A)$ is $M_{B'}(\A) \supset M_{B'}(\A)^1$, we can factor this integral as $$\vol([M_{B'}]^1)\int_{M_{B'}(\A)^1 \backslash G'(\A)} f(x^{-1}\eta x) \left[\int_{A_{B'}^{\infty}} \psi^T\!(ax)\,da \right]\,dx.$$ Here we are using the decomposition $A_{B'}^{\infty} M_{B'}(\A)^1 = M_{B'}(\A)$.

Finally, we calculate the inner-most integral.

\begin{lemma}
    Fix $x \in G'(\A)$. Then
    $$\int_{A_{B'}^{\infty}} \psi^T\!(ax)\,da = -2(H_{0'}(x) + H_{0'}(wx)) + 4T.$$
\end{lemma}
\begin{proof}
    The function $a \mapsto \psi^T\!(ax)$ can only take on values $-1, 0, 1$. It is equal to $-1$ precisely when $H_{0'}(ax) - T = H_{0'}(a) + H_{0'}(x) - T > 0$ and $H_{0'}(wax) - T = -H_{0'}(a) + H_{0'}(wx) - T > 0$. Identifying $A_{B'}^{\infty} \simeq \R$ via $H_{0'}(-)$ and setting $X \coloneqq H_{0'}(a)$, we see that the corresponding contribution to the integral is $$-\int_{T - H_{0'}(x)}^{H_{0'}(wx) - T} \,dX = 2T - (H_{0'}(wx) + H_{0'}(x)).$$

    Similarly $\psi^T\!(x)$ is equal to $1$ precisely when $H_{0'}(a) + H_{0'}(x) - T \leq 0$ and $-H_{0'}(a) + H_{0'}(wx) - T \leq 0$, and the corresponding contribution to the integral is $$\int_{H_{0'}(wx) - T}^{T - H_{0'}(x)}\,dX = 2T - (H_{0'}(wx) + H_{0'}(x)),$$ and the claim follows.
\end{proof}

Define $v(x) \coloneqq H_{0'}(x) + H_{0'}(wx)$. We conclude that \begin{multline*}
    j^T_\o(f) = -2 \vol([M_{B'}]^1) \int_{M_{B'}(\A)^1 \backslash G'(\A)} f(x^{-1}\eta x) v(x)\,dx \\ + 4T \vol([M_{B'}]^1) \int_{M_{B'}(\A)^1 \backslash G'(\A)} f(x^{-1}\eta x)\,dx.
\end{multline*}

Thus $j^T_{\o}$ is manifestly a linear polynomial in $T$. By our earlier remarks, we are done. \qed

\subsection{The Unipotent Terms}

There are two unipotent geometric datum, $\mathfrak{o}_{unip}^+ \coloneqq \chi^{-1}(1)$ and $\mathfrak{o}_{unip}^{-} \coloneqq \chi^{-1}(-1)$. We focus on $\mathfrak{o}^+ \coloneqq \mathfrak{o}_{unip}^+$ in this subsection. At the end, we will describe the modifications required to treat $\o_{unip}^-$.

\subsubsection{Description of Unipotent Data}

For the remainder of this section, set $\varepsilon = -1$. Recall that we have chosen $\tau^{1/2} \in E$ of trace zero such that $E = F(\tau^{1/2})$.

\begin{lemma}
    We have that $$\mathfrak{o}^+ = \{ \Mat{1 + \tau^{1/2} a}{\tau^{1/2} b}{\tau^{1/2} c}{1 - \tau^{1/2} a} : a^2 + bc = 0, a, b, c \in F \}.$$ This space is isomorphic to $\mathcal{N}$, the nilpotent cone of $\mathfrak{sl}_{2, F}$ via $\frac{1}{\tau}\kappa_{\varepsilon}^{-1}$.
\end{lemma}
\begin{proof}
    Let $\eta = \mat{x}{\tau^{1/2} b}{\tau^{1/2} c}{\overline{x}} \in X(F)$ such that $\chi(\eta) = 1$. Then $\eta$ has trace 2 and $x + \overline{x} = 2$. Set $y = x - 1$. Then $\overline{x} - 1 = (2 - x) - 1 = 1 - x = -y$, and $\overline{y} = \overline{x} - 1 = -y$, so $y = \tau^{1/2} a$, and we win. The last claim follows immediately from an explicit calculation with $\kappa_{\varepsilon}^{-1}$.
\end{proof}

Said differently, $\mathfrak{o}^+$ is isomorphic to the nilpotent cone in the $\theta$-antiinvariant subspace of the $\mathfrak{sl}_{2, E}$.

Write $\mathfrak{o}^{+, reg}$ for the complement of $\mat{1}{}{}{1} \in \mathfrak{o}^+$. Base changing to the algebraic closure, $\mathfrak{o}^+$ consists of two geometric orbits: the identity, and the regular unipotent locus. Write $Z' \simeq \mu_2$ for the center of $G'$.

\begin{lemma}
    Let $\eta = \mat{1}{\tau^{1/2} a}{}{1} \in \mathfrak{o}^{+, reg}$. Then $G'_{\eta} = Z' \cdot N_{B'}$. 
\end{lemma}
\begin{proof}
    This is a simple, explicit calculation.
\end{proof}

It follows that for $\eta = \mat{1}{\tau^{1/2}}{}{1} \in \mathfrak{o}^+$ a basepoint in $\o^{+, \reg}$, we have $H^1(F, G'_{\eta}) \simeq H^1(F, \mu_2) \simeq F^{\x}/(F^{\x})^2$, and thus the rational $G'$-orbits in the regular locus are indexed by $F^{\x}/(F^{\x})^2$. In particular, given a family of representatives $\{ b \in F^{\x} \}$ for $F^{\x}/(F^{\x})^2$ containing $1$, $\mat{1}{\tau^{1/2} b}{}{1}$ is a representative for the $G'$-orbit corresponding to $b \in F^{\x}/(F^{\x})^2$.

\subsubsection{Explicit Formulae}

By definition, $$k_{f, G, \mathfrak{o}^+}(x) = \sum_{\eta \in \mathfrak{o}^+} f(x^{-1}\eta x),$$ and $$k_{f, B, \mathfrak{o}^+}(x) = \int_{N_B^{opp}(\A)} f(x^{-1} n \theta(n)^{-1} x)\,dn.$$

Using the orbit decomposition described above, we can rewrite $k_{f, G, \mathfrak{o}^+}(x)$ as $$k_{f, G, \mathfrak{o}^+}(x) = f(1) + \sum_{\delta \in Z'N_{B'}(F) \backslash G'(F)} \sum_{b \in F^{\x}/(F^{\x})^2}  f(x^{-1} \delta^{-1} \mat{1}{\tau^{1/2} b}{}{1} \delta x).$$

Consider the map $\nu : Z'N_{B'}\backslash G' \to B' \backslash G'$. It is easy to see that this map is surjective, and the fibers above $F$-points on the target are $Z'(F) \backslash M_{B'}(F)$-torsors. By grouping the points of $(Z'N_{B'})(F)\backslash G'(F)$ according to their image under $\nu$, we can rewrite the above as $$k_{f, G, \mathfrak{o}^+}(x) = f(1) + \sum_{\delta \in B'(F) \backslash G'(F)} \sum_{\mu \in Z'(F)\backslash M_{B'}(F)} \sum_{b \in F^{\x}/(F^{\x})^2} f(x^{-1} \delta^{-1} \mu^{-1} \mat{1}{\tau^{1/2} b}{}{1} \mu \delta x).$$

\begin{lemma}
    As $\mu$ runs over a set of representatives for $Z'(F)\backslash M_{B'}(F)$, and $b$ runs over a set of representatives for $F^{\x}/(F^{\x})^2$, the elements $\mu^{-1} \mat{1}{\tau^{1/2} b}{}{1} \mu$ exhaust the set $\{ \mat{1}{\tau^{1/2} c}{}{1} : c \in F^{\x} \}$
\end{lemma}
\begin{proof}
    Indeed choose representatives for $Z'(F) \backslash M_{B'}(F)$, say $\mu = \mat{y^{-1}}{}{}{y}$ for $y \in F^{\x}$. Then $$\mu^{-1}\mat{1}{\tau^{1/2}b}{}{1}\mu = \mat{1}{\tau^{1/2}by^2}{}{1},$$ and the claim is proved.
\end{proof}

By the lemma above, we can further rewrite this expression as
$$k_{f, G, \mathfrak{o}^+}(x) = f(1) + \sum_{\delta \in B'(F) \backslash G'(F)} \sum_{b \in F^{\x}} f(x^{-1} \delta^{-1} \mat{1}{\tau b}{}{1} \delta x).$$

Finally, let us use the fact that $\mat{1}{\tau^{1/2} b}{}{1} = \mat{1}{\tau^{1/2} b/2}{}{1}\theta(\mat{1}{\tau^{1/2} b/2}{}{1})^{-1}$, along with a change of variables $b/2 \mapsto b$, to write this as $$k_{f, G, \mathfrak{o}^+}(x) = f(1) + \sum_{\delta \in B'(F) \backslash G'(F)} \sum_{b \in F^{\x}} f(x^{-1} \delta^{-1} \mat{1}{\tau^{1/2} b}{}{1}\theta(\mat{1}{\tau^{1/2} b}{}{1})^{-1} \delta x).$$

Using the isomorphism $a \mapsto \mat{1}{\tau^{1/2} a}{}{1}$ to identify $\A \simeq N_B^{opp}(\A)$, we can also rewrite \begin{multline}
    \sum_{\delta \in B'(F) \backslash G'(F)} \widehat{\tau}_B(H_{0'}(\delta x) - T)k_{B, f, \mathfrak{o}^+}(\delta x) \\ = \sum_{\delta \in B'(F) \backslash G'(F)} \widehat{\tau}_B(H_{0'}(\delta x) - T)\int_{\A} f(x^{-1} \delta^{-1} \mat{1}{\tau^{1/2} a}{}{1} \theta(\mat{1}{\tau^{1/2} a}{}{1})^{-1} \delta x)\,da.
\end{multline}

Now for $x \in G'(\A)$, define $f_x : \A \to \C$ by $$f_x(b) \coloneqq f(x^{-1}\mat{1}{\tau^{1/2} b}{}{1}\theta(\mat{1}{\tau^{1/2} b}{}{1})^{-1}x).$$

\begin{lemma}
    The function $f_x : \A \to \C$ lies in $\Cc(\A)$. If $f = \otimes'_v f_v$ is factorizable and basic outside of the finite set of places $V$, then $f_x = \1_{\mathcal{O}_F[1/V]} \otimes f_{x_V}$, where $f_{x_V} = \otimes_{v \in V} f_{v, x_v}$.
\end{lemma}
\begin{proof}
    First note that if $f$ is factorizable, so is $f_x$. Thus, without loss of generality, we may assume $f$ is factorizable. From the definition of $\Cc(X(\A))$ (see \cref{int-models-basic-fns}) we have that at almost all places $v$, $f_v = \1_{X(\mathcal{O}_v)} = \1_{G(\O_v) \star 1} \in \Cc(X(F_v))$.
        
    Moreover, for any fixed $x \in G'(\A)$, at almost all places, $x_v \in G'(\mathcal{O}_v)$, and the corresponding conjugation action preserves $X(\mathcal{O}_v)$. Thus, for almost all places $v$, \begin{align*}
        f_{x, v}(b_v) &= \1_{X(\mathcal{O}_v)}(x_v^{-1}\mat{1}{\tau^{1/2} b_v}{}{1}\theta(\mat{1}{\tau^{1/2} b_v}{}{1})^{-1}x_v) \\ &= \1_{X(\mathcal{O}_v)}(\mat{1}{\tau^{1/2} b_v}{}{1}\theta(\mat{1}{\tau^{1/2} b_v}{}{1})^{-1}) \\ &= \1_{\mathcal{O}_v}(b_v),
    \end{align*} as desired.
\end{proof}

\begin{remark}
    In particular, we see that if $f = \1_{X(\widehat{\O}_F)} \otimes f_{\infty}$ is unramified away from the infinite places, $f_x = \1_{\widehat{\O}_F} \otimes f_{x, \infty} \in \Cc(\A)$.
\end{remark}

It follows that $$k^T_{f, \mathfrak{o}^+}(x) = f(1) + \sum_{\delta \in B'(F) \backslash G'(F)} \sum_{b \in F^{\x}} f_{\delta x}(b) - \widehat{\tau}_B(H_{0'}(\delta x) - T)\widehat{f}_{\delta x}(0),$$ where $\widehat{f}_{\delta x} \in \Schw(\A)$ denotes the Fourier transform of $f_{\delta x} \in \Cc(\A)$.

We now describe some transformation properties of $f_x$ and its Fourier transform.

\begin{proposition}
    Fix $b \in \A$.
    \begin{enumerate}
        \item The function $x \mapsto f_x(b)$ extends to a function on $\GL_{2, F}(\A) \supset G'(\A)$, which satisfies $f_{zx}(b) = f_x(b)$ for all $z \in Z_{\GL_2}(\A)$.
        \item Let $x = ntk \in \PGL_{2, F}(\A)$ be the Iwasawa decomposition of an element in $\PGL_2(\A)$, where we identify $t \in \A^{\x}$ with its representative $\mat{t}{}{}{1} \in \PGL_2(\A)$.
        \begin{enumerate}
            \item $f_{ntk}(b) = f_{tk}(b)$ and $\widehat{f}_{ntk}(b) = \widehat{f}_{tk}(b)$.
            \item $f_{tk}(b) = f_k(b/t)$ and $\widehat{f}_{tk}(b) = \widehat{f}_{k}(at)|t|$.
        \end{enumerate}
    \end{enumerate}
\end{proposition}
\begin{proof}
    The first claim follows immediately from the the fact that the conjugation action of $G'(\A)$
    on $X(\A)$ extends to a conjugation action of $\GL_{2, F}(\A)$ on $X(\A)$. To see claim (a), note that $\mat{1}{\tau^{1/2} b}{}{1}\theta(\mat{1}{\tau^{1/2} b}{}{1})^{-1} \in X(\A)$ is stabilized by $n \in N_{B'}(\A)$. The first part of claim (b) follows immediately from the following computation
    $$t^{-1}\mat{1}{\tau^{1/2} b}{}{1}\theta(\mat{1}{\tau^{1/2} b}{}{1})^{-1}t = \mat{1}{\tau^{1/2} b/t}{}{1}\theta(\mat{1}{\tau^{1/2} b/t}{}{1})^{-1}.$$ The second part of the claim (b) follows from this same computation coupled with a change of variables.
\end{proof}

We also record some standard properties of Tate's global zeta function:

\begin{lemma}\
    Let $h \in \Schw(\A)$ be a Schwartz function, and let $\omega : F^{\x} \backslash \A^{\x} \to \C^{\x}$ be a unitary Hecke character. Then Tate's global zeta function $$Z(h, \omega |\cdot|^s) \coloneqq \int_{\A^{\x}} h(t)\omega(t)|t|^s\,dt$$ is entire at $s = 1$ unless $\omega$ is trivial, in which case $Z(h, |\cdot|^s)$ has a simple pole at $s = 1$ with residue $$\vol([\G_m]^1)\widehat{h}(0).$$
\end{lemma}
\begin{proof}
    This is well-known, and goes back to Tate's thesis.
\end{proof}

We are now ready to evaluate the unipotent contribution to the geometric RTF.

\begin{proposition}
    \begin{multline}
        \int_{[G']} k^T_{f, \mathfrak{o}^+}(x)\,dx = \vol([\SL_2])f(1) + \sum_{\substack{\kappa \in \widehat{[\G_m]}\\\kappa^2 = 1, \kappa \neq 1}} Z(f^{\kappa}_{\widetilde{K}}, \kappa|\cdot|^1)
        \\ + \lim_{s \to 0}\left[Z(f_{\widetilde{K}}, |\cdot|^{1 + s}) - \vol([\G_m]^1)\widehat{f}_{\widetilde{K}}(0)\frac{e^{sT}}{s}\right].
    \end{multline}
\end{proposition}
\begin{proof}
    Recall from our discussion above that we were able to rewrite the truncated unipotent kernel as $$k^T_{f, \mathfrak{o}^+}(x) = f(1) + \sum_{\delta \in B'(F) \backslash G'(F)} \sum_{b \in F^{\x}} f_{\delta x}(b) - \widehat{\tau}_B(H_{0'}(\delta x) - T)\widehat{f}_{\delta x}(0).$$ The contribution of $f(1)$ is clear. Thus, we are left to evaluate $$\int_{[G']} k^T_{f, \mathfrak{o}^+}(x) - f(1)\,dx.$$ Set $k^{T, *}_{f, \mathfrak{o}^+}(x) \coloneqq k^T_{f, \mathfrak{o}^+}(x) - f(1)$. To do so, we apply Fourier inversion relative to $\SL_2(\A) \subset \GL_2(\A)^1$ to rewrite the integral as $$\frac{1}{\vol([\G_m]^1)}\sum_{\kappa \in \widehat{[\G_m]^1}} \int_{\GL_2(F) \backslash \GL_2(\A)^1} k^{T, *}_{f, \mathfrak{o}^+}(x)\kappa(\det(x))\,dx.$$
    
    Here we have used the fact that $x \mapsto f_x(b)$ can be extended to a function on $\GL_{2, F}(\A)$. Moreover, note that $k^{T, *}_{f, \mathfrak{o}^+}(x)$ is invariant under the center of $\GL_2(\A)$. Thus, we see that the only unitary characters $\kappa : [\G_m]^1 \to \C^{\x}$ which contribute to this integral are quadratic, and we can further rewrite this integral as $$\frac{1}{\vol([\G_m]^1)}\sum_{\substack{\kappa \in \widehat{[\G_m]^1}\\ \kappa^2 = 1}}\int_{\GL_2(F) \backslash \GL_2(\A)^1} k^{T, *}_{f, \mathfrak{o}^+}(x)\kappa(\det(x))\,dx.$$ Next, the natural map $\GL_2(\A)^1 \to \PGL_2(\A)$ is surjective with kernel isomorphic to $\A^{\x, 1}$. It follows that we can further rewrite this integral as
    $$\sum_{\substack{\kappa \in \widehat{[\G_m]^1}\\ \kappa^2 = 1}}\int_{[\PGL_2]} k^{T, *}_{f, \mathfrak{o}^+}(x)\kappa(\det(x))\,dx.$$

    Notice that we can now rewrite the integrand entirely in terms of $\PGL_2$ and its related subgroups. Write $\widetilde{B}$ for the standard Borel subgroup of $\PGL_2$, with Levi decomposition $\widetilde{B} = M_{\widetilde{B}} N_{\widetilde{B}}$, and $\widetilde{K}$
    for the standard maximal compact subgroup. Then we can rewrite the integrand as $$k^{T, *}_{f, \mathfrak{o}^+}(x)\kappa(\det(x)) = \kappa(\det(x))\sum_{\delta \in \widetilde{B}(F) \backslash \PGL_2(F)} \sum_{b \in F^{\x}} f_{\delta x}(b) - \widehat{\tau}_B(H_{0'}(\delta x) - T)\widehat{f}_{\delta x}(0).$$ Indeed, $B'(F) \backslash G'(F) = (B' \backslash G')(F) = \mathbb{P}^1(F) = (\widetilde{B} \backslash \PGL_2)(F) = \widetilde{B}(F) \backslash \PGL_2(F)$. Also note that we can replace the sum over quadratic characters of $[\G_m]^1$ with quadratic characters of $[\G_m]$.

    Note that $H_{0'}(\delta x)$ descends (up to a constant factor) to the usual Harish-Chandra height for $\PGL_2$.
    
    There are two cases. First, suppose the quadratic Hecke character $\kappa$ is non-trivial.  
    Decompose $x$ according to the Iwasawa decomposition for $\PGL_2$ as $x = ntk$, with
    corresponding measure $dx = |t|^{-1}\,dn\,dt\,dk$. Using the transformation properties of $\widehat{f}_x$, we compute that \begin{multline}
        -\int_{\PGL_2(F) \backslash \PGL_2(\A)} \sum_{\delta \in \widetilde{B}(F) \backslash \PGL_2(F)} \kappa(\det(\delta x))\widehat{\tau}_B(H_{0'}(\delta x) - T)\widehat{f}_{\delta x}(0)\,dx \\ = -\int_{\widetilde{K}} \int_{[M_{\widetilde{B}}]} \int_{[N_{\widetilde{B}}]}
        \kappa(\det(tk))\widehat{\tau}_B(H_{0'}(t) - T)\widehat{f}_{ntk}(0)|t|^{-1}\,dn\,dt\,dk
        \\ = -\int_{\widetilde{K}} \kappa(\det(k))\widehat{f}_{k}(0) \int_{[M_{\widetilde{B}}]} \kappa(t)\widehat{\tau}_B(H_{0'}(t) - T)\,dt\,dk.
    \end{multline}
    Now, note that the inner-most integral factors through the integral of a non-trivial character over the compact abelian group $[M_{\widetilde{B}}]^1$. It follows that this integral vanishes.

    Using the Iwasawa decomposition again, we compute
    \begin{multline}
        \int_{\widetilde{K}} \int_{[M_{\widetilde{B}}]} \int_{[N_{\widetilde{B}}]} \sum_{b \in F^{\x}} f_{ntk}(b)\kappa(\det(tk))|t|^{-1}\,dn\,dt\,dk \\ = \int_{\widetilde{K}} \int_{F^{\x}\backslash\A^{\x}} \sum_{b \in F^{\x}} f_{tk}(b)\kappa(\det(k))\kappa(t)|t|^{-1}\,dt\,dk \\ = \int_{\widetilde{K}} \int_{\A^{\x}} f_k(1/t)\kappa(\det(k))\kappa(t^{-1})|t|^{-1}\,dt\,dk = Z(f^{\kappa}_{\widetilde{K}}, \kappa|\cdot|^1)
    \end{multline}
    where we've used the fact that $\kappa$ is quadratic in the second last equality. Here we have set \begin{equation}\label{f-kappa}
        f^{\kappa}_{\widetilde{K}}(b) \coloneqq \int_{\widetilde{K}} \kappa(\det(k))f_k(b)\,dk.
    \end{equation} The final equality follows from unfolding and the change of variables $t \mapsto t^{-1}$.

    It remains to understand the contribution of the trivial character. For this, we introduce a modified test function $f_x^s : \A \to \C$ defined by $$f_x^s(b) \coloneqq f_x(b)e^{-sH_{0'}(x)}.$$ This function is manifestly Schwartz, and is factorizable if $f_x$ is factorizable.
    
    For $\Re s > 0$, we use the same techniques as above to compute
    \begin{align*}
        -\int_{\PGL_2(F) \backslash \PGL_2(\A)} &\sum_{\delta \in \widetilde{B}(F) \backslash \PGL_2(F)} \widehat{\tau}_B(H_{0'}(\delta x) - T)\widehat{f}^s_{\delta x}(0)\,dx\\
         &= -\int_{\widetilde{K}} \widehat{f}_{k}(0) \int_{[M_{\widetilde{B}}]} \widehat{\tau}_B(H_{0'}(t) - T)e^{-sH_{0'}(t)}\,dt\,dk \\ &= - \vol([\G_m]^1) \widehat{f}_{\widetilde{K}}(0) \int_{A^{\infty}_{\widetilde{B}}} \widehat{\tau}_B(H_{0'}(a) - T)e^{-sH_{0'}(a)}\,da
        \\ &= -\vol([\G_m]^1) \widehat{f}_{\widetilde{K}}(0) \int_{T}^{\infty} e^{-sX}\,dX \\ &= -\vol([\G_m]^1) \widehat{f}_{\widetilde{K}}(0) \frac{e^{-sT}}{s}.
    \end{align*}

    The remainder is given by
    \begin{multline}
        \int_{\widetilde{K}} \int_{[M_{\widetilde{B}}]} \int_{[N_{\widetilde{B}}]} \sum_{b \in F^{\x}} f^s_{ntk}(b)|t|^{-1}\,dn\,dt\,dk = \int_{\widetilde{K}} \int_{\A^{\x}} f_k(1/t)|t|^{-1 - s}\,dt\,dk = Z(f_{\widetilde{K}}, |\cdot|^{1 + s}).
    \end{multline}
    Here we are using the fact that for $k \in \widetilde{K}$, $f_k^s = f_k$.

    Thus, when $\Re s > 0$, the contribution from the trivial character is given by
    $$Z(f_{\widetilde{K}}, |\cdot|^{1 + s}) - \vol([\G_m]^1) \widehat{f}_{\widetilde{K}}(0) \frac{e^{-sT}}{s}.$$ This expression is meromorphic in $s$ with potential poles at $s = -1, 0$. By our recollection from Tate's thesis, we compute that the residue at $s = 0$ is zero. The value of this function at $s = 0$ is equal to the integral we were trying to get our hands on. Thus, the contribution of the trivial character is precisely \begin{equation}\label{cancellation}
        \lim_{s \to 0} \left[Z(f_{\widetilde{K}}, |\cdot|^{1 + s}) - \vol([\G_m]^1) \widehat{f}_{\widetilde{K}}(0) \frac{e^{-sT}}{s}\right],
    \end{equation} and this concludes our calculation.
\end{proof}

\begin{corollary}
    With the notation as in the previous proposition, 
    $$J_{\mathfrak{o}^+}(f) = \vol([\SL_2])f(1) + \sum_{\substack{\kappa \in \widehat{[\G_m]}\\\kappa^2 = 1, \kappa \neq 1}} Z(f^\kappa_{\widetilde{K}}, \kappa|\cdot|^1) + \left.\frac{d}{ds} sZ(f_{\widetilde{K}}, |\cdot|^{1 + s})\right|_{s = 0}.$$
\end{corollary}
\begin{proof}
    The power series expansion of the contribution of the trivial character about $s = 0$ is given by
    $$\gamma + \vol([\G_m]^1)\widehat{f}_{\widetilde{K}}(0)T + O(s),$$ where $\gamma$ is the constant
    coefficient in the expansion of $Z(f_{\widetilde{K}}, |\cdot|^{1 + s})$ about $s = 0$. Since none of the other terms in $J^T_{\mathfrak{o}^+}(f)$ depend on $T$, we win. 
\end{proof}

The remaining unipotent datum $\mathfrak{o}^-_{unip}$ can be handled entirely analogously. We simply state the relevant changes. Choose $\gamma_0 \coloneqq \mat{\tau^{1/2}}{}{}{\tau^{-1/2}} \in G(F)$. Then one easily verifies that $\gamma_0\theta(\gamma_0)^{-1} = -1$, and moreover that the image of the $\theta$-twisted action of $\gamma_0$ on $\mathfrak{o}^+$ is precisely $\mathfrak{o}^-$. For a test function $f \in \Cc(X(\A))$ as above, define \begin{equation}\label{f-prime}
    f'_x(b) \coloneqq f(x^{-1}(\gamma_0 \star \mat{1}{\tau^{1/2} b}{}{1})x).
\end{equation} One checks that this function transforms exactly like $f_x$. Then one can verify that the contribution of the datum $\mathfrak{o}^-$ for the test function $f$ is the same as the contribution of $\mathfrak{o}^+$, except where every instance of $f_x$ has been replaced with $f'_x$. Finally, note that $f'_x(0) = f(-1)$.

We record the final formula in the following proposition.

\begin{proposition}
    Let $f \in \Cc(X(\A))$. Then
    $$J_{\mathfrak{o}^-}(f) = \vol([\SL_2])f(-1) + \sum_{\substack{\kappa \in \widehat{[\G_m]}\\\kappa^2 = 1, \kappa \neq 1}} Z({f}'^\kappa_{\widetilde{K}}, \kappa|\cdot|^1) + \left.\frac{d}{ds} sZ(f'_{\widetilde{K}}, |\cdot|^{1 + s})\right|_{s = 0},$$ where ${f}'^\kappa_{\widetilde{K}}$ is defined as in \cref{f-kappa}, except every instance of $f_?$ has been replaced by $f'_?$ as defined in \cref{f-prime}.
\end{proposition}

\subsection{A Geometric Viewpoint}\label{geometric-viewpoint}

Understanding the unipotent contributions to the geometric side of the relative trace formula is a difficult problem, with no clear conceptual resolution.

Here we suggest a geometric viewpoint which motivated the calculations above. The basic idea is that there is a very close relationship between the truncated unipotent kernel $k^T_{f, \o^+}$ and the geometry of the following ``correspondence'':

\begin{center}
    \begin{tikzcd}
	& {T^*(B'\backslash G')} \\
	\mathbb{P}^1 = {B'\backslash G'} && {\overline{\o^{+, \reg}} = \mathcal{N}} %& {G' \cdot \eta_b}
	\arrow[from=1-2, to=2-1]
	\arrow[from=1-2, to=2-3]
	%\arrow[hook', from=2-4, to=2-3]
\end{tikzcd}
\end{center}

The map on the left is the standard projection and the map on the right is the moment map for $T^*(B'\backslash G')$ as a $G'$-space, i.e. the Springer resolution of the nilpotent cone.

Away from the zero orbit, the Springer resolution is an isomorphism, and the diagram above yields a genuine correspondence $\o^{+, \reg} \dashrightarrow B'(F)\backslash G'(F)$. Organizing the points in $\o^{+, reg}$ according to their images under this correspondence is precisely the manipulation that allows one to rewrite $$\sum_{\eta \in \o^{+, \reg}}f(x^{-1}\eta x) \;\;\rightsquigarrow \;\; \sum_{\delta \in B'(F) \backslash G'(F)} \sum_{b \in F^{\x}} f(x^{-1} \delta^{-1} \mat{1}{\tau^{1/2} b}{}{1} \delta x).$$ Indeed the map $\nu$ described above is simply the inverse to the Springer resolution away from zero, composed with the projection onto $\mathbb{P}^1$.

We can now interpret the inner-most sum over $b \in F^{\x}$ as a sum over the rational points in (the open $\G_m$-orbit in) the fibers of the rank one vector bundle $T^*(B' \backslash G') \to B' \backslash G'$. In our calculations, the fibers of this bundle were identified with $N_B^{opp} \simeq \A^1$.

At zero, this diagram fails to give a correspondence; the fiber of the Springer resolution is a copy of $\mathbb{P}^1 \subset T^*(B' \backslash G')$ embedded as the zero section of $\mathbb{P}^1 = B' \backslash G' \into T^*(B' \backslash G')$.

From this point of view, we can interpret the integral of $k^{T, *}_{f, \o^+}$ as a kind of regularized period associated to a theta series on the $M_{B'}$-torsor over $\mathbb{P}^1$ given by $$T^*(B' \backslash G') - \mathbb{P}^1 \to \mathbb{P}^1,$$ where the $M_{B'}$-action is by fiberwise \emph{squaring}. The problem with this $M_{B'}$-torsor is that it is not trivial. However, by identifying $T^*(B' \backslash G')$ with $T^*(\widetilde{B} \backslash \PGL_2)$ we see that the fibers of the projection onto $\mathbb{P}^1$ acquire the natural $\G_m \subset \PGL_2$ action for which we do obtain a trivial torsor $T^*(\widetilde{B} \backslash \PGL_2) - \mathbb{P}^1 \to \mathbb{P}^1$. The final formula now reflects the fact that (via Fourier inversion) this regularized theta series can be interpreted as a theta series for $\PGL_2$, and it unfolds to a sum of $\G_m$-periods over the fibers of $T^*(\widetilde{B} \backslash \PGL_2) - \mathbb{P}^1 \to \mathbb{P}^1$, i.e. Tate zeta integrals.

This geometric viewpoint interacts ``correctly'' with the machinery of relative truncation to produce the expected period integrals. In the calculation above, we see this in the evaluation of the contribution of the trivial character, where the pole of the unramified zeta integral is precisely cancelled by the pole arising from truncation, see \cref{cancellation}.

This suggests a similar perspective might be fruitful to study regularized unipotent terms in geometric RTF for higher rank Galois symmetric pairs $(H, H')$, especially in type A, where every nilpotent orbit over the algebraic closure can be realized as a dense open in the image of a moment map from $T^*(Q' \backslash H')$, for some class of parabolic subgroup $Q'$. (Nilpotent orbits with this property are called Richardson.)

\begin{remark}
    The idea of understanding unipotent terms in the usual trace formula in terms of zeta integrals has been suggested and explicated by Hoffmann and Hoffmann-Wakatsuki  \cite{hoffmann, hoffmann-wakatsuki}. See also the work of Chaudouard \cite{chaudouard-unip-1, chaudouard-unip-2} where substantial progress is made in understanding the global coefficients appearing in the unipotent terms of the fine spectral expansion of the usual trace formula for $\GL_n$ in terms of zeta integrals. What seems to be new here is the idea of organizing the unfolding of the unipotent orbital integral using the geometry of (generalized) Springer maps. We thank Spencer Leslie for bringing these references to our attention.
\end{remark}

\begin{remark}
    In forthcoming joint work with Spencer Leslie, we prove relative smooth transfer and a relative fundamental lemma for the Galois symmetric space associated to $(\Res_{E/F}\SL_{2, E}, \SL_{2, F})$. Write $X_\e$ for the endoscopic symmetric space for $X$ indexed by $\e$, an endoscopic datum corresponding to a quadratic étale $F$-algebra. Denote the transfer by $f \mapsto f^{\e}$. We show that $f^{\e}(\pm 1)$ agrees with the contributions to $J_{\o^{\pm}}(f)$ arising from the quadratic character associated to $\e$, in analogy with the results in \cite{labesse-langlands}.
\end{remark}

\newpage

\part*{Appendices}
\appendix

\section{Explicit Truncation Calculations for $\SL_2$}\label{explicit-trunc-calcs}

\subfile{sl2conecalc.tex}
\newpage

\section{The Relative Trace Formula for Galois Periods of Tori}\label{rtf-for-rk-one-tori}

\subfile{galoisrkone.tex}

\bibliographystyle{plain}
\bibliography{main}

\end{document}

%% file: sl2conecalc.tex
\ifSubfilesClassLoaded{%
\section{Explicit Truncation Calculations for $\SL_2$}%
}{}

In this appendix we specialize our discussion from \cref{setup} to the Galois symmetric pair associated to $\SL_2$. Recall the notation from \cref{part-two}. Our definitions are taken directly from \cite{zydor}.

\subsection{Chambers, Relative Chambers}

Recall that the pairing on $\a_{0'} = \a_{0} = \R$ is simply given by multiplication. Using the notation from \cref{chambers}, \cref{rel-chambers}, we easily calculate $\a_{G} = \{ 0 \}$, $\a_{G}^+ = \{ 0 \}$, and
$\a_{B} = \R$, $\a_{B}^+ = \R_{>0}$. Thus, $\varepsilon_G^G = 1$,
$\varepsilon_{B}^G = -1$ and $\varepsilon_{B}^{B} = 1$.

\subsection{Various Indicator Functions}

\begin{proposition}\
    \begin{enumerate}
        \item $\widehat{\tau}_G^G = [\R]$
        \item $\widehat{\tau}_B^B = [\{0\}]$
        \item $\widehat{\tau}_B^G = [\R_{> 0}]$.
    \end{enumerate}
\end{proposition}
\begin{proof}
        By definition, $\widehat{\tau}^G_G = [\rint A(\overline{\a_G^+},
    \overline{\a_G^+})^{\vee}]$ and $\widehat{\tau}^{B}_{B} = [\rint
    A(\overline{\a_{B}^+}, \overline{\a_{B}^+})^{\vee}]$. \quash{By
    \cref{easy-angle-cone-computation},}We have $A(\overline{\a_G^+},
    \overline{\a_G^+})^{\vee} = (\{ 0 \})^{\vee} = \a_{0'}$ and thus
    $\widehat{\tau}^G_G$ is identically one. Similarly, $A(\overline{\a_{B}^+},
    \overline{\a_{B}^+})^{\vee} = (\a_{0'})^{\vee} = \{ 0 \}$. Finally $[\rint\{ 0 \}] = [\{ 0 \}]$. Thus $\widehat{\tau}^{B}_{B}$ is equal to $[\{ 0 \}]$.

    Next, $\widehat{\tau}_{B}^G = [\rint A(\overline{\a_G^+},
    \overline{\a_{B}^+})^{\vee}]$, and noting that $\overline{\a_G^+} = \{ 0
    \}$ is the minimal face of $\a_{0'}$, \quash{by
    \cref{easy-angle-cone-computation},} we have $A(\overline{\a_G^+},
    \overline{\a_{B}^+})^{\vee} = (\overline{\a_{B}^+})^{\vee} =
    \overline{\a_{B}^+}$. It follows that $\widehat{\tau}_{B}^G =
    [\rint\overline{\a_{B}^+}] = [\a_{B}^+] = [\R_{>0}]$.
\end{proof}

%TODO: Write out $H_{0'}$ explicitly in the dual-basis coordinate.

We also record some of the intermediate computations above for future use.

\begin{proposition}\
    \begin{enumerate}
        \item $A(\overline{\a_G^+}, \overline{\a_G^+}) = \{0\}$.
        \item $A(\overline{\a_{B}^+}, \overline{\a_{B}^+}) = \a_{0'} = \R$.
        \item $A(\overline{\a_G^+}, \overline{\a_{B}^+}) = \overline{\a_{B}^+} = \R_{\geq 0}.$
    \end{enumerate}
\end{proposition}

We now record the corresponding $\sigma$ and $\Gamma$ functions.

\begin{proposition}\label{sigma-calculation} For $H \in \a_{0'}$,
    \begin{enumerate}
        \item $\sigma_G^G(H) = [\{0\}](H)$,
        \item $\sigma_B^B(H) = 0$,
        \item $\sigma_B^G(H) = [\R_{>0}](H)$.
    \end{enumerate}
\end{proposition}
\begin{proof}

    By definition, for any $P, Q \in \F^G(P_0')$ such that $P \subset Q$, we
    have $\sigma_P^Q(H) = \sigma(\overline{\a_Q^+}, \overline{\a_P^+})(H)$,
    where $$\sigma(F, C) \coloneqq \sum_{E \subset F} \varepsilon_F^E[\rint
    A(E, C)][\rint E^\vee].$$ Now, in general, the faces of the cone
    $\overline{\a_P^+}$ are in bijection with the $Q \in \F^G(P) \cap
    \F^G(P_0')$ via the map $Q \mapsto \overline{\a_Q^+}$. Thus, the faces of
    $\overline{\a_G^+}$ are itself (as $\F^G(G) \cap \F^G(P_0') = \{ G \}$),
    and the faces of $\overline{\a_B^+}$ are itself and $\overline{\a_G^+}$ (as
    $\F^G(B) \cap \F^G(P_0') = \{ G, B \}$). Thus,
    \begin{align*}
        \sigma_G^G &= 1\cdot[\rint A(\overline{\a_G^+}, \overline{\a_G^+})][\rint (\overline{\a_G^+})^\vee] = 1\cdot[\{0\}][\R] = [\{0\}]\\
        \sigma_B^B &= 1\cdot[\rint A(\overline{\a_B^+}, \overline{\a_B^+})][\rint (\overline{\a_B^+})^\vee] -1 \cdot [\rint A(\overline{\a_G^+}, \overline{\a_B^+})][\rint (\overline{\a_G^+})^\vee] = 0\\
        \sigma_B^G &= 1 \cdot [\rint A(\overline{\a_G^+}, \overline{\a_B^+})][\rint (\overline{\a_G^+})^\vee] = 1\cdot[\R_{>0}][\R] = [\R_{>0}],
    \end{align*}
    as desired.
\end{proof}

\begin{proposition}\label{gamma-calc}
    For $H, X \in \a_{0'}$,
    \begin{enumerate}
        \item $\Gamma_G^G(H, X) = [\{ 0 \}](H)$,
        \item $\Gamma_B^B(H, X) = [\{ 0 \}](H - X)$,
        \item $\Gamma_B^G(H, X) = [\R_{>0}](H) - [\R_{>0}](H - X)$.
    \end{enumerate}
\end{proposition}

\begin{proof}

    By definition, for any $P, Q \in \F^G(P_0')$ such that $P \subset Q$ and
    $H, X \in \a_{0'}$, $\Gamma_P^Q(H, X) = \Gamma(A(\overline{\a_Q^+},
    \overline{\a_P^+}), H, X)$, where $$\Gamma(C, H, X) \coloneqq \sum_{F \subset C} \varepsilon_F^{F_0}[\rint A(F, C)](H)[\rint F^\vee](H - X).$$ In
    general, there is an inclusion-preserving bijection between faces of $A(F,
    C)$ and faces of $C$ containing $F$, given by sending $G \subset C$ such
    that $G \supset F$ to $A(F, G)$. Thus, $A(\overline{\a_G^+},
    \overline{\a_G^+})$ and $A(\overline{\a_B^+}, \overline{\a_B^+})$ have no
    proper faces, while $A(\overline{\a_G^+}, \overline{\a_B^+})$ has a single
    proper face $A(\overline{\a_G^+}, \overline{\a_G^+})$. It is also a general
    fact that $A(A(F, G), A(F, C)) = A(G, C)$. Thus, we compute
    \begin{align*}
        \Gamma_G^G(H, X) &= 1 \cdot [\rint A(\overline{\a_G^+}, \overline{\a_G^+})](H)[\rint A(\overline{\a_G^+}, \overline{\a_G^+})^\vee](H - X)\\ &= [\{ 0 \}](H)[\R](H - X) \\ &= [\{ 0 \}](H).
    \end{align*}
    Similarly,
    \begin{align*}
        \Gamma_B^B(H, X) &= 1 \cdot [\rint A(\overline{\a_B^+}, \overline{\a_B^+})](H)[\rint A(\overline{\a_B^+}, \overline{\a_B^+})^\vee](H - X)\\
        &= [\R](H)[\{ 0 \}](H - X)\\
        &= [\{ 0 \}](H - X).
    \end{align*}
    Finally,
    \begin{align*}
        \Gamma_B^G(H, X) &= 1 \cdot [\rint A(\overline{\a_G^+}, \overline{\a_B^+})](H)[\rint A(\overline{\a_G^+}, \overline{\a_G^+})^\vee](H - X) \\ &\phantom{=} -1 \cdot [\rint A(\overline{\a_B^+}, \overline{\a_B^+})](H)[\rint A(\overline{\a_G^+}, \overline{\a_B^+})^\vee](H - X)\\
        &= [\R_{>0}](H)[\R](H - X) - [\R](H)[\R_{>0}](H - X)\\
        &= [\R_{>0}](H) - [\R_{>0}](H - X),
    \end{align*}
    so we win.
\end{proof}

%\subsection{Mixed Truncation Operators}

%\begin{proposition}
%    The mixed truncation operators are given by \begin{align*}
%        \Lambda^T\phi(x) &= \phi(x) - \sum_{\delta \in B'(F) \backslash G'(F)} [\R_{>0}](H_{0'}(\delta x) - T)\phi_{B}(\delta x)\\
%        \Lambda^{T, B}\phi(x) &= \phi_{B}(x).
%    \end{align*}
%\end{proposition}
%\begin{proof}
%    This follows immediately from the propositions in the previous section.
%\end{proof}

\end{document}

%% file: galoisrkone.tex
\ifSubfilesClassLoaded{%
\section{The Relative Trace Formula for Galois Periods of Tori}%
}{}

Let $E/F$ be a quadratic extension of number fields. This appendix is devoted to the statement and proof of the relative trace formula for Galois periods of tori of rank one. We begin by classifying the $E$-tori of $E$-rank one which admit Galois involutions. We then classify the corresponding Galois symmetric pairs, and describe the adelic and rational points of the associated Galois symmetric space. Finally, we state and prove the relative trace formula for such Galois symmetric spaces.

\subsection{Rank One Tori and Galois Symmetric Spaces}

\subsubsection{Classification of Rank One Tori}

We recall the classification of rank one tori over $E$. The proposition below is well-known; we recall the proof for the convenience of the reader.

\begin{proposition}
    A rank one torus over $E$ is either isomorphic to $\G_{m, E}$ (and thus split), or the kernel of the norm map from a quadratic field extension $M/E$, $$H = \ker[\Nm_{M/E} : \Res_{M/E} \G_{m, M} \to \G_{m, E}],$$ and thus anisotropic.
\end{proposition}

\begin{proof}
    For any field $K$, the endomorphisms of the split $K$-torus of rank one are isomorphic to $\Z$, so the automorphisms of the split $K$-torus are precisely $\{ \pm 1 \}$. A rank one torus over $E$ is by definition an $E$-form of $\G_{m, E}$. The $E$-forms of $\G_{m, E}$ are determined (up to isomorphism) by $H^1(E, \Aut_E\G_{m, E}) = H^1(E, \{ \pm 1 \})$ so we are left to compute this Galois cohomology group. Here $\{ \pm 1 \}$ has the trivial $\Gal(\overline{E}/E)$ action, because $\Gal(\overline{E}/E)$ acts trivially on $\Aut_E\G_{m, E}$. Thus, $H^1(E, \{ \pm 1 \}) = \Hom_{cts}(\Gal(\overline{E}/E), \{ \pm 1 \})$; this later module is in correspondence with quadratic extensions of $E$ via the bijection between continuous quadratic characters $\chi : \Gal(\overline{E}/E) \to \{ \pm 1 \}$ and $\ker \chi$. The trivial character clearly corresponds to the split torus over $E$. A short calculation shows that the remaining $E$-tori of rank one are as described (choose a basis for the quadratic extension $M/E$ and realize the torus in $\GL_{2, E} \simeq \GL(M)$).
\end{proof}

We can arrive at a more uniform statement of the proposition by also considering the unique split quadratic étale $E$-algebra $M = E \x E$ and interpreting $\Res_{M/E} \G_{m, M}$ as $\G_{m, E} \x \G_{m, E}$, with norm map $\Nm_{M/E} : \Res_{M/E} \G_{m, M} = \G_{m, E} \x \G_{m, E} \to \G_{m, E}$ given by multiplying both factors. In this way, we can understand the split $E$-torus as $\ker[\Nm_{M/E} : \Res_{M/E} \G_{m, M} \to \G_{m, E}]$, i.e. the copy of $\G_{m, E}$ in $\G_{m, E} \x \G_{m, E}$ embedded via the anti-diagonal map $x \mapsto (x, x^{-1})$.

For $M/E$ an étale quadratic $E$-algebra, write $\Nm^1_{M/E}$ for the $E$-torus given by the kernel of the norm map $\Nm_{M/E} : \Res_{M/E} \G_{m, M} \to \G_{m, E}$, $$\Nm^1_{M/E} \coloneqq \ker[\Nm_{M/E} : \Res_{M/E} \G_{m, M} \to \G_{m, E}].$$

\begin{proposition}
    The isomorphism classes of rank one tori over $E$ are in correspondence with étale quadratic $E$-algebras; the torus corresponding to $M/E$ is $\Nm^1_{M/E}$.
\end{proposition}

\subsubsection{Galois Involutions on Tori}

Recall the definition of Galois structure from \cref{galois-symmetric-space-review}.

%The collection of Galois structures on $H$ forms a groupoid with the obvious notion of isomorphism of Galois involution.

If $M/E$ is a quadratic étale $E$-algebra, a Galois structure on $T = \Res_{E/F} \Nm_{M/E}^1$ is the data of an $F$-group $T'$ and an isomorphism such that $$T \simeq \Res_{E/F} T'_E.$$ Since $T_E \simeq T'_E \x T'_E$ and $T$ has rank two, we see that $T'$ must be an $F$-torus of rank one.

\begin{figure}[b!]
    \centering
    \begin{tikzcd}
                                       & M \arrow[ld, no head] \arrow[d, no head] \arrow[rd, no head] &                                 \\
    L = M^{\sigma} \arrow[rd, no head] & E = M^{\sigma\tau} \arrow[d, no head]                        & L' = M^\tau \arrow[ld, no head] \\
                                       & F                                                            &                                
    \end{tikzcd}
    \caption{A biquadratic extension of $M/F$ with Galois group $\Gal(M/F) = \{ 1, \sigma, \tau, \sigma\tau = \tau\sigma \}$. We say $L'$ is the reflection of $L$ through $E$. }
    \label{biquadratic-ext}
\end{figure}
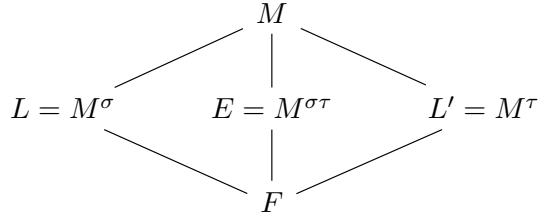

\begin{proposition}\label{rank-one-tori-galois-symmetric-space}
    Let $M/E$ be a quadratic étale $E$-algebra.
    \begin{enumerate}
        \item When $M$ is field extension, $\Res_{E/F} \Nm^1_{M/E}$ admits a Galois structure if and only if $M/F$ is biquadratic. In this case, there are two non-isomorphic Galois structures, given by the Galois symmetric pairs $$(\Res_{E/F} \Nm^1_{M/E}, \Nm^1_{L/F}) \text{ and } (\Res_{E/F} \Nm^1_{M/E}, \Nm^1_{L'/F}),$$ where $L, L'$ are the two quadratic subextensions of $M/F$ different from $E$. See \cref{biquadratic-ext}. 
        \item When $M$ is split, i.e. $M = E \x E$, $\Res_{E/F} \Nm^1_{M/E} = \Res_{E/F} \G_{m, E}$ admits two non-isomorphic Galois structures, given by the Galois symmetric pairs $$(\Res_{E/F} \G_{m, E}, \Nm^1_{E/F}) \text{ and } (\Res_{E/F} \G_{m, E}, \G_{m, F}).$$
    \end{enumerate}
\end{proposition}
\begin{proof}
    Fix a quadratic étale $E$-algebra $M/E$ and set $T \coloneqq \Res_{E/F} \Nm^1_{M/E}$. By our remark above, any $T'$ such that $T \simeq \Res_{E/F} T'_E$ must be an $F$-torus of rank one. Thus, let $L/F$ be a quadratic étale $F$-algebra. Now, it is well-known that $$\Nm^1_{L/F} \x_F E \simeq \Nm^1_{L \otimes_F E/E}.$$ It follows that for every quadratic étale $F$-algebra $L$ such that $L \otimes_F E = M$, we obtain a Galois structure. It is easy to see that the only étale $E$-algebras you obtain via $L \otimes_F E$ are split (when $L$ is split, or $L = E$) or biquadratic over $F$ and containing $E$ (when $L$ is a field distinct from $E$). Indeed, when $L$ is a field distinct from $E$, $L \otimes_F E$ will not be a field unless the natural multiplication map $L \otimes_F E \to L \cdot E$ is an isomorphism. Thus we have $M = L \otimes_F E$ precisely when $M$ is biquadratic and $L$ is a complimentary quadratic subfield to $E \subset M$. This covers the non-split cases. In the split cases, $L = F \x F$ or $L = E$; when $L = F \x F$, $\Nm_{L/F}^1 = \G_{m, F}$; when $L = E$, $\Nm_{L/F}^1 = \Nm_{E/F}^1$, the norm torus associated to $E/F$.

    To complete the proof, it suffices to notice that
    given an isomorphism $$\Res_{E/F} \Nm_{M/E}^1 \simeq \Res_{E/F} \Nm^1_{M'/E}$$ with $M/E$ arising from $L \otimes_F E$ as above, and $M'/E$ arbitrary, we necessarily obtain an isomorphism $\Nm_{M/E}^1 \simeq \Nm^1_{M'/E}$. Indeed, if we write $c \in \Gal(E/F)$ for the non-trivial element, after base change to $E$ we obtain an isomorphism $$\Nm_{M/E}^1 \x \Nm_{M/E}^1 \simeq \Nm^1_{M'/E} \x c^*\Nm^1_{M'/E}.$$ Let $\chi_{M/E}, \chi_{M'/E} : \Gal(\overline{F}/E) \to \{ \pm 1\}$ be the quadratic characters associated to $\Nm_{M/E}^1$ and $\Nm_{M'/E}^1$. Passing to character lattices, we obtain two $\Gal(\overline{F}/E)$-representations of rank two; the one on the left is isotypic for $\chi_{M/E}$ and the one on the right is a sum of $\chi_{M'/E}$ and $\chi_{c(M')/E}$. Since these modules are isomorphic, we must have that $\chi_{M/E} = \chi_{M'/E} = \chi_{c(M')/E}$, and we conclude that $\Nm_{M/E}^1 \simeq \Nm^1_{M'/E}$ as desired.
\end{proof}

Again, we can unify the two cases above by considering quartic étale $F$-algebras $M/F$, which contain $E/F$ and satisfy $\Aut_F(M) \simeq \Z/2 \x \Z/2$. We call such generalized biquadratic extensions.

\begin{proposition}\label{biquadratic-tori}
    Let $M/E$ be a quadratic étale $E$-algebra. Set $T = \Res_{E/F} \Nm^1_{M/E}$.
    \begin{enumerate}
        \item $T$ admits a Galois structure if and only if $M/F$ is a generalized biquadratic extension.
        \item If $M/F$ is a generalized biquadratic extension, $T$ admits two non-isomorphic Galois structures, corresponding to the two quadratic étale $F$-subalgebras $L$, $L'$ of $M/F$ different from the copy of $E$ in $M$, given by the Galois symmetric pairs $$(\Res_{E/F} \Nm^1_{M/E}, \Nm^1_{L/F}) \text{ and } (\Res_{E/F} \Nm^1_{M/E}, \Nm^1_{L'/F}).$$
    \end{enumerate}
\end{proposition}

If $M/F$ is a generalized biquadratic extension, and $L/F$ is a quadratic étale $F$-subalgebra fixed by the automorphism $\sigma \in \Aut_F(M)$, the Galois involution associated to the pair $(\Res_{E/F} \Nm^1_{M/E}, \Nm^1_{L/F})$ is given informally by restricting $\sigma$ to $\Res_{E/F} \Nm^1_{M/E}$.

% TODO: Now also allow $E/F$ to be an étale quadratic extension; the so-called group case, i.e. $E = F \x F$. In that situation, (in the non-split case) $M/E$ should be relative rank 2, so $M = L_1 \x L_2$ for two quadratic extensions of $F$; I suspect this  has a $F$-struct when $L_1 = L_2 = L$. In that case, the norm tori coming from $M/E$ should look like two copies of norm one elements of $L_1 = L_2 = L$, and the Galois involution will simply swap the two factors, so the fixed points will just be a copy of the norm one elements of $L$ embedded diagonally. In the split case, $M = E \x E$ again, which will end up giving a copy of $\Res_{E/F} \G_{m, E} = \G_{m, F} \x \G_{m, F}$ as the norm torus; this has a canonical $F$-struct, and I believe Galois involution will simply be conjugation, so the Galois fixed points will be $E^1$.

\quash{TODO: How to differentiate which case you're in when you're embedded in a larger group? This must be related to roots of the larger group and how the character group of the torus embeds into the character group of a maximal torus containing it.}

\quash{
\begin{remark}
    I think the arguments given above also work for quadratic étale $\mathcal{O}_E$-algebras; i.e. for a fixed quadratic extension $E/F$, a norm torus associated to a quadratic étale $\mathcal{O}_E$-algebra $M$ has a canonical Galois involution/is defined over $\mathcal{O}_F$ if and only if $M = L \otimes_{\mathcal{O}_F} \mathcal{O}_E$, where $L$ is some quadratic étale $\mathcal{O}_F$-algebra. This might be useful if unramified orbital integrals in the torus case end up being some kind of lattice point counts.
\end{remark}
}

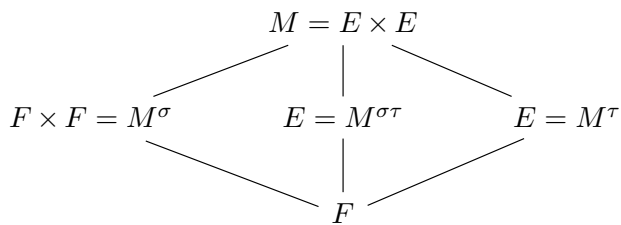
\begin{figure}[h!]
    \centering
    \begin{tikzcd}
                                       & M = E \x E \arrow[ld, no head] \arrow[d, no head] \arrow[rd, no head] &                                 \\
    F \x F = M^{\sigma} \arrow[rd, no head] & E = M^{\sigma\tau} \arrow[d, no head]                        & E = M^\tau \arrow[ld, no head] \\
                                       & F                                                            &                                
    \end{tikzcd}
    \caption{A generalized biquadratic extension $M = E \x E$ with automorphism group $\Aut(M/F) = \{ 1, \sigma, \tau, \sigma\tau = \tau\sigma \}$. Here $\sigma$ is component-wise Galois conjugation, and $\tau$ swaps coordinates. Note that while $M$ contains two copies of $E$, they are embedded differently (diagonally vs antidiagonally).}
    \label{generalized-biquadratic-ext}
\end{figure}

The following terminology will be useful going forward.

\begin{definition}\label{reflection}
    Let $L/F$ be a quadratic étale $F$-algebra. The reflection of $L$ through $E$ is the quadratic étale $F$-algebra $L'$ obtained by forming the generalized biquadratic extension $M \coloneqq L \otimes_F E$, and taking fixed points for the unique involution in $\Aut_F(M)$ not corresponding to $E$ or $L$. See \cref{biquadratic-ext}.
\end{definition}

We denote the reflection of $L$ through $E$ by a prime superscript, $L'$.

It is clear from the definition that reflection through $E$ is an involution on the set of all quadratic étale $F$-algebras, which preserves the set of quadratic field extensions distinct from $E$ and interchanges $E$ and the split algebra $F \x F$.

\subsubsection{Galois Symmetric Spaces from Tori}\label{galois-symmectric-spaces-from-tori}

Let $(T, T')$ be a Galois symmetric pair associated to the generalized biquadratic extension $M/F$ and quadratic étale subextension $L/F$. Write $\theta$ for the corresponding Galois involution, and write $S = T/T'$ for the corresponding Galois symmetric space. By \cref{concrete-description-of-symm-spc}, under the symmetrization map we can identify $S$ with $S = \{ t \in T : t\theta(t) = 1\}$. Recall that $T$ acts on $S$ by $\theta$-twisted conjugation: $t \star x \coloneqq tx\theta(t)^{-1}$, for $t \in T, x \in S$. We now compute $S$ in terms of $M/F$ and $L/F$.

\begin{proposition}
    Let $(T, T')$ be a Galois symmetric pair associated to the generalized biquadratic extension $M/F$ and quadratic étale subextension $L/F$. Then the corresponding Galois symmetric space $S$ is given by $$S = \Nm^1_{L'/F}$$ where $L'/F$ is the quadratic étale $F$-subalgebra distinct from $E, L$ in $M$, i.e. $L'$ is the reflection of $L$ through $E$.
\end{proposition}

\begin{proof}
     Set $T = \Res_{E/F} \Nm^1_{M/E}$, $T' = \Nm^1_{L/F}$, and consider the following commutative diagram of $F$-algebraic groups, where all rows and columns are exact:

    \begin{center}
        \begin{tikzcd}
        {\G_{m, F}} \arrow[r]                                                  & {\Res_{E/F} \G_{m, E}} \arrow[r]           & \Nm^1_{E/F}                         \\
        {\Res_{L/F} \G_{m, L}} \arrow[u] \arrow[r] & {\Res_{M/F} \G_{m, M}} \arrow[u] \arrow[r] & \Res_{L/F}\Nm^1_{M/L} \arrow[u] \\
        \Nm^1_{L/F} \arrow[u] \arrow[r]                          & \Res_{E/F}\Nm^1_{M/E} \arrow[u] \arrow[r]  & S \arrow[u]
        \end{tikzcd}
    \end{center}
    
    All of these maps can be understood as arising from various Galois involutions on $\Res_{M/F} \G_{m, M}$. Indeed, suppose $\Aut_F(M) = \{ 1, \sigma, \rho, \sigma\rho \}$, and moreover suppose $\sigma$ is the non-trivial automorphism fixing $L/F$. Write $L'/F$ for the remaining quadratic subextension of $M$ distinct from $E$ and $L$. It is fixed by $\rho$. By a mild abuse, let us write these same symbols to denote the corresponding automorphism of $\Res_{M/F} \G_{m, M}$.
    
    The vertical maps in the bottom-most row are given by inclusion, and the vertical maps in the top-most row of the diagram above can be summarized as $m \mapsto m\rho\sigma m$. The horizontal maps in the first column are given by inclusions, and the horizontal maps in the second column can be summarized as $m \mapsto m/\rho\sigma m$.
    
    Thus, we see that $S$ can be realized as $$S = \{ m \in \Res_{M/F} \G_{m, M} : m\sigma m = 1, m\rho\sigma m = 1 \},$$ and in turn, this can easily be identified with $$\Nm_{L'/F}^1 = \{ m \in \Res_{M/F} \G_{m, M} : m\sigma m = 1, \rho m = m \}.$$
\end{proof}

From the exact sequence $$0 \to T' \to T \to S \to 0$$ we obtain various long exact sequences in Galois cohomology relating $T(F) \star 1$ and $S(F)$, and $T(\A) \star 1$ and $S(\A)$. The rational (resp. adelic) orbits of $T$ on $S$ are controlled by the Galois cohomology groups $\ker^1(F, S)$ (resp. $\ker^1(\A, S)$), as defined in \cref{rational-adelic-pts}. The following lemma is a refinement of \cref{fibers-of-alpha-are-finite} in the setting of Galois symmetric pairs of rank one tori.

\begin{lemma}\label{support-lemma}
    Let $\xi \in \ker^1(\A, S)$ be a relevant cocycle. Then there exists a rational point $x_{\xi} \in S(F)$ in the $T(\A)$-orbit corresponding to $\xi$. Moreover, any other rational point meeting $T(\A) \star x_{\xi}$ lies in the same $T(F)$-orbit as $x_{\xi}$, or in other words $$(T(\A) \star x_{\xi}) \cap S(F) = T(F) \star x_{\xi}.$$
\end{lemma}
\begin{proof}
    The existence of a rational point is an immediate consequence of relevance. To verify the second claim, it suffices to check the claim when $x_{\xi} = 1$. In this case, the assertion boils down to understanding the fiber of the natural map $\ker^1(F, S) \to \ker^1(\A, S)$ above the trivial cocycle. We claim that this fiber is trivial. Indeed, if $\delta \in \ker^1(F, S)$ maps to the trivial cocycle in $\ker^1(\A, S)$ then $\delta \in \Sha(T')$, where $\Sha(T') = \ker[H^1(F, T') \to H^1(\A, T')]$. Since (at worst) $T'$ splits over a quadratic extension of $F$, by the Hasse norm theorem we conclude that $\Sha(T')$ is trivial \cite[Theorem 6.1.1, (ii)]{ono}, and we win.
\end{proof}

\subsection{The Relative Trace Formula}

For the remainder of this section, fix $(T, T')$ a Galois symmetric pair associated to a generalized biquadratic extension $M/F$ and a quadratic étale $F$-subalgebra $L/F$. Let $S$ be the associated symmetric space.

\subsubsection{Setup}

Let $A_T$ be the maximal $\theta$-stable $F$-split torus in $T$, and let $A_{T'} = A_T^{\theta}$ be the maximal split $F$-torus in $T'$. Write $\a_0 \coloneqq X_*(A_T) \otimes_{\Z} \R$ and $\a_{0'} \coloneqq X_*(A_{T'}) \otimes_{\Z} \R$. There is a Harish-Chandra map $H_T : T(\A) \to \a_0$ defined as the unique map such that $$\langle \xi, H_T(t) \rangle = \log |\xi(t)|_{\A},\;\;\xi \in X^*(T)$$ where $t \in T(\A)$. Define $H_{T'} : T'(\A) \to \a_{0'}$ analogously. Let $A_T^{\infty}$ (resp. $A_{T'}^{\infty}$) be the connected component of the $\R$-points of the maximal $\Q$-split torus in $\Res_{F/\Q} A_T$ (resp. $\Res_{F/\Q} A_{T'}$).

Define $T(\A)^1 \coloneqq \{ t \in T(\A) : H_T(t) = 0 \}$. Then we have $T(\A) = A_T^{\infty}T(\A)^1$. Set $[T]^1 \coloneqq T(F) \backslash T(\A)^1$. Similarly, define $T'(\A_F)^{1, T} \coloneqq \{ t \in T'(\A_F) : H_{T'}(t)_{T} = 0 \}$. Set $Z_T^{\infty} \coloneqq A_T^{\infty} \cap A_{T'}^{\infty}$. Then we have $T'(\A) = Z_T^{\infty}T'(\A)^{1, T}$. Finally, set $[T']^{1, T} \coloneqq T(F) \backslash T(\A_F)^{1, T'}$.

\begin{remark}
    If $T$ is anisotropic, $T(\A)^1 = T(\A)$.
\end{remark}

For each relevant $\xi \in \ker^1(\A, S)$ choose a basepoint $x_{\xi} \in S(F)$ in the $\theta$-twisted $T(\A)$-orbit corresponding to $\xi$. Note that since our group is abelian, the stabilizer of $x_{\xi}$ is the same as the stabilizer of $1 \in S(F) \subset S(\A)$, which is $T'(\A)$.

\subsubsection{Test Functions}

Let $f \in \Cc(S(\A))$ be a factorizable test function. Define $$f^{\xi} \coloneqq f|_{T(\A) \star x_{\xi}}.$$ Note that since $T(\A) \star x_{\xi}$ is open in $S(\A)$, and $f$ is compactly supported, for all but finitely many $\xi$, $f^{\xi}$ is zero. Since $f$ is factorizable, $f^{\xi}$ is again factorizable, and by identifying $T(\A) \star x_{\xi}$ with $T(\A)/T'(\A)$ we see that there exists $\Phi^{\xi} \in \Cc(T(\A))$ such that for all $\eta \in T(\A) \star x_{\xi}$, $$f^{\xi}(\eta) = \int_{T'(\A)} \Phi^{\xi}(st)\,dt$$ where $s \star x_{\xi} = sx_{\xi}\theta(s)^{-1} = \eta$. Using the factorization $T'(\A) = Z_T^{\infty}T'(\A)^{1, T}$, we can rewrite this as $$f^{\xi}(\eta) = \int_{T'(\A)^{1, T}} \Phi_1^{\xi}(st)\,dt$$ where $\Phi_1^{\xi}(x) = \int_{Z^{\infty}_T} \Phi^{\xi}(zx)\,dz$. Finally, we restrict each $\Phi_1^{\xi}$ to $T(\A)^1 \subset T(\A)$ to obtain a family of functions which are compactly supported and invariant under $Z_T^{\infty}$.

% Cpt supp: Let C = supp \Phi, which is compact. As a function on G(\A), supp \Phi_1 is contained in Z_T^{\infty} \cdot C. Thus, the restriction to G(\A)^1 has support contained in [Z_T^{\infty} \cdot C] \cap G(\A)^1. Let zc be any such element where z \in Z_T^{\infty} and c \in C, and write c = c^1c', where c' \in G(A)^1. Then zc = (zc^1)c', and this decomposition is unique. Thus it lies in G(A)^1 if and only if zc^1 is trivial, 

%Define $\Phi_f \coloneqq \sum_{\xi \in \im \alpha} \Phi_1^{\xi}.$

%TODO: It might be worth writing out an example here; for example, when $F = \Q$, $E = \Q(\sqrt{2})$ and $M = \Q(\sqrt{2}, \sqrt{3})$ vs $M = \Q(\sqrt{2}, \sqrt{-3})$. In the first case, $T = \Nm^1_{M/E}$ at $\infty$ looks like $\R^{\x} \x \R^{\x}$ and in the second case it looks like $S^1 \x S^1$. Let $L = \Q(\sqrt{3})$ (resp $\Q(\sqrt{-3})$). Write out $A_T$, $A_{T'}$ in these cases.

%Technical point: It seems to me that the correct thing to do is choose a function $\Phi \in \Cc(T(\A))$ and then average it over $A_T^{\infty}$ to obtain a function on $T(\A)^1$ (the measures are compatible in the obvious way). Henceforth when I write, $\Cc(T(\A)^1)$, I really mean the space of functions obtained in this way.

\subsubsection{Kernels}

Now for each $\Phi^{\xi}_1 \in \Cc(T(\A)^1)$, consider the corresponding right regular action of $\Phi^{\xi}_1$ on $L^2([T]^1)$, given by the kernel $$k_{\xi}(x, y) \coloneqq \sum_{\gamma \in T(F)} \Phi^{\xi}_1(x^{-1}\gamma y).$$

% TODO: Missing discussion of Tamagawa measures for tori! This is important; see Ono's paper; the Tamagawa number of $T$ is the volume of $[T]^1$ with respect to the (measure induced from the) Tamagawa measure.

\begin{lemma}
    The torus $[T]^1$ is compact.
\end{lemma}
\begin{proof}
    This is well-known.
\end{proof}

Thus, $L^2([T]^1)$ decomposes as a Hilbert space direct sum of unitary characters, each occurring with multiplicity one. For a fixed unitary character $\chi : [T]^1 \to \C^{\x}$, the right regular action of $\Phi^{\xi}_1$ is given by $$(R(\Phi^{\xi}_1)\chi)(x) = \int_{T(\A)^1} \Phi^{\xi}_1(y)\chi(xy)\,dy = \chi(x) \int_{T(\A)^1} \Phi^{\xi}_1(y)\chi(y)\,dy.$$ Define the Fourier transform of $\Phi \in \Cc(T(\A)^1)$ to be the function $\widecheck{\Phi} : \widehat{[T]^1} \to \C^{\x}$ given by $$\widecheck{\Phi}(\chi) \coloneqq (R(\Phi)\chi)(1).$$

% By abuse of notation, we write $$S(F) = \bigsqcup_{\xi \in \ker[H^1(F, T') \to H^1(F, T)]} T(F) \star \xi.$$

% We now explain how to transfer functions from $\Cc(T(\A)^1)$ to $\Cc(S(\A))$. Recall that we can decompose $S(\A)$ as a union of $\theta$-twisted $T(\A)$-orbits, indexed by $\ker[H^1(\A, T') \to H^1(\A, T)]$, $$S(\A) = \bigsqcup_{\eta \in \ker[H^1(\A, T') \to H^1(\A, T)]} T(\A) \star \eta.$$ Given $\Phi \in \Cc(T(\A)^1)$, define $f_{\Phi} \in \Cc(S(\A))$ by $$f_{\Phi}(\eta) \coloneqq \int_{T'(\A)^{1, T}} \Phi(s t)\,dt$$ for any $s \in T(\A)$ such that $s \star 1 = s\theta(s)^{-1} = \eta$, and otherwise, define $f_{\Phi}(\eta)$ to be zero. Thus, $f_{\Phi}$ is supported in $T(\A) \star 1 \subset S(\A)$.

Finally, for $f \in \Cc(S(\A))$, define $$k_{f^{\xi}}(x) \coloneqq \sum_{\eta \in S(F)} f^{\xi}(x^{-1}\eta x)$$ for all $x \in T'(\A)^{1, T}$ and define $$k_f(x) \coloneqq \sum_{\xi \in \im \alpha} k_{f^{\xi}}(x).$$ Note that there are only finitely many relevant $\xi$ which contribute non-trivially to this sum.

Of course, since our group is commutative, the kernels $k_{f^{\xi}}(x)$ are independent of $x$. Nevertheless, we are writing it in this way to demonstrate that all of our definitions work just as well in this degenerate case.

\subsubsection{The Basic Comparison}

We now come to the key comparison on which the relative trace formula for tori rests.

\begin{proposition}\label{main-comparison-for-tori}
    Let $f \in \Cc(S(\A))$ be a test function as above.
    For all $x \in T'(\A)^{1, T}$, $$\sum_{\xi \in \im \alpha} \int_{[T']^{1, T}} k_{\xi}(x, y)\,dy = k_f(x).$$
\end{proposition}
\begin{proof}
    It suffices to prove the claim for a single relevant $\xi$. Fix such, and let $x_{\xi} \in S(F)$ be the corresponding adelic basepoint. We calculate
    \begin{align*}
        k_{f^{\xi}}(x) &= \sum_{\eta \in S(F)} f^{\xi}(x^{-1}\eta x)\\
        &= \sum_{\mu \in \ker^1(F, S)} \sum_{\eta \in T(F) \star \mu} f^{\xi}(x^{-1}\eta x)
    \end{align*}
    By \cref{support-lemma}, $(T(\A) \star x_{\xi}) \cap S(F) = T(F) \star x_{\xi}$, and thus the only term which contributes to the sum is $\mu = \xi$. Therefore,
    \begin{align*}
        k_{f^{\xi}}(x) &= \sum_{\eta \in T(F) \star x_{\xi}} f^{\xi}(x^{-1}\eta x)\\
        &= \sum_{\gamma \in T(F)/T'(F)} \int_{T'(\A)^{1, T}} \Phi^{\xi}_1(x^{-1}\gamma y)\,dy\\
        &= \sum_{\gamma \in T(F)} \int_{[T']^{1, T}} \Phi^{\xi}_1(x^{-1}\gamma y)\,dy \\
        &= \int_{[T']^{1, T}} \sum_{\gamma \in T(F)} \Phi^{\xi}_1(x^{-1}\gamma y)\,dy \\
        &= \int_{[T']^{1, T}} k_{\xi}(x, y)\,dy,
    \end{align*}
    and we win.
\end{proof}

\quash{
\begin{remark}
    For test functions with supports contained in several orbits, one would have to modify the above proposition; instead of a single kernel $k_{\Phi}$, summing over $T(F)$, there would be multiple such.
\end{remark}}

\subsubsection{The Coarse RTF}

In this setting, the collection of cuspidal data is the collection of all distinct unitary characters $\chi : [T]^1 \to \C^{\x}$. Since there is no stable conjugacy, the collection of geometric data is identified with the set of rational points $S(F)$.

For a character $\chi \in \widehat{[T]^1}$, define $k_{\Phi, \chi}$ by projecting $k_{\Phi}$ onto the $\chi$-isotypic component of $L^2([T]^1)$. Define the functional $$J_{\chi}(\Phi) \coloneqq \int_{[T']^{1, T}}\int_{[T']^{1, T}} k_{\Phi, \chi}(x, y)\,dy\,dx,$$ for $\Phi \in \Cc(T(\A)^1)$. For a point $\eta \in S(F)$, define the functional $$J_{\eta}(f) \coloneqq \vol([T']^{1, T})f(\eta),$$ where $f \in \Cc(S(\A))$.

% Remark: We could be slightly more general here, and integrate a second time over a torus that is not $T'$, for example, we could integrate over the subgroup associated to the other $F$-structure of $T$. (It seems like the symmetric subgroups of $T$ generate $T$ up to almost direct product. Doesn't this imply that a representation of $[T]^1$ that is both $T'$ and $T''$-distinguished is essentially some Dirichlet type character? I.e if $\chi|_{T'} = 1$ and $\chi|_{T''} = 1$ then $\chi$ is trivial on the connected component of $T(\A)^1$.

% To organize our results, define $\Phi_f \in \bigoplus_{\alpha \in \rel^1(T', T)} \Cc([T]^1)$, and consider the right regular action of $\Phi_f$ on 

% To organize our results, consider $L^2([T^1])_{\rel} \coloneqq \bigoplus_{\alpha \in \rel^1(T', T)} L^2([T^1])$

\begin{theorem}[Coarse RTF for Tori]
    For $f \in \Cc(S(\A))$ and $\{ \Phi^{\xi}_1 \}_{\xi \in \im \alpha}$ the corresponding collection of test functions supported on relevant adelic orbits, $$\sum_{\xi \in \im \alpha} \int_{[T']^{1, T}}\int_{[T']^{1, T}} \sum_{\chi} |k_{\Phi^{\xi}_1, \chi}(x, y)|\,dy\,dx < \infty,$$ and
    $$\sum_{\xi \in \im \alpha} \int_{[T']^{1, T}} \sum_{\eta \in S(F)} |k_{f^{\xi}}(x)| \,dx < \infty,$$ where the sum over the relevant $\xi$ is finite.
    Consequently we have the following equality of distributions $$\sum_{\xi \in \im \alpha} \sum_{\chi} J_{\chi}(\Phi^{\xi}_1) = \sum_{\eta \in S(F)} J_\eta(f).$$
\end{theorem}
\begin{proof}
    This is an immediate consequence of \cref{main-comparison-for-tori}, as $[T']^1$ is compact and $[T']^{1, T}$ is a closed subgroup. The equalities follow immediately from the definitions.
\end{proof}

\quash{
\begin{theorem}[Coarse RTF]
    The RTF for the Galois symmetric pair of tori $(T, T')$ reads: for all $\Phi \in \Cc(T(\A)^1)$ $$\int_{[T']^{1, T}}\int_{[T']^{1, T}} \sum_{\chi} |k_{\Phi, \chi}(x, y)| < \infty,$$
    and for all $f \in \Cc(S(\A))$
    $$\int_{[T']^{1, T}} \sum_{\eta \in S(F)} |k_f(x)| \,dx < \infty.$$
    Finally, $$\sum_{\chi} J_{\chi}(\Phi) = \sum_{\eta \in S(F)} J_\eta(f_\Phi).$$
\end{theorem}
\begin{proof}
    This is trivial; $[T]^1$ is compact and $[T']^{1, T}$ is a closed subgroup. The equalities follow immediately from the definitions.
\end{proof}}

We call the above equality of distributions the coarse RTF for $S$. We now compute all of its terms explicitly.

First, note that $k_{\Phi, \chi}(x, y) = (R(\Phi)\chi)(x)\overline{\chi(y)}$, because the $\chi$-isotypic summand of $L^2([T]^1)$ is generated by $\chi$. Thus, $$J_{\chi}(\Phi) = \int_{[T']^{1, T}} (R(\Phi)\chi)(x) \,dx \int_{[T']^{1, T}} \overline{\chi(y)}\,dy.$$ This final expression is equal to zero if $\chi|_{[T']^{1, T}}$ is non-trivial, otherwise it is equal to $\vol([T']^{1, T})(R(\Phi)\chi)(1) = \vol([T']^{1, T})\widecheck{\Phi}(\chi)$. In summary, $$J_{\chi}(\Phi) = \begin{cases}
    \vol([T']^{1, T})\widecheck{\Phi}(\chi) & \text{ if $\chi|_{[T']^{1, T}} = 1$} \\
    0 & \text{ otherwise. }
\end{cases}$$

\quash{Finally, since $f_{\Phi}$ is only supported on $T(\A) \star 1$, by \cref{support-lemma}, we have that $$J_\eta(f_{\Phi}) = \begin{cases}
    \vol([T']^{1, T})\int_{T'(\A)^{1, T}}\Phi(st)\,dt & \text{ if $\eta = s\theta(s)^{-1} \in T(F) \star 1$} \\
    0 & \text{ otherwise. }
\end{cases}$$}

Thus, we see that the content of the formal equality above is simply $$\sum_{\eta \in S(F)} f(\eta) = \sum_{\chi |_{[T']^{1, T}} = 1} \sum_{\xi \in \im \alpha} \widecheck{\Phi}^{\xi}_1(\chi).$$ This is an incarnation of the Poisson summation formula.

%Finally, since there is no non-trivial stable conjugacy, the relative trace formula for $(T, T')$ is automatically stable.

\ifSubfilesClassLoaded{%
\bibliography{main.bib}%
}{}

\end{document}